%% file: smms_sigmak.tex
\documentclass{amsart}
\input{header}

\begin{document}

\title{The weighted $\sigma_k$-curvature of a smooth metric measure space}
\author{Jeffrey S. Case}
\address{109 McAllister Building \\ Penn State University \\ University Park, PA 16802}
\email{jscase@psu.edu}
\keywords{smooth metric measure space; $\sigma_k$-curvature; quasi-Einstein; weighted Einstein}
\subjclass[2010]{Primary 53C21; Secondary 35J60, 58E11}
\begin{abstract}
 We propose a definition of the weighted $\sigma_k$-curvature of a smooth metric measure space and justify it in two ways.  First, we show that the weighted $\sigma_k$-curvature prescription problem is governed by a fully nonlinear second order elliptic PDE which is variational when $k=1,2$ or the smooth metric measure space is locally conformally flat in the weighted sense.  Second, we show that, in the variational cases, quasi-Einstein metrics are stable with respect to the total weighted $\sigma_k$-curvature functional.  We also discuss related conjectures for weighted Einstein manifolds.
\end{abstract}
\maketitle

\input{intro}

\subsection*{Acknowledgments}
This work was begun during a month-long visit to the Centre de Recerca Matem\'atica.  The author would like to thank the CRM for providing a productive research environment.  It was also partially supported by NSF Grant DMS-1004394.

\input{algebra}
\input{smms}
\input{moduli}
\input{variational}
\input{stable}
\input{obata}
\input{full}

\bibliographystyle{abbrv}
\bibliography{bib}
\end{document}

%% file: header.tex
\usepackage{amsmath}
\usepackage{amssymb}
\usepackage{amsthm}

\DeclareMathOperator{\Id}{Id}

\DeclareMathOperator{\tr}{tr}

\DeclareMathOperator{\dvol}{dvol}

\DeclareMathOperator{\Ric}{Ric}

\DeclareMathOperator{\Rm}{Rm}

\DeclareMathOperator{\End}{End}

\DeclareMathOperator{\hash}{\sharp}

\DeclareMathOperator{\Met}{Met}

\newcommand{\Diff}{\mathrm{Diff}}

\newcommand{\opsi}{\overline{\psi}}

\newcommand{\cs}{\widetilde{s}}

\newcommand{\cB}{\widetilde{B}}

\newcommand{\cE}{\widetilde{E}}
\newcommand{\cN}{\widetilde{N}}

\newcommand{\cQ}{\widetilde{Q}}

\newcommand{\cS}{\widetilde{S}}
\newcommand{\cT}{\widetilde{T}}
\newcommand{\cU}{\widetilde{U}}
\newcommand{\cY}{\widetilde{Y}}
\newcommand{\cZ}{\widetilde{Z}}

\newcommand{\csigma}{\widetilde{\sigma}}

\newcommand{\cmF}{\widetilde{\mathcal{F}}}
\newcommand{\cmY}{\widetilde{\mathcal{Y}}}
\newcommand{\hg}{\widehat{g}}
\newcommand{\hv}{\widehat{v}}
\newcommand{\hA}{\widehat{A}}
\newcommand{\hE}{\widehat{E}}
\newcommand{\hJ}{\widehat{J}}

\newcommand{\hP}{\widehat{P}}

\newcommand{\hZ}{\widehat{Z}}

\newcommand{\hlambda}{\widehat{\lambda}}
\newcommand{\hkappa}{\widehat{\kappa}}
\newcommand{\hsigma}{\widehat{\sigma}}

\newcommand{\lp}{\langle}
\newcommand{\rp}{\rangle}
\newcommand{\lv}{\lvert}
\newcommand{\rv}{\rvert}

\newcommand{\mF}{\mathcal{F}}

\newcommand{\mM}{\mathcal{M}}

\newcommand{\mQ}{\mathcal{Q}}

\newcommand{\mS}{\mathcal{S}}

\newcommand{\mV}{\mathcal{V}}

\newcommand{\mY}{\mathcal{Y}}
\newcommand{\mZ}{\mathcal{Z}}

\newcommand{\kC}{\mathfrak{C}}

\newcommand{\kM}{\mathfrak{M}}

\newcommand{\bC}{\mathbb{C}}

\newcommand{\bN}{\mathbb{N}}

\newcommand{\bR}{\mathbb{R}}

\DeclareMathOperator{\tf}{tf}

\def\sideremark#1{\ifvmode\leavevmode\fi\vadjust{\vbox to0pt{\vss
 \hbox to 0pt{\hskip\hsize\hskip1em
 \vbox{\hsize3cm\tiny\raggedright\pretolerance10000
 \noindent #1\hfill}\hss}\vbox to8pt{\vfil}\vss}}}

\newcommand{\suchthat}{\mathrel{}\middle|\mathrel{}}

\newcommand{\comment}[1]{}

\newtheorem{thm}{Theorem}[section]
\newtheorem{prop}[thm]{Proposition}
\newtheorem{lem}[thm]{Lemma}
\newtheorem{cor}[thm]{Corollary}

\theoremstyle{definition}
\newtheorem{defn}[thm]{Definition}

\newtheorem{conj}[thm]{Conjecture}
\newtheorem{example}[thm]{Example}

\theoremstyle{remark}
\newtheorem{remark}[thm]{Remark}

\numberwithin{equation}{section}

%% file: intro.tex
\section{Introduction}
\label{sec:intro}

In Riemannian geometry, the $\sigma_k$-curvatures are scalar Riemannian invariants which have proven to be useful tools for studying geometric and analytic properties of Riemannian manifolds.  For example, locally conformally flat Einstein metrics with nonzero scalar curvature locally extremize the total $\sigma_k$-curvature functional within their conformal class for generic values of $k$; when $k\leq2$, the same is true for all Einstein metrics with nonzero scalar curvature~\cite{Viaclovsky2000}.  This greatly expands the set of Riemannian functionals which one can use to study Einstein metrics and leads to new variational characterizations of such manifolds; e.g.\ \cite{GuanViaclovskyWang2003,GurskyViaclovsky2001}.  Moreover, one can classify all critical points of the total $\sigma_k$-curvature functional in the positive $k$-cone of the conformal class of the round metrics on the sphere~\cite{ChangGurskyYang2003b,Gonzalez2006c,LiLi2003,Viaclovsky2000}, providing an important first step to proving sharp fully nonlinear Sobolev inequalities~\cite{GuanWang2004}.  Since the $\sigma_1$-curvature is a dimensional multiple of the scalar curvature, these facts naturally generalize well-known properties of the Yamabe functional; cf.\ \cite{LeeParker1987}.

Roughly speaking, smooth metric measure spaces are Riemannian manifolds equipped with a smooth measure.  The geometric study of smooth metric measure spaces is based in large part on the $m$-Bakry--\'Emery Ricci tensor.  This tensor generalizes the Ricci tensor and thereby leads to the notion of a ``gradient Einstein-type manifold'', obtained by requiring the $m$-Bakry--\'Emery Ricci tensor to be a multiple of the metric.  Such manifolds include as special cases gradient Ricci solitons and static metrics in general relativity, and this framework provides a useful uniform approach to their study; cf.\ \cite{CaoChen2011,CatinoMastroliaMonticelliRigoli2014,ChenHe2011,HuangWei2013}.  Moreover, many of these manifolds admit a characterization as critical points of a generalization of the Yamabe functional~\cite{Case2010b,Case2013y}.  In particular, the family of sharp Gagliardo--Nirenberg inequalities studied by Del Pino and Dolbeault~\cite{DelPinoDolbeault2002} can be understood via such functionals~\cite{Case2011gns,Case2013y}.  This family of sharp Gagliardo--Nirenberg inequalities has proven useful for studying certain fast diffusion equations~\cite{CarlenCarrilloLoss2010,DelPinoDolbeault2002}.

In this article we introduce the weighted $\sigma_k$-curvatures as appropriate generalizations of the $\sigma_k$-curvatures to smooth metric measure spaces.  These curvatures form a family of scalar invariants on smooth metric measure spaces with properties which suitably generalize those algebraic and variational properties of the $\sigma_k$-curvatures which are important to the study of Einstein metrics.  We expect that the weighted $\sigma_k$-curvatures will find further use in studying smooth metric measure spaces in general, and quasi-Einstein and weighted Einstein manifolds in particular.  A more precise explanation of these results requires some notation and terminology.

A \emph{smooth metric measure space} is a five-tuple $(M^n,g,v,m,\mu)$ consisting of a Riemannian manifold $(M^n,g)$, a positive function $v\in C^\infty(M;\bR_+)$, a dimensional parameter $m\in\bR$, and an auxiliary curvature parameter $\mu\in\bR_+$.  The metric $g$, the function $v$ and the parameter $m$ together determine the \emph{weighted volume element} $d\nu:=v^m\dvol_g$ on $M$.  Roughly speaking, the dimensional parameter $m$ indicates that we want to regard $(M^n,g,d\nu)$ as an $(m+n)$-dimensional metric measure space, in the sense that we consider the \emph{$m$-Bakry--\'Emery Ricci tensor}
\[ \Ric_\phi^m := \Ric + \nabla^2\phi - \frac{1}{m}d\phi\otimes d\phi \]
as the weighted analogue of the Ricci tensor, where $\phi=-m\ln v$.  The auxiliary curvature parameter $\mu$ indicates that we want to regard $(M^n,g,d\nu)$ as the base of the warped product
\begin{equation}
 \label{eqn:intro/wp}
 \left(M^n\times F^m(\mu),g\oplus v^2h\right),
\end{equation}
where $(F^m(\mu),h)$ is the $m$-dimensional simply-connected spaceform with constant sectional curvature $\mu$.  Most weighted invariants can be regarded as the restriction to $M$ of Riemannian invariants on the warped product~\eqref{eqn:intro/wp} when the latter makes sense.  For example, $d\nu$ is the restriction of the Riemannian volume element of~\eqref{eqn:intro/wp} and the \emph{weighted scalar curvature}
\[ R_\phi^m := R + 2\Delta\phi - \frac{m+1}{m}\lv\nabla\phi\rv^2 + m(m-1)\mu e^{\frac{2\phi}{m}} \]
is the scalar curvature of~\eqref{eqn:intro/wp}.  The weighted $\sigma_k$-curvatures are defined in terms of the \emph{weighted Schouten tensor}
\[ P_\phi^m := \Ric_\phi^m - \frac{1}{m+n-2}J_\phi^mg, \]
where
\[ J_\phi^m := \frac{m+n-2}{2(m+n-1)}R_\phi^m . \]
Note that all of the tensors just defined make sense in the limits $m=0$ and $m=\infty$.  The weighted $\sigma_k$-curvatures in the case $m=0$ are the Riemannian $\sigma_k$-curvatures.  The weighted $\sigma_k$-curvatures in the case $m=\infty$ have been previously considered by the author~\cite{Case2014sd}.  Since the results of this article are already known in the limiting cases $m=0$ and $m=\infty$, we shall restrict our attention here to the cases $m\in\bR_+$.

Informally, given nonnegative integers $k,n\in\bN$ and a dimensional parameter $m\in\bR_+$, the \emph{$m$-weighted $k$-th elementary symmetric polynomial of $n$-variables} is the $k$-th elementary symmetric polynomial of $(m+n)$-variables; i.e.
\begin{equation}
 \label{eqn:intro/informal_sigmak_defn}
 \sigma_k^m(\lambda;\lambda_1,\dotsc,\lambda_n) := \sigma_k\Bigl(\underbrace{\frac{\lambda}{m},\dotsc,\frac{\lambda}{m}}_{\text{$m$ times}},\lambda_1,\dotsc,\lambda_n\Bigr) .
\end{equation}
This can made precise by evaluating the right-hand side of~\eqref{eqn:intro/informal_sigmak_defn} when $m\in\bN_0$ and then extending the definition by treating $m\in\bR_+$ as a formal variable; cf.\ Section~\ref{sec:algebra}.  We extend this definition to symmetric matrices by considering the eigenvalues of the matrix; i.e.\ if $\lambda\in\bR$ and if $P$ is a symmetric $n$-by-$n$ matrix, then we define
\[ \sigma_k^m(\lambda;P) := \sigma_k^m(\lambda;\lambda_1,\dotsc,\lambda_n), \]
where $\lambda_1,\dotsc,\lambda_n$ are the eigenvalues of $P$.

Given $k\in\bN$ and a smooth metric measure space $(M^n,g,v,m,\mu)$, the \emph{weighted $\sigma_k$-curvature} is
\[ \sigma_{k,\phi}^m := \sigma_k^m\left(Y_\phi^m;P_\phi^m\right), \]
where $Y_\phi^m:=J_\phi^m-\tr_g P_\phi^m$.  The most relevant cases are $k=1,2$, where we directly compute that
\begin{align*}
 \sigma_{1,\phi}^m & = J_\phi^m, \\
 \sigma_{2,\phi}^m & = \frac{1}{2}\left(\left(J_\phi^m\right)^2 - \lv P_\phi^m\rv^2 - \frac{1}{m}\left(Y_\phi^m\right)^2\right) .
\end{align*}
Given $\kappa\in\bR$ and a smooth metric measure space $(M^n,g,v,m,\mu)$, the \emph{weighted $\sigma_k$-curvature with scale $\kappa$} is
\[ \csigma_{k,\phi}^m := \sigma_k^m\left(Y_\phi^m+m\kappa v^{-1};P_\phi^m\right) . \]
With the goal of keeping our notation simple, we do not explicitly incorporate $\kappa$ into our notation for the weighted $\sigma_k$-curvatures with scale $\kappa$.  Instead, we always use tildes to denote quantities defined in terms of a possibly nonzero scale $\kappa$, and omit the tilde when we fix $\kappa=0$; i.e.\ $\sigma_{k,\phi}^m$ is the weighted $\sigma_k$-curvature with scale $\kappa=0$.  The role of the scale $\kappa$, especially when the auxiliary curvature parameter $\mu$ vanishes, is to specify the ``size'' of the function $v$.  Note that in the special cases $k=1,2$, it holds that
\begin{align*}
 \csigma_{1,\phi}^m & = \sigma_{1,\phi}^m + m\kappa v^{-1}, \\
 \csigma_{2,\phi}^m & = \sigma_{2,\phi}^m + \frac{m(m+n-2)}{m+n-3}\kappa v^{-1}\sigma_{1,\phi}^{m-1} + \frac{m(m-1)}{2}\kappa^2v^{-2},
\end{align*}
where $\sigma_{1,\phi}^{m-1}$ is the weighted $\sigma_1$-curvature of $(M^n,g,v,m-1,\mu)$; see Lemma~\ref{lem:csigma_to_sigma}.

Our study of the weighted $\sigma_k$-curvatures is focused on their variational properties within a weighted conformal class.  A \emph{weighted manifold} is a triple $(M^n,m,\mu)$ of a smooth manifold $M^n$, a dimensional parameter $m\in\bR_+$, and an auxiliary curvature parameter $\mu\in\bR$.  A \emph{weighted conformal class} $\kC$ on a weighted manifold $(M^n,m,\mu)$ is an equivalence class with respect to the equivalence relation
\[ \text{$(g,v)\sim(\hg,\hv)$ if and only if $(\hg,\hv)=(u^{-2}g,u^{-1}v)$ for some $u\in C^\infty(M;\bR_+)$.} \]
The weighted conformal class determined by $(M^n,g,v,m,\mu)$ is equivalent to the subset of the conformal class of the formal warped product~\eqref{eqn:intro/wp} determined by restricting attention to conformal factors which depend only on the base $M$.  We say that $\kC$ is \emph{locally conformally flat in the weighted sense} if the formal warped product~\eqref{eqn:intro/wp} is locally conformally flat; see Section~\ref{sec:smms} for an intrinsic characterization of this condition.  We say that the weighted $\sigma_k$-curvature $\csigma_{k,\phi}^m$ is \emph{variational in $\kC$} if there is a functional $\mS\colon\kC\to\bR$ such that
\[ \left.\frac{d}{dt}\right|_{t=0}\mS\left(e^{-\frac{2t\psi}{m+n}}g,e^{-\frac{t\psi}{m+n}}v\right) = \int_M \csigma_{k,\phi}^m\psi\,d\nu \]
for all $\psi\in C^\infty(M)$ and all $(g,v)\in\kC$.  Our first main result is a characterization of when the weighted $\sigma_k$-curvatures are variational (cf.\ \cite{BransonGover2008}).

\begin{thm}
 \label{thm:intro/variational}
 Fix $k\in\bN$ and $\kappa\in\bR$.  Let $\kC$ be a weighted conformal class on a weighted manifold $(M^n,m,\mu)$.  Then $\csigma_{k,\phi}^m$ is conformally variational in $\kC$ if and only if $k\leq 2$ or $\kC$ is locally conformally flat in the weighted sense.
\end{thm}

We prove Theorem~\ref{thm:intro/variational} by showing that, as a function of $\kC$, the linearization of $\csigma_{k,\phi}^m$ is formally self-adjoint for every representative of $\kC$; for details, see Section~\ref{sec:moduli} and Section~\ref{sec:variational}.

Our other main results concern the variational properties of the weighted $\sigma_k$-curvatures on weighted Einstein manifolds in the cases when $\csigma_{k,\phi}^m$ is variational.  A \emph{weighted Einstein manifold} is a smooth metric measure space $(M^n,g,v,m,\mu)$ for which there is a constant $\lambda\in\bR$ such that $P_\phi^m=\lambda g$.  For such manifolds, there is a scale $\kappa\in\bR$ such that $\csigma_{1,\phi}^m=(m+n)\lambda$; see~\cite[Lemma~9.1]{Case2013y} or Lemma~\ref{lem:we_kappa}.  We highlight here two special classes of weighted Einstein manifolds.  First, weighted Einstein manifolds with $\kappa=0$ are equivalent to quasi-Einstein manifolds.  Second, weighted Einstein manifolds for which the auxiliary curvature parameter $\mu$ vanishes are precisely the critical points of the weighted Yamabe functional~\cite{Case2013y} through variations of the metric and the measure; cf.\ Theorem~\ref{thm:full_Y1}.  Much more is known about quasi-Einstein manifolds than weighted Einstein manifolds, and for this reason we can prove more in the former setting.

Weighted Einstein manifolds can be understood in terms of the total weighted $\sigma_k$-curvature functionals.  Our best such results are for quasi-Einstein manifolds, and are most naturally stated in terms of the set
\[ \kC_1 := \left\{ (g,v)\in\kC \suchthat \int_M d\nu = 1 \right\} \]
of representatives of $\kC$ with respect to which $M$ has unit weighted volume.

\begin{thm}
 \label{thm:intro/stable_qe}
 Fix $k\in\bN$.  Let $\kC$ be a weighted conformal class on a closed weighted manifold $(M^n,m,\mu)$; if $k\geq3$, assume additionally that $\kC$ is locally conformally flat in the weighted sense.  Define the $\mF_k$-functional $\mF_k\colon\kC_1\to\bR$ by
 \[ \mF_k(g,v) := \int_M \sigma_{k,\phi}^m\,d\nu . \]
 Suppose that $(g,v)\in\kC_1$ is such that $P_\phi^m=\lambda g$ and $\sigma_{1,\phi}^m=(m+n)\lambda$ for some constant $\lambda>0$.  Then $(g,v)$ is a critical point of $\mF_k$, and moreover
 \begin{enumerate}
  \item if $k<\frac{m+n}{2}$, then
  \[ \left.\frac{d^2}{dt^2}\right|_{t=0}\mF_k\left(\gamma(t)\right) > 0 \]
  for all $\gamma\colon\bR\to\kC_1$ such that $\gamma(0)=(g,v)$ and $\gamma^\prime(0)\not=0$;
  \item if $\frac{m+n}{2}<k\leq m+n$, then
  \[ \left.\frac{d^2}{dt^2}\right|_{t=0}\mF_k\left(\gamma(t)\right) < 0 \]
  for all $\gamma\colon\bR\to\kC_1$ such that $\gamma(0)=(g,v)$ and $\gamma^\prime(0)\not=0$.
 \end{enumerate}
\end{thm}

In the case $k=\frac{m+n}{2}$, the functional $\mF_k$ is constant.  In the cases $k>m+n$, the functional $\mF_k$ is constant (and identically zero) if and only if $m\in\bN$.  When $k>m+n$ and $m$ is not an integer, the sign of the second variation of $\mF_k$ depends on the parity of the integer part of $m+n-k$; cf.\ \eqref{eqn:stable_qe_second_var}.

The proof of Theorem~\ref{thm:intro/stable_qe} depends on two ingredients.  First, one computes the first and second variations of the $\mF_k$-functional.  In particular, when $k\leq2$ or $\kC$ is locally conformally flat in the weighted sense, $(g,v)\in\kC_1$ is a critical point of the $\mF_k$-functional if and only if
\begin{equation}
 \label{eqn:intro/mFeuler}
 \sigma_{k,\phi}^m = c
\end{equation}
for some constant $c\in\bR$.  Second, one uses the weighted Lichnerowicz--Obata Theorem~\cite{BakryQian2000} to conclude that under the assumptions of Theorem~\ref{thm:intro/stable_qe}, the first eigenvalue $\lambda_1$ of the \emph{weighted Laplacian} $-\Delta_\phi:=-\Delta+\nabla\phi$ satisfies $\lambda_1>\frac{2(m+n)}{m+n-2}\lambda$.  Applying this to the second variation of the $\mF_k$-functional yields the result.

For weighted Einstein manifolds with $\mu=0$ and positive scale, one should instead consider the $\mY_k$-functional $\mY_k\colon\kC\times\bR_+\to\bR$ defined by
\begin{align*}
 \mY_k(g,v,\kappa) & := \mZ_k(g,v,\kappa)\left(\int v^{-1}\,d\nu\right)^{-\frac{2mk}{(m+n)(2m+n-2)}}\left(\int d\nu\right)^{-\frac{m+n-2k}{m+n}} , \\
 \mZ_k(g,v,\kappa) & := \kappa^{-\frac{2mk(m+n-1)}{(m+n)(2m+n-2)}}\int_M \csigma_{k,\phi}^m\,d\nu .
\end{align*}
If $(M^n,g,v,m,0)$ is a weighted Einstein manifold with scale $\kappa>0$, then $(g,v)$ is a critical point of the $\mY_k$-functional whenever $\csigma_{k,\phi}^m$ is variational.  Indeed, when $\csigma_{k,\phi}^m$ is variational, a triple $(g,v,\kappa)\in\kC\times\bR_+$ is a critical point of the $\mY_k$-functional if and only if
\begin{subequations}
 \label{eqn:intro/mYeuler}
 \begin{multline}
  \label{eqn:intro/mYeulerpw}
  \csigma_{k,\phi}^m + \frac{m}{m+n-2k}\kappa v^{-1}\cs_{k-1,\phi}^m \\ = \frac{\int\csigma_{k,\phi}^m\,d\nu}{\int d\nu} + \frac{m}{m+n-2k}\left(\frac{\int \cs_{k-1,\phi}^mv^{-1}\,d\nu}{\int v^{-1}\,d\nu}\right)\kappa v^{-1}
 \end{multline}
 and
 \begin{equation}
  \label{eqn:intro/mYeulerintegral}
  \int_M \csigma_{k,\phi}^m\,d\nu = \frac{(m+n)(2m+n-2)}{2k(m+n-1)}\int_M \kappa v^{-1}\cs_{k-1,\phi}^m\,d\nu ,
 \end{equation}
\end{subequations}
where $\cs_{k-1,\phi}^m$ is defined by Definition~\ref{defn:newton_endomorphism} and~\eqref{eqn:defn_cs} below.  Note that $\csigma_{k,\phi}^m$ and $\cs_{k-1,\phi}^m$ are both constant for weighted Einstein metrics.  However, such weighted Einstein manifolds need not have positive $m$-Bakry--\'Emery Ricci curvature (cf.\ Example~\ref{ex:model_positive}).  In particular, a new weighted Lichnerowicz--Obata theorem seems to be needed in order to prove that weighted Einstein manifolds with $\mu=0$ and positive scale are local extrema of the $\mY_k$-functionals whenever $\csigma_{k,\phi}^m$ is variational; for further discussion, including a conjectural form of the required weighted Lichnerowicz--Obata theorem, see Subsection~\ref{subsec:stable/we}.

In light of Theorem~\ref{thm:intro/stable_qe}, one might hope that, up to scaling, quasi-Einstein manifolds are the only critical points of the $\mF_k$--functionals which lie in the weighted elliptic $k$-cones.  For quasi-Einstein manifolds which are locally conformally flat in the weighted sense, this is true.

\begin{thm}
 \label{thm:intro/obata}
 Let $(M^n,g,v,m,\mu)$ be a closed smooth metric measure space which satisfies $P_\phi^m=\lambda g$ and $\sigma_{1,\phi}^m=(m+n)\lambda$ for some $\lambda\in\bR$ and which is locally conformally flat in the weighted sense.  Fix $k\in\bN$, let $\kC$ be the weighted conformal class containing $(g,v)$, and suppose that $(\hg,\hv)\in\kC$ is a critical point of the $\mF_k$-functional such that $\sigma_{j,\phi}^m(\hg,\hv)>0$ for all $1\leq j\leq k$.  Then $(\hg,\hv)=(c^2g,cv)$ for some constant $c\in\bR_+$.
\end{thm}

The proof of Theorem~\ref{thm:intro/obata} is analogous to Obata's proof of the classification of conformally Einstein metrics with constant scalar curvature on closed manifolds~\cite{Obata1971} and its generalization by Viaclovsky~\cite{Viaclovsky2000} to the $\sigma_k$-curvatures; for details, see Subsection~\ref{subsec:obata/qe}.  The condition $\sigma_{j,\phi}^m(\hg,\hv)>0$, $1\leq j\leq k$, implies that the PDE~\eqref{eqn:intro/mFeuler} is elliptic at $(\hg,\hv)$.  By developing some additional integral estimates, we expect Theorem~\ref{thm:intro/obata} to also hold for on the round hemisphere $(S_+^n,d\theta^2,1,m,1)$; cf.\ \cite{ChangGurskyYang2003b,Gonzalez2006c}.  Indeed, we expect further study of the PDE $\sigma_{k,\phi}^m=f$ to lead to a sharp fully nonlinear Sobolev inequality in this setting; see Conjecture~\ref{conj:sobolev/qe}.

Many of the ideas in the proof of Theorem~\ref{thm:intro/obata} can be used to study the analogous classification for weighted Einstein manifolds.  More precisely, the proof of Obata's theorem~\cite{Obata1971} begins by using the variational structure of the $\sigma_k$-curvatures to find a $(0,2)$-tensor which is divergence-free for any metric $g$ with $\sigma_k$-constant and by using the assumption that $g$ is conformally Einstein to show that the trace-free part of the Schouten tensor is in the image of the adjoint of the divergence operator on trace-free $(0,2)$-tensors.  The proof ends by using the Newton inequalities and the specific form of the divergence-free $(0,2)$-tensor to conclude that $g$ is Einstein.  The first step carries through for solutions of~\eqref{eqn:intro/mFeuler} (resp.\ \eqref{eqn:intro/mYeuler}) which are conformally quasi-Einstein (resp.\ conformally weighted Einstein manifolds with $\mu=0$), though with a more complicated stand-in for the divergence operator and its adjoint; see Section~\ref{sec:obata}.  At present, while Theorem~\ref{thm:weighted_newton} asserts the weighted Newton inequalities, we can only carry out the second step in the setting of quasi-Einstein manifolds.  Indeed, carrying out the second step for weighted Einstein manifolds is even problematic in the case $k=1$ (cf.\ \cite[Conjecture~1.5]{Case2013y}).  Nevertheless, we expect that the analogue of Theorem~\ref{thm:intro/obata} for weighted Einstein manifolds is true.  Inspired by the sharp Gagliardo--Nirenberg inequalities of Del Pino and Dolbeault~\cite{DelPinoDolbeault2002}, we expect further study of the $\mY_k$-functionals to lead to sharp fully nonlinear Gagliardo--Nirenberg inequalities; see Conjecture~\ref{conj:sobolev/we}.

The fact that quasi-Einstein manifolds and weighted Einstein manifolds are critical points of the $\mF_k$- and $\mY_k$-functionals, respectively, within a weighted conformal class indicates that these functionals should be useful in studying such manifolds in a general variational context.  Our final main result verifies this expectation, at least in the cases $k=1,2$.  To be more precise, let
\[ \kM(M^n,m,\mu) := \left\{ (g,v) \suchthat g\in\Met(M), v\in C^\infty(M), v>0 \right\} \]
denote the space of \emph{metric-measure structures} on $(M^n,m,\mu)$ and denote
\[ \kM_1(M^n,m,\mu) := \left\{ (g,v)\in\kM \suchthat \int_M v^m\,\dvol = 1 \right\} . \]
It is clear that we can extend the $\mF_k$- and $\mY_k$-functionals to functionals on $\kM_1$ and $\kM\times\bR_+$, respectively.  Weighted Einstein manifolds are related to the critical points of these functionals in the following way:

\begin{thm}
 \label{thm:intro/full}
 Let $(M^n,m,\mu)$ be a weighted manifold.
 \begin{enumerate}
  \item $(g,v)\in\kM_1$ is a critical point of $\mF_1\colon\kM_1\to\bR$ if and only if $P_\phi^m=\lambda g$ and $\sigma_{1,\phi}^m=(m+n)\lambda$ for some $\lambda\in\bR$.
  \item If $(g,v)\in\kM_1$ satisfies $P_\phi^m=\lambda g$ and $\sigma_{1,\phi}^m=(m+n)\lambda$ for some $\lambda\in\bR$, then $(g,v)$ is a critical point of $\mF_2\colon\kM_1\to\bR$.
 \end{enumerate}
 Suppose additionally that $\mu=0$.
 \begin{enumerate}\addtocounter{enumi}{2}
  \item $(g,v,\kappa)\in\kM\times\bR_+$ is a critical point of $\mY_1\colon\kM\times\bR_+\to\bR$ if and only if $(g,v)$ is a weighted Einstein metric-measure structure with scale $\kappa$.
  \item If $(g,v)\in\kM$ is a weighted Einstein manifold with scale $\kappa>0$, then $(g,v,\kappa)$ is a critical point of $\mY_2\colon\kM\times\bR_+\to\bR$.
 \end{enumerate}
\end{thm}

The final claim of Theorem~\ref{thm:intro/full} is the most noteworthy, as it explains the seemingly complicated definition of the $\mY_k$-functional and its Euler equation~\eqref{eqn:intro/mYeuler}.  More precisely, while Theorem~\ref{thm:intro/variational} guarantees that there is a functional on $\kC$ for which its critical points are precisely those metric-measure structures with $\csigma_{2,\phi}^m$ constant, a computation involving the behavior of the weighted Bach tensor of a weighted Einstein manifold with $\mu=0$ implies that such manifolds are only critical points of the $\mY_2$-functional; cf.\ Remark~\ref{rk:Y2_unique}.

This article is organized as follows.  In Section~\ref{sec:algebra} we establish the key algebraic properties of the $m$-weighted elementary symmetric polynomials.  In Section~\ref{sec:smms} we discuss the necessary background for smooth metric measure spaces, including some important facts about weighted Einstein manifolds.  In Section~\ref{sec:moduli} we set up a useful formalism for studying the space of metric-measure structures.  In Section~\ref{sec:variational} we prove Theorem~\ref{thm:intro/variational}.  In Section~\ref{sec:stable} we prove Theorem~\ref{thm:intro/stable_qe} and discuss its analogue for weighted Einstein manifolds.  In Section~\ref{sec:obata} we prove the ellipticity of~\eqref{eqn:intro/mFeuler} and~\eqref{eqn:intro/mYeuler} within the appropriate elliptic cones, prove Theorem~\ref{thm:intro/obata}, and discuss our related conjectures for weighted Einstein manifolds and sharp fully nonlinear Sobolev inequalities.  In Section~\ref{sec:full} we prove Theorem~\ref{thm:intro/full}.

%% file: algebra.tex
\section{Algebraic preliminaries}
\label{sec:algebra}

We begin our study with a discussion of the algebraic properties of the $m$-weighted $k$-th elementary symmetric polynomials of $n$-variables.  Importantly, provided $m$ is sufficiently large relative to $k$ (cf.\ Theorem~\ref{thm:weighted_newton} below), these invariants possess all of the properties expected from the informal definition~\eqref{eqn:intro/informal_sigmak_defn}.  This section makes this assertion precise.  We begin with the formal definition of the weighted elementary symmetric polynomials.

\begin{defn}
 Fix $m\in\bR_+$ and $k,n\in\bN_0$.  The \emph{$m$-weighted $k$-th elementary symmetric polynomial $\sigma_k^m$ of $n$ variables} is the function $\sigma_k^m\colon\bR\times\bR^n$ defined recursively by
 \begin{align*}
  \sigma_0^m(\lambda;\Lambda) & = 1, \quad \text{if $k=0$}, \\
  \sigma_k^m(\lambda;\Lambda) & = \frac{1}{k}\sum_{j=0}^{k-1}(-1)^j\sigma_{k-1-j}^m(\lambda;\Lambda) N_{j+1}^m(\lambda;\Lambda), \quad \text{if $k\geq1$},
 \end{align*}
 where $\Lambda=(\lambda_1,\dotsc,\lambda_n)\in\bR^n$ and $N_k^m\colon\bR\times\bR^n\to\bR$ is defined by
 \[ N_k^m(\lambda;\Lambda) = m\left(\frac{\lambda}{m}\right)^k + \sum_{j=1}^n \lambda_j^k . \]
\end{defn}

One may also regard the weighted elementary symmetric polynomials as perturbations of the elementary symmetric polynomials through lower order terms.  In other words, $\sigma_k^m(\lambda;\Lambda)$ and $\sigma_k(\Lambda)$ differ by an inhomogeneous polynomial in $\Lambda$ of degree $k-1$.  The precise relationship is as follows.

\begin{lem}
 \label{lem:reln_to_trivial}
 Fix $m\in\bR_+$ and $k,n\in\bN_0$.  Then
 \begin{equation}
  \label{eqn:reln_to_trivial}
  \sigma_k^m\left(\lambda;\Lambda\right) = \sum_{j=0}^k\binom{m}{j}\left(\frac{\lambda}{m}\right)^j\sigma_{k-j}(\Lambda)
 \end{equation}
 for all $(\lambda;\Lambda)\in\bR\times\bR^n$.
\end{lem}

\begin{proof}
 The proof is by induction.  Clearly~\eqref{eqn:reln_to_trivial} holds if $k=0$.  Suppose that~\eqref{eqn:reln_to_trivial} holds for some $k\in\bN_0$.  The definition of the weighted elementary symmetric polynomials and the inductive hypothesis imply that
 \begin{equation}
  \label{eqn:rtt1}
  (k+1)\sigma_{k+1}^m(\lambda;\Lambda) = \sum_{j=0}^k\sum_{\ell=0}^{k-j}(-1)^\ell\binom{m}{j}\left(\frac{\lambda}{m}\right)^j\sigma_{k-j-\ell}(\Lambda)N_{\ell+1}^m(\lambda;\Lambda) .
 \end{equation}
 Using the identity $\sum_{j=0}^k(-1)^{k-j}\binom{m}{j}=\binom{m-1}{k}$, we compute that
 \begin{multline}
  \label{eqn:rtt2}
  \sum_{j=0}^k\sum_{\ell=0}^{k-j}(-1)^\ell\binom{m}{j}\left(\frac{\lambda}{m}\right)^{j+\ell+1}\sigma_{k-j-\ell}(\Lambda) \\ = \sum_{j=0}^k\binom{m-1}{j}\left(\frac{\lambda}{m}\right)^{j+1}\sigma_{k-j}(\Lambda) .
 \end{multline}
 Using the recursive definition of the elementary symmetric polynomials, we compute that
 \begin{multline}
  \label{eqn:rtt3}
  \sum_{j=0}^k\sum_{\ell=0}^{k-j}(-1)^\ell\binom{m}{j}\left(\frac{\lambda}{m}\right)^j\sigma_{k-j-\ell}(\Lambda)\sum_{s=1}^n\lambda_s^{\ell+1} \\ = \sum_{j=0}^k(k+1-j)\binom{m}{j}\left(\frac{\lambda}{m}\right)^j\sigma_{k+1-j}(\Lambda) .
 \end{multline}
 Combining~\eqref{eqn:rtt1}, \eqref{eqn:rtt2} and~\eqref{eqn:rtt3} yields the desired result.
\end{proof}

There is a similar relationship between weighted elementary symmetric polynomials when the value $\lambda$ is changed.

\begin{cor}
 \label{cor:reln_to_genl}
 Fix $m\in\bR_+$ and $k,n\in\bN_0$.  Then
 \[ \sigma_k^m\left(\lambda_1+\lambda_2;\Lambda\right) = \sum_{j=0}^k\binom{m}{j}\left(\frac{\lambda_1}{m}\right)^j\sigma_{k-j}^{m-j}\left(\frac{m-j}{m}\lambda_2;\Lambda\right) \]
 for all $\mu_1,\lambda_2\in\bR$ and $\Lambda\in\bR^n$.
\end{cor}

\begin{proof}
 Lemma~\ref{lem:reln_to_trivial} and the binomial theorem yield
 \[ \sigma_k^m\left(\lambda_1+\lambda_2;\Lambda\right) = \sum_{j=0}^k\sum_{s=0}^j\binom{m}{j}\binom{j}{s}\left(\frac{\lambda_1}{m}\right)^{j-s}\left(\frac{\lambda_2}{m}\right)^s\sigma_{k-j}(\Lambda) . \]
 Combining this with~\eqref{eqn:reln_to_trivial} yields the desired result.
\end{proof}

A useful fact is that the relationship between an elementary symmetric polynomial of $\Lambda=(\lambda_1,\dotsc,\lambda_n)$ and the corresponding elementary symmetric polynomial of $\Lambda(i):=(\lambda_1,\dotsc,\lambda_{i-1},\lambda_{i+1},\dotsc,\lambda_n)$ persists to the weighted case.

\begin{lem}
 \label{lem:reln_to_remove}
 Fix $m\in\bR_+$ and $k,n\in\bN_0$.  Then
 \begin{equation}
  \label{eqn:reln_to_remove}
  \sigma_k^m\left(\lambda;\Lambda\right) = \sigma_k^m\left(\lambda;\Lambda(i)\right) + \lambda_i\sigma_{k-1}^m\left(\lambda;\Lambda(i)\right) .
 \end{equation}
 for all $\lambda\in\bR$, $\Lambda=(\lambda_1,\dotsc,\lambda_n)\in\bR^n$, and $1\leq i\leq n$.
\end{lem}

\begin{proof}
 It is readily verified that~\eqref{eqn:reln_to_remove} holds when $m=0$.  Lemma~\ref{lem:reln_to_trivial} then implies that~\eqref{eqn:reln_to_remove} holds in general.
\end{proof}

Our interest is in considering the weighted elementary symmetric polynomials of a self-adjoint endomorphism of $\bR^n$.  These are defined in terms of the eigenvalues of the endomorphism.  In this context, there are also natural weighted analogues of the Newton transforms.

\begin{defn}
 \label{defn:newton_endomorphism}
 Fix $m\in\bR_+$ and $k,n\in\bN_0$.  Let $\mM_n$ denote the space of self-adjoint endomorphisms of $\bR^n$ with its standard inner product.
 \begin{enumerate}
  \item The \emph{$m$-weighted $k$-th elementary symmetric function} $\sigma_k^m\colon\bR\times\mM_n\to\bR$ is defined by $\sigma_k^m(\lambda;P):=\sigma_k^m\left(\lambda;\Lambda(P)\right)$, where $\Lambda(P):=(\lambda_1,\dotsc,\lambda_n)$ is the list of the eigenvalues (with multiplicity) of $P$.
  \item The \emph{$m$-weighted $k$-th Newton transform} $T_k^m\colon\bR\times\mM_n\to\mM_n$ is defined by
  \[ T_k^m(\lambda;P) = \sum_{j=0}^k (-1)^j\sigma_{k-j}^m(\lambda;P)\,P^j . \]
  \item The \emph{$m$-weighted $k$-th Newton scalar} $s_k^m\colon\bR\times\mM_n\to\bR$ is defined by
  \[ s_k^m\left(\lambda;P\right) = \sum_{j=0}^k (-1)^j \left(\frac{\lambda}{m}\right)^j \sigma_{k-j}^m\left(\lambda;P\right) . \]
 \end{enumerate}
\end{defn}

Informally, the $m$-weighted $k$-th elementary symmetric function $\sigma_k^m\left(\lambda;P\right)$ is the $k$-th elementary symmetric function of the $(m+n)\times (m+n)$ block-diagonal matrix
\begin{equation}
 \label{eqn:informal_block_diagonal}
 P \oplus \frac{\lambda}{m}\Id_m,
\end{equation}
where $\Id_m$ is the $m\times m$ identity matrix.  The $k$-Newton transform $T_k$ of~\eqref{eqn:informal_block_diagonal} decomposes as
\[ T_k = T_k^m \oplus s_k^m\Id_m . \]

The eigenvalues of the weighted Newton transforms are readily computed in terms of the weighted elementary symmetric functions.

\begin{lem}
 \label{lem:eigenvalues_T}
 Fix $m\in\bR_+$ and $k,n\in\bN_0$.  Given $\lambda\in\bR$ and $P\in\mM_n$, the eigenvalues of $T_k^m(\lambda;P)$ are $\sigma_k^m\left(\lambda;\Lambda(i)\right)$ for $1\leq i\leq n$, where $\Lambda=\Lambda(P)$ are the eigenvalues of $P$ and $\Lambda(i)$ is as in Lemma~\ref{lem:reln_to_remove}.
\end{lem}

\begin{proof}
 The proof is by induction.  It is clear that the conclusion holds when $k=0$.  Suppose that the eigenvalues of $T_k^m(\lambda;P)$ are $\sigma_k^m\left(\lambda;\Lambda(i)\right)$ for $1\leq i\leq n$.  Note that
 \begin{equation}
  \label{eqn:T_recursive}
  T_{k+1}^m(\lambda;P) = \sigma_{k+1}^m(\lambda;P)I - PT_k^m(\lambda;P)
 \end{equation}
 for $I$ the identity endomorphism.  The conclusion follows readily from Lemma~\ref{lem:reln_to_remove} and the inductive hypothesis.
\end{proof}

\begin{remark}
 It follows immediately from Definition~\ref{defn:newton_endomorphism} that the analogue
 \begin{equation}
  \label{eqn:s_recursive}
  s_{k+1}\left(\lambda;P\right) = \sigma_{k+1}\left(\lambda;P\right) - \frac{\lambda}{m}s_k^m\left(\lambda;P\right)
 \end{equation}
 of~\eqref{eqn:T_recursive} holds.
\end{remark}

The values of the weighted Newton scalars are also readily computed from the weighted elementary symmetric functions.

\begin{lem}
 \label{lem:s_to_sigma-1}
 Fix $m\in\bR_+$ and $k,n\in\bN_0$.  Then
 \begin{equation}
  \label{eqn:s_to_sigma-1}
  s_k^m\left(\lambda;P\right) = \sigma_k^{m-1}\left(\frac{m-1}{m}\lambda;P\right)
 \end{equation}
 for all $\lambda\in\bR$ and $P\in\mM_n$.
\end{lem}

\begin{proof}
 The proof is by induction.  It is clear that~\eqref{eqn:s_to_sigma-1} holds when $k=0$.  Suppose now that~\eqref{eqn:s_to_sigma-1} holds for all $j\in\{0,\dotsc,k-1\}$.  We thus compute that
 \begin{align*}
  k\sigma_k^{m-1}\left(\frac{m-1}{m}\lambda;P\right) & = \sum_{j=0}^{k-1}(-1)^js_{k-1-j}^m\left(\lambda;P\right)\left( N_{j+1}^m(\lambda;P) - \left(\frac{\lambda}{m}\right)^{j+1}\right) \\
  & = \sum_{j=0}^{k-1}\sum_{\ell=0}^{k-1-j}(-1)^{j+\ell}\left(\frac{\lambda}{m}\right)^\ell\sigma_{k-1-j-\ell}^m\left(\lambda;P\right)N_{j+1}^m(\lambda;P) \\
   & \quad - \sum_{j=0}^{k-1}\sum_{\ell=0}^{k-1-j}(-1)^{j+\ell}\left(\frac{\lambda}{m}\right)^{j+\ell+1}\sigma_{k-1-j-\ell}^m\left(\lambda;P\right) .
 \end{align*}
 Switching the order of the first summation and computing the second summation by summing over $\ell$ and then $j+\ell$ yields the desired result.
\end{proof}

Lemma~\ref{lem:s_to_sigma-1} yields another useful interpretation of the $m$-weighted Newton scalars.

\begin{lem}
 \label{lem:vary_kappa}
 Fix $m\in\bR_+$ and $k,n\in\bN_0$.  Let $\lambda\in\bR$ and $P\in\mM_n$.  Consider the function $S_k^m\colon\bR\to\bR$ defined by $S_k^m(\kappa):=\sigma_k^m\left(\lambda+m\kappa;P\right)$.  Then
 \[ \frac{dS_k^m}{d\kappa} = ms_{k-1}^m\left(\lambda+m\kappa;P\right) . \]
\end{lem}

\begin{proof}
 Expanding $S_k^m(\kappa)$ as a power series in $\kappa$ via Corollary~\ref{cor:reln_to_genl} and differentiating yields
 \[ \frac{dS_k^m}{d\kappa} = m\sum_{j=0}^\infty\binom{m-1}{j}\kappa^j\sigma_{k-1-j}^{m-1-j}\left(\frac{m-1-j}{m}\lambda;P\right) . \]
 Corollary~\ref{cor:reln_to_genl} then implies that
 \[ \frac{dS_k^m}{d\kappa} = m\sigma_{k-1}^{m-1}\left(\frac{m-1}{m}\left(\lambda+m\kappa\right);P\right) . \]
 The final conclusion now follows from Lemma~\ref{lem:s_to_sigma-1}.
\end{proof}

One difference between the weighted and unweighted case is that the weighted elementary symmetric functions are not recovered by taking inner products between the endomorphism and the associated weighted Newton transforms.  A specific relationship is as follows.

\begin{lem}
 \label{lem:newton_inner_product}
 Fix $m\in\bR_+$ and $k,n\in\bN_0$.  Then
 \[ \left\lp T_k^m(\lambda;P), P-\frac{\lambda}{m}I \right\rp = (k+1)\sigma_{k+1}^m(\lambda;P) - (m+n-k)\frac{\lambda}{m}\sigma_k^m(\lambda;P) \]
 for all $\lambda\in\bR$ and $P\in\mM_n$, where $I\in\mM_n$ is the identity and $\lp A,B\rp:=\tr (AB)$ is the standard inner product on $\mM_n$.
\end{lem}

\begin{proof}
 This follows immediately from the definitions of $\sigma_k^m$ and $T_k^m$ and the observation that
 \[ \left\lp P^j, P - \frac{\lambda}{m}I\right\rp = N_{j+1}^m\left(\lambda;\Lambda(P)\right) - \frac{\lambda}{m}N_j^m\left(\lambda;\Lambda(P)\right) \]
 for all $j\in\bN_0$.
\end{proof}

Nevertheless, there is a simple and useful relationship between inner products involving weighted Newton transforms and weighted elementary symmetric polynomials.

\begin{cor}
 \label{cor:newton_inner_product}
 Fix $m\in\bR_+$ and $k,n\in\bN_0$.  Define $E_k^m\colon\bR\times\mM_n\to\mM_n$ by
 \begin{equation}
  \label{eqn:weighted_tracefree_newton}
  E_k^m := T_k^m - \frac{m+n-k}{m+n}\sigma_k^m\,I
 \end{equation}
 for $I\in\mM_n$ the identity.  Then
 \begin{equation}
  \label{eqn:newton_inner_product_E}
  \left\lp E_k^m\left(\lambda;P\right), P - \frac{\lambda}{m}I\right\rp = (k+1)\sigma_{k+1}^m(\lambda;P) - \frac{m+n-k}{m+n}\sigma_1^m(\lambda;P)\sigma_k^m(\lambda;P)
 \end{equation}
 for all $\lambda\in\bR$ and $P\in\mM_n$.
\end{cor}

The significance of Corollary~\ref{cor:newton_inner_product} is contained in Corollary~\ref{cor:inner_product_maclaurin} below, which concludes that~\eqref{eqn:newton_inner_product_E} has a sign under a natural condition on $\lambda$ and $P$.  The intuition motivating the consideration of~\eqref{eqn:newton_inner_product_E} is as follows: one can regard~\eqref{eqn:weighted_tracefree_newton} as defining the trace-free part of the weighted Newton transform and the left-hand side of~\eqref{eqn:newton_inner_product_E} as the inner product of $E_k^m$ with $P$.  On the one hand, in the unweighted setting, the inner product of the trace-free part of a Newton transform with the underlying endomorphism is well-known to have a sign when the endomorphism lies in one of the G{\aa}rding cones (cf.\ \cite[Lemma~23]{Viaclovsky2000}).  On the other hand,  if $P\in\End(\bR^n)$, $\lambda\in\bR$, and $m\in\bN$, then the trac-efree part of the $k$-th Newton transform of~\eqref{eqn:informal_block_diagonal} is
\[ E_k := E_k^m \oplus \left(-\frac{1}{m}\tr E_k^m\right)I_m . \]
It follows that
\[ \left\lp E_k, P\oplus \frac{\lambda}{m}I_m\right\rp = \left\lp E_k^m, P - \frac{\lambda}{m}I_n\right\rp . \]

\subsection{The weighted Newton inequalities}
\label{subsec:algebra/netwon}

In this subsection we show that the $m$-weighted elementary symmetric polynomials of $n$-variables satisfy the same Newton inequalities as the elementary symmetric polynomials of $(m+n)$-variables.  Our proof of this fact is similar to the usual proof of the Newton inequalities (cf.\ \cite{HardyLittlewoodPolya1952}).  To that end, we use the following generating function for the weighted elementary symmetric polynomials.

\begin{prop}
 \label{prop:generating_function}
 Fix $m\in\bR_+$, $k,n\in\bN_0$, $\lambda\in\bR$, and $\Lambda=(\lambda_1,\dotsc,\lambda_n)\in\bR^n$.  Then
 \[ \left(1+\frac{\lambda t}{m}\right)^m\prod_{i=1}^n\left(1+\lambda_it\right) = \sum_{j=0}^\infty \sigma_j^m\left(\lambda;\Lambda\right)t^j \]
 for all $t\in(-M,M)$, where
 \[ M = \begin{cases}
         \infty, & \text{if $\lambda=0$ or $m\in\bN_0$,} \\
         \frac{m}{\lv\lambda\rv}, & \text{otherwise} .
        \end{cases} \]
\end{prop}

\begin{proof}
 It is well-known that
 \begin{align*}
  \prod_{i=1}^n\left(1+\lambda_it\right) & = \sum_{j=0}^n\sigma_j(\Lambda)t^j, \\
  \left(1+\frac{\lambda t}{m}\right)^m & = \sum_{j=0}^\infty\binom{m}{j}\left(\frac{\lambda}{m}\right)^j t^j
 \end{align*}
 for all $t\in(-M,M)$.  Multiplying these expressions yields
 \[ \left(1+\frac{\lambda t}{m}\right)^m\prod_{i=1}^n\left(1+\lambda_it\right) = \sum_{k=0}^\infty\sum_{j=0}^k\binom{m}{j}\left(\frac{\lambda}{m}\right)^j\sigma_{k-j}(\Lambda)t^k . \]
 The conclusion follows from Lemma~\ref{lem:reln_to_trivial}.
\end{proof}

We are now ready to prove the weighted Newton inequalities.

\begin{thm}
 \label{thm:weighted_newton}
 Let $k\in\bN$ and let $m\in[k-1,\infty)$.  Then
 \begin{equation}
  \label{eqn:weighted_newton}
  \sigma_{k-1}^m\left(\lambda;\Lambda\right)\sigma_{k+1}^m\left(\lambda;\Lambda\right) \leq \frac{k(m+n-k)}{(k+1)(m+n-k+1)}\left(\sigma_k^m(\lambda;\Lambda)\right)^2
 \end{equation}
 for all $\lambda\in\bR$ and $\Lambda\in\bR^n$.  Moreover, equality holds if and only if one of the following holds:
 \begin{enumerate}
  \item $\Lambda=(\lambda/m,\dotsc,\lambda/m)$;
  \item $\lambda=0$ and at most $k-1$ components of $\Lambda$ are nonzero;
  \item $m=k-1$ and $\Lambda=(0,\dotsc,0)$.
 \end{enumerate}
\end{thm}

\begin{remark}
 If $m\not\in\bN_0$, the assumption $m\geq k-1$ is necessary.  This can be seen by computing both sides of~\eqref{eqn:weighted_newton} with $\Lambda=0$ and $\lambda\in\bR$ arbitrary.
\end{remark}

\begin{proof}
 Set $p_k^m:=(m+n)^k\binom{m+n}{k}^{-1}\sigma_k^m$, so that, as a functional inequality on $\bR\times\bR^n$, \eqref{eqn:weighted_newton} is equivalent to
 \begin{equation}
  \label{eqn:weighted_newton_p}
  p_{k-1}^mp_{k+1}^m \leq \left(p_k^m\right)^2 .
 \end{equation}
 If $m\in\bN_0$, the conclusion follows from the usual Newton inequality~\cite{HardyLittlewoodPolya1952}.  Suppose now that $m\not\in\bN_0$.  We separate the proof into three cases.
 
 First, suppose that $\lambda=0$.  Lemma~\ref{lem:reln_to_trivial} and the (unweighted) Newton inequalities imply that
 \[ \sigma_{k-1}^m\sigma_{k+1}^m \leq \frac{k(n-k)}{(k+1)(n-k+1)}\left(\sigma_k^m\right)^2 \leq \frac{k(m+n-k)}{(k+1)(m+n-k+1)}\left(\sigma_k^m\right)^2 \]
 with equality if and only if at most $k-1$ components of $\Lambda$ are nonzero.
 
 Second, suppose $k=1$.  Write $\Lambda=(\lambda_1,\dotsc,\lambda_n)$.  We compute that
 \begin{align*}
  p_2^m(\lambda;\Lambda) & = \frac{m+n}{m+n-1}\left( \left(p_1^m(\lambda;\Lambda)\right)^2 - \frac{\lambda^2}{m} - \sum_{s=1}^n\lambda_s^2\right) \\
  & \leq \frac{m+n}{m+n-1}\left( \left(p_1^m(\lambda;\Lambda)\right)^2 - \frac{\lambda^2}{m} - \frac{1}{n}\left(\sum_{s=1}^n \lambda_s\right)^2 \right) \\
  & = \left(p_1^m(\lambda;\Lambda)\right)^2 - \frac{m}{n(m+n-1)}\left(\sum_{s=1}^n \lambda_s-\frac{n\lambda}{m}\right)^2
 \end{align*}
 with equality if and only if $\lambda_1=\dotsb=\lambda_n$.  The conclusion readily follows.
 
 Third, suppose $\lambda\not=0$ and $k\geq2$.  Set
 \[ P(t) = \left(1+\frac{\lambda t}{m}\right)^m\prod_{j=1}^n\left(1+\lambda_jt\right) . \]
 By Proposition~\ref{prop:generating_function},
 \begin{equation}
  \label{eqn:gf0}
  P(t)=\sum_{j=0}^\infty\binom{m+n}{j}\left(\frac{t}{m+n}\right)^j p_j^m .
 \end{equation}
 Write $\Lambda=(\lambda_1,\dotsc,\lambda_n)$ and let $\ell=\left|\left\{j\suchthat\lambda_j=0\right\}\right|$ denote the number of components of $\Lambda$ which vanish.  Up to reindexing the components of $\Lambda$, we see that
 \[ P(t) = \left(1+\frac{\lambda t}{m}\right)^mQ_0(t) \]
 for $Q_0$ a polynomial of degree $n-\ell$ with $Q_0(0)=1$ and roots $r_j=-1/\lambda_j$ such that $r_1^{(0)}\leq \dotsb \leq r_{n-\ell}^{(0)}$ are all nonzero.
 
 For any $s\in\{0,\dotsc,k-1\}$, we compute from~\eqref{eqn:gf0} that
 \begin{equation}
  \label{eqn:gfs}
  \frac{d^sP}{dt^s} = \frac{(m+n)!}{(m+n)^s(m+n-s)!}\sum_{j=0}^\infty\binom{m+n-s}{j}\left(\frac{t}{m+n}\right)^j p_{j+s}^m .
 \end{equation}
 On the other hand, by regarding $P$ as an analytic function in $\bC\setminus\left\{-\frac{m}{\lambda}+is\suchthat s\leq0\right\}$ and applying Rolle's Theorem along the real rays $x\geq-\frac{m}{\lambda}$ and $x\leq-\frac{m}{\lambda}$, we deduce that
 \begin{equation}
  \label{eqn:PQ0}
  \frac{d^sP}{dt^s} = \left(1+\frac{\lambda t}{m}\right)^{m-s}Q_s(t)
 \end{equation}
 for $Q_s$ a polynomial of degree $n-\ell$ with roots $r_1^{(s)}\leq\dotsb\leq r_{n-\ell}^{(s)}$, at most one of which is zero.
 
 We now consider two subcases.  Suppose first that there is a $j\in\{1,\dotsc,n-\ell\}$ such that $r_j^{(k-1)}=0$.  Then $Q_{k-1}(t)=ct+O(t^2)$ near $t=0$ for some $c\not=0$.  Comparing~\eqref{eqn:gfs} and~\eqref{eqn:PQ0} yields $p_{k-1}^m=0$ and $p_k^m\not=0$, from which the desired conclusion readily follows.
 
 Suppose instead that $r_j^{(k-1)}\not=0$ for all $j\in\{1,\dotsc,n-\ell\}$.  Define
 \[ \lambda_j^{(k-1)} = \begin{cases}
                         -1/r_j^{(k-1)}, & \text{if $j\leq n-\ell$} \\
                         0, & \text{otherwise} .
                        \end{cases} \]
 and set $\Lambda^{(k-1)}:=(\lambda_1^{(k-1)},\dotsc,\lambda_n^{(k-1)})$.  Proposition~\ref{prop:generating_function} and comparison of~\eqref{eqn:gfs} and~\eqref{eqn:PQ0} imply that
 \[ p_j^{m-k+1}\left(\frac{m-k+1}{m}\lambda;\Lambda^{(k-1)}\right) = \left(\frac{m+n-k+1}{m+n}\right)^j\frac{p_{k+j-1}^m(\lambda;\Lambda)}{p_{k-1}^m(\lambda;\Lambda)} \]
 for all $j\in\bN_0$.  In particular, applying this to the cases $j\in\{0,1,2\}$ yields~\eqref{eqn:weighted_newton_p} with equality if and only if $\lambda_j^{(k-1)}=\frac{\lambda}{m}$ for all $j\in\{1,\dotsc,n\}$.  Our applications of Rolle's Theorem imply that the latter conclusion is equivalent to $\lambda_j=\frac{\lambda}{m}$ for all $j\in\{1,\dotsc,n\}$, as desired. 
\end{proof}

\subsection{Weighted elliptic cones}
\label{subsec:algebra/cones}

An important feature of the elementary symmetric polynomials is that they are monotone with respect to a single variable within the G{\aa}rding cones; this is closely related to the ellipticity of the $\sigma_k$-curvature prescription problem (cf.\ \cite{CaffarelliNirenbergSpruck1985,Garding1959,Viaclovsky2000}).  There are similar cones in the weighted case; we state our definition only in terms of the space $\mM_n$ of self-adjoint endomorphisms of $\bR^n$, from which one could easily formulate analogous definitions in terms of the weighted elementary symmetric polynomials.

\begin{defn}
 Fix $m\in\bR_+$ and $k,n\in\bN_0$.  The \emph{positive (resp.\ negative) $m$-weighted elliptic $k$-cone} is the set $\Gamma_k^{m,+}$ (resp.\ $\Gamma_k^{m,-}$) defined by
 \[ \Gamma_k^{m,\pm} := \left\{ (\lambda;P)\in\bR\times\mM_n \suchthat (\pm1)^j\sigma_j^m(\lambda;P)>0 \text{ for all $1\leq j\leq k$} \right\} . \]
\end{defn}

One useful fact is that~\eqref{eqn:newton_inner_product_E} has a sign in the appropriate weighted elliptic cone.

\begin{prop}
 \label{prop:weighted_maclaurin}
 Let $k,n\in\bN_0$ and $m\in[k-1,\infty)$.  Then
 \begin{equation}
  \label{eqn:weighted_maclaurin}
  (\pm1)^{k+1}\sigma_{k+1}^m(\lambda;P) \leq (\pm1)^{k+1}\frac{m+n-k}{(m+n)(k+1)}\sigma_1^m(\lambda;P)\sigma_k^m(\lambda;P)
 \end{equation}
 for all $(\lambda;P)\in\Gamma_k^{m,\pm}$.  Moreover, equality holds if and only if $P=\frac{\lambda}{m}I$.
\end{prop}

\begin{proof}
 The proof is by induction.  Observe that, as a functional inequality on $\Gamma_k^{m,\pm}$, \eqref{eqn:weighted_maclaurin} is equivalent to
 \begin{equation}
  \label{eqn:weighted_maclaurin_p}
  (\pm1)^{k+1}p_{k+1}^m \leq (\pm1)^{k+1}p_1^mp_k^m .
 \end{equation}
 Theorem~\ref{thm:weighted_newton} implies that~\eqref{eqn:weighted_maclaurin_p} holds when $k=1$.  Suppose now that~\eqref{eqn:weighted_newton_p} holds in $\Gamma_k^{m,\pm}$.  Since $\Gamma_{k+1}^{m,\pm}\subset\Gamma_k^{m,\pm}$, it follows from Theorem~\ref{thm:weighted_newton} and the inductive hypothesis that
 \[ p_{k}^mp_{k+2}^m \leq \left((\pm1)^{k+1}p_{k+1}^m\right)^2 \leq p_1^mp_k^mp_{k+1}^m \]
 in $\Gamma_{k+1}^{m,\pm}$.  Dividing both sides by $(\pm1)^kp_k^m>0$ yields the desired result.
\end{proof}

\begin{cor}
 \label{cor:inner_product_maclaurin}
 Let $k,n\in\bN_0$ and $m\in[k-1,\infty)$.  Then
 \[ (\pm1)^{k+1}\left\lp E_k^m(\lambda;P), P-\frac{\lambda}{m}I\right\rp \leq 0 \]
 for all $(\lambda;P)\in\Gamma_k^{m,\pm}$.  Moreover, equality holds if and only if $P=\frac{\lambda}{m}I$.
\end{cor}

\begin{proof}
 This follows immediately from Corollary~\ref{cor:newton_inner_product} and Proposition~\ref{prop:weighted_maclaurin}.
\end{proof}

Another useful fact is that the weighted Newton transform and the weighted Newton scalar have a sign in the corresponding weighted elliptic cone.  This result encodes the relationship between the weighted elliptic cones and ellipticity of the weighted $\sigma_k$-curvatures; see Proposition~\ref{prop:linearization_sigmak} below.

\begin{cor}
 \label{cor:weighted_elliptic}
 Let $k\in\bN$, $n\in\bN_0$ and $m\in[k-1,\infty)$.  Suppose $(\lambda;P)\in\Gamma_k^{m,\pm}$.  Then
 \[ (\pm1)^{k-1}T_{k-1}^m(\lambda;P)>0 \quad\text{and}\quad (\pm1)^{k-1}s_{k-1}^m(\lambda;P)>0 . \]
\end{cor}

\begin{proof}
 Let $\Lambda=\Lambda(P)$ denote the eigenvalues of $P$.  By Lemma~\ref{lem:eigenvalues_T}, the conclusion is equivalent to the assertions
 \begin{subequations}
  \label{eqn:we1}
  \begin{align}
   \label{eqn:we1/T} (\pm1)^{k-1}\sigma_{k-1}^m\left(\lambda;\Lambda(i)\right) & > 0\qquad\text{for all $1\leq i\leq n$} , \\
   \label{eqn:we1/s} (\pm1)^{k-1}s_{k-1}^m\left(\lambda;P\right) & > 0.
  \end{align}
 \end{subequations}
 We prove~\eqref{eqn:we1} by induction on $k$.  Clearly~\eqref{eqn:we1} holds when $k=1$.  Suppose that~\eqref{eqn:we1} holds for all $(\lambda;P)\in\Gamma_k^{m,\pm}$.  Let $(\lambda;P)\in\Gamma_{k+1}^{m,\pm}$.  In particular, $(\pm1)^{k+1}\sigma_{k+1}^m(\lambda;P)>0$.  Lemma~\ref{lem:reln_to_remove} and~\eqref{eqn:s_recursive} respectively imply that
 \begin{subequations}
  \label{eqn:we2}
  \begin{align}
   \label{eqn:we2/T} (\pm1)^{k+1}\left[\sigma_{k+1}^m\left(\lambda;\Lambda(i)\right) + \lambda_i\sigma_k^m\left(\lambda;\Lambda(i)\right)\right] & > 0, \\
   \label{eqn:we2/s} (\pm1)^{k+1}\left[s_{k+1}^m\left(\lambda;P\right) + \frac{\lambda}{m}s_k^m\left(\lambda;P\right) \right] & > 0
  \end{align}
 \end{subequations}
 for all $1\leq i\leq n$.  By the inductive hypothesis, $(\pm1)^{k-1}\sigma_{k-1}^m\left(\lambda;\Lambda(i)\right)>0$ and $(\pm1)^{k-1}s_{k-1}^m\left(\lambda;P\right)>0$.  Multiplying these to both sides of~\eqref{eqn:we2/T} and~\eqref{eqn:we2/s}, respectively, and then using Theorem~\ref{thm:weighted_newton} (and Lemma~\ref{lem:s_to_sigma-1} for the inequality involving $s_k^m$) and Lemma~\ref{lem:reln_to_remove} in succession yields
 \begin{align*}
  0 & < \sigma_k^m\left(\lambda;\Lambda(i)\right)\left[ \sigma_k^m\left(\lambda;\Lambda(i)\right) + \lambda_i\sigma_{k-1}^m\left(\lambda;\Lambda(i)\right) \right] = \sigma_k^m\left(\lambda;\Lambda(i)\right)\sigma_k^m\left(\lambda;P\right) , \\
  0 & < s_k^m(\lambda;P)\left( s_k^m(\lambda;P) + \frac{\lambda}{m}s_{k-1}^m(\lambda;P)\right) = s_{k}^m(\lambda;P)\sigma_{k}^m(\lambda;P) .
 \end{align*}
 The conclusion now follows from the assumption $(\lambda;P)\in\Gamma_{k+1}^{m,\pm}$.
\end{proof}

%% file: smms.tex
\section{Smooth metric measure spaces}
\label{sec:smms}

Recall that a smooth metric measure space is a five-tuple $(M^n,g,v,m,\mu)$ consisting of a Riemannian manifold $(M^n,g)$, a positive function $v\in C^\infty(M;\bR_+)$, a dimensional parameter $m\in\bR_+$, and an auxiliary curvature parameter $\mu\in\bR$.  The geometric study of smooth metric measure spaces is based on \emph{weighted local invariants} of smooth metric measure spaces; i.e.\ tensor-valued functions on the space $\kM(M,m,\mu)$ of metric-measure structures on $(M^n,m,\mu)$ which are invariant with respect to the natural action of the diffeomorphism group of $M$ (see Section~\ref{sec:moduli} for further discussion).  Many weighted local invariants can be realized as local Riemannian invariants of the formal warped product~\eqref{eqn:intro/wp}.  In this section, we intrinsically define and discuss the weighted local invariants which are important to our study of the weighted $\sigma_k$-curvatures, as well as the properties of these invariants for weighted Einstein manifolds.

Among the most familiar weighted local invariants of a smooth metric measure space $(M^n,g,v,m,\mu)$ are the \emph{Bakry--\'Emery Ricci tensor}
\[ \Ric_\phi^m := \Ric + \nabla^2\phi - \frac{1}{m}d\phi\otimes d\phi, \]
the \emph{weighted Laplacian}
\[ \Delta_\phi := \Delta - \nabla\phi , \]
and the \emph{weighted volume element}
\[ d\nu := v^m\,\dvol_g . \]
The Bakry--\'Emery Ricci tensor plays an important role in the comparison geometry of smooth metric measure spaces (cf.\ \cite{Wei_Wylie}).  This is because of its appearance in the weighted Bochner formula for the weighted Laplacian on functions.  Note also that the weighted Laplacian is the natural formally self-adjoint (rough) Laplacian on smooth metric measure spaces, in that $\Delta_\phi=-\nabla^\ast\nabla$, where $\nabla^\ast$ is the adjoint of the Levi-Civita connection $\nabla$ of $g$ with respect to the $L^2$-inner product induced by the weighted volume element.

Another important local invariant of a smooth metric measure space is the \emph{weighted scalar curvature}
\[ R_\phi^m := R + 2\Delta\phi - \frac{m+1}{m}\lv\nabla\phi\rv^2 + m(m-1)\mu e^{\frac{2}{m}\phi} . \]
Among the reasons that $R_\phi^m$ is the natural analogue of the scalar curvature are that it is the scalar curvature of the warped product~\eqref{eqn:intro/wp}, it plays the role of the scalar curvature in O'Neill's submersion theorem~\cite{Lott2007} and in the weighted Weitzenb\"ock formula for the Dirac operator on spinors~\cite{Perelman1}, and the variational properties of the total weighted scalar curvature functional are closely related to sharp Sobolev inequalities~\cite{Case2013y} and special Einstein-type structures~\cite{Case2010b,Case2013y}.

Two smooth metric measure spaces $(M^n,g,v,m,\mu)$ and $(M^n,\hg,\hv,m,\mu)$ are \emph{pointwise conformally equivalent} if there is a positive function $u\in C^\infty(M;\bR_+)$ such that $\hg=u^{-2}g$ and $\hv=u^{-1}v$.  Given $(M^n,g,v,m,\mu)$, we denote by $\kC$ the set of all $(\hg,\hv)$ such that $(M^n,\hg,\hv,m,\mu)$ is pointwise conformally equivalent to $(M^n,g,v,m,\mu)$.  A smooth metric measure space $(M^n,g,v,m,\mu)$, $m\not=1$, is \emph{locally conformally flat in the weighted sense} if around every point there is a neighborhood $U$ such that the restriction $(U,g\rv_{TU},v\rv_U,m,\mu)$ is conformally equivalent to $(B,g_{-\mu},1,m,\mu)$, where $B$ is an open set in the simply-connected spaceform $(X^n,g_{-\mu})$ with constant sectional curvature $-\mu$.  A smooth metric measure space $(M^n,g,v,1,\mu)$ is \emph{locally conformally flat in the weighted sense} if around every point there is a neighborhood $U$ such that the restriction $(U,g\rv_{TU},v\rv_U,m,\mu)$ is conformally equivalent to $(B,g_c,1,1,\mu)$, where $B$ is an open set in a simply-connected spaceform $(X^n,g_c)$.

In order to discuss the conformal properties of smooth metric measure spaces, it is convenient to consider the following modifications of the weighted scalar curvature, the Bakry--\'Emery Ricci tensor, and the Riemann curvature tensor $\Rm$:
\begin{align*}
 J_\phi^m & := \frac{m+n-2}{2(m+n-1)}R_\phi^m, \\
 P_\phi^m & := \Ric_\phi^m - \frac{1}{m+n-2}J_\phi^mg, \\
 A_\phi^m & := \Rm - \frac{1}{m+n-2}P_\phi^m\wedge g .
\end{align*}
Here $\wedge\colon \Gamma\left(S^2T^\ast M\right)\times \Gamma\left(S^2T^\ast M\right)\to \Gamma\left(\Lambda^2S^2T^\ast M\right)$ denotes the Kulkarni--Nomizu product.  We call $P_\phi^m$ the \emph{weighted Schouten tensor} and $A_\phi^m$ the \emph{weighted Weyl tensor}.  By Lemma~\ref{lem:conf_change} below, the weighted Weyl tensor is a weighted conformal invariant of smooth metric measure spaces, so that, in a weighted conformal class, the Riemann curvature tensor is completely controlled by the weighted Schouten tensor.  Indeed, it is straightforward to check that a smooth metric measure space with $n\geq3$ and $m+n\not=3$ is locally conformally flat in the weighted sense if and only if $A_\phi^m=0$; cf.\ \cite[Lemma~6.6]{Case2011t}.

The scalar invariant $J_\phi^m$ should be regarded as the weighted analogue of the trace of the Schouten tensor.  However, it is not the trace of the weighted Schouten tensor.  The following formula for the difference $Y_\phi^m:=J_\phi^m-\tr P_\phi^m$ will be useful.

\begin{lem}
 \label{lem:eval_Yphim}
 Let $(M^n,g,v,m,\mu)$ be a smooth metric measure space.  Then
 \[ Y_\phi^m = \Delta_\phi\phi - \frac{m}{m+n-2}J_\phi^m + m(m-1)\mu v^{-2} . \]
\end{lem}

\begin{proof}
 The definitions of the weighted scalar curvature and the Bakry--\'Emery Ricci tensor yield
 \[ R_\phi^m = \tr\Ric_\phi^m + \Delta_\phi\phi + m(m-1)\mu v^{-2}, \]
 from which the result readily follows.
\end{proof}

We also need two tensors formed by taking certain derivatives of the weighted Schouten tensor.  The \emph{weighted Cotton tensor} $dP_\phi^m\in\Gamma\left(\Lambda^2T^\ast M\otimes T^\ast M\right)$ is defined by
\[ dP_\phi^m(x,y,z) := \nabla_x P_\phi^m(y,z) - \nabla_y P_\phi^m(x,z) \]
for all $x,y,z\in T_pM$ and all $p\in M$.  The \emph{weighted Bach tensor $B_\phi^m\in\Gamma\left(S^2T^\ast M\right)$} is defined by
\[ B_\phi^m(x,y) := (\delta_\phi dP_\phi^m)(x,y) - \frac{1}{m}d\phi(y)\,\tr dP_\phi^m(\cdot,x,\cdot) + \left\lp A_\phi^m(\cdot,x,\cdot,y), P_\phi^m - \frac{Y_\phi^m}{m}g\right\rp \]
for all $x,y,z\in T_pM$ and all $p\in M$, where we define
\[ (\delta_\phi dP_\phi^m)(x,y) := \sum_{i=1}^n \nabla_{e_i}dP_\phi^m(e_i,x,y) - dP_\phi^m(\nabla\phi,x,y) \]
for $\{e_i\}_1^n$ an orthonormal basis for $T_pM$; see~\cite{Case2011t,Case2011o} for further discussion. The following identities involving the weighted Schouten, weighted Cotton, and weighted Weyl tensors are useful; see~\cite{Case2011t} for their derivations.

\begin{lem}
 \label{lem:div_and_tr}
 Let $(M^n,g,v,m,\mu)$ be a smooth metric measure space.  Then
 \begin{align*}
  \tr dP_\phi^m & = P_\phi^m(\nabla\phi) + dY_\phi^m - \frac{1}{m}Y_\phi^m\,d\phi , \\
  \tr A_\phi^m & = \frac{m}{m+n-2}P_\phi^m - \nabla^2\phi + \frac{1}{m}d\phi\otimes d\phi + \frac{1}{m+n-2}Y_\phi^mg, \\
  \delta_\phi P_\phi^m & = dJ_\phi^m - \frac{1}{m}Y_\phi^m\,d\phi , \\
  \delta_\phi A_\phi^m & = \frac{m+n-3}{m+n-2}dP_\phi^m - v^{-1}dv\wedge\tr A_\phi^m .
 \end{align*}
\end{lem}

Note that Lemma~\ref{lem:div_and_tr} implies that $dP_\phi^m$ vanishes if $(M^n,g,v,m,\mu)$ is locally conformally flat in the weighted sense.

We also need to know the behavior of these weighted curvatures under pointwise conformal transformations; see~\cite{Case2011t} for derivations.

\begin{lem}
 \label{lem:conf_change}
 Let $(M^n,g,v,m,\mu)$ and $(M^n,\hg,\hv,m,\mu)$ be pointwise conformally equivalent smooth metric measure spaces.  Define $f\in C^\infty(M)$ by $\hg=e^{-\frac{2f}{m+n-2}}g$.  Then
 \begin{align*}
  e^{-\frac{2f}{m+n-2}}\hJ_\phi^m & = J_\phi^m + \Delta_\phi f - \frac{1}{2}\lv\nabla f\rv^2, \\
  \hP_\phi^m & = P_\phi^m + \nabla^2f + \frac{1}{m+n-2}df\otimes df - \frac{1}{2(m+n-2)}\lv\nabla f\rv^2g, \\
  \hA_\phi^m & = e^{\frac{2f}{m+n-2}}A_\phi^m, \\
  \widehat{dP_\phi^m} & = dP_\phi^m - A_\phi^m(\cdot,\cdot,\nabla f,\cdot) .
 \end{align*}
\end{lem}

\subsection{The weighted $\sigma_k$-curvatures}
\label{subsec:smms/defn}

We are now prepared to define the weighted $\sigma_k$-curvatures of a smooth metric measure space.

\begin{defn}
 Fix $k\in\bN_0$ and $\kappa\in\bR$.  The \emph{weighted $\sigma_k$-curvature (with scale $\kappa$)} of a smooth metric measure space $(M^n,g,v,m,\mu)$ is
 \begin{equation}
  \label{eqn:defn_csigma}
  \csigma_{k,\phi}^m := \sigma_k^m\left(Y_\phi^m+m\kappa v^{-1};P_\phi^m\right) .
 \end{equation}
 The \emph{$k$-th weighted Newton tensor (with scale $\kappa$)} of $(M^n,g,v,m,\mu)$ is
 \begin{equation}
  \label{eqn:defn_cT}
  \cT_{k,\phi}^m := T_k^m\left(Y_\phi^m+m\kappa v^{-1};P_\phi^m\right) .
 \end{equation}
 The \emph{$k$-th weighted Newton scalar (with scale $\kappa$)} of $(M^n,g,v,m,\mu)$ is
 \begin{equation}
  \label{eqn:defn_cs}
  \cs_{k,\phi}^m := s_k^m\left(Y_\phi^m+m\kappa v^{-1};P_\phi^m\right) .
 \end{equation}
\end{defn}

We shall omit the tilde from our notation and denote by $\sigma_{k,\phi}^m$, $T_{k,\phi}^m$, and $s_{k,\phi}^m$ the weighted $\sigma_k$-curvature, the $k$-th weighted Newton tensor, and the $k$-th weighted Newton scalar, respectively, with scale $\kappa=0$.  In order to give more succinct derivations, given a smooth metric measure space $(M^n,g,v,m,\mu)$, a parameter $\kappa\in\bR$, and a nonnegative integer $k$, we denote
\begin{align*}
 \cY_\phi^m & := Y_\phi^m+m\kappa v^{-1}, \\
 \cZ_\phi^m & := \frac{1}{m}\cY_\phi^m , \\
 \cN_{k,\phi}^m & := \tr\left(P_\phi^m\right)^k + m\bigl(\cZ_\phi^m\bigr)^k .
\end{align*}

As discussed in Section~\ref{sec:variational} below, the variational properties of the weighted $\sigma_k$-curvatures are closely related to the properties of the weighted divergence of the $k$-th weighted Newton tensor.  For example, the fact~\cite{Case2013y} that the weighted $\sigma_1$-curvature is variational is closely related to the following formula for the divergence of the first weighted Newton tensor.

\begin{lem}
 \label{lem:div_T1}
 Let $(M^n,g,v,m,\mu)$ be a smooth metric measure space and fix $\kappa\in\bR$.  Then
 \[ \delta_\phi \cT_{1,\phi}^m = -\cs_{1,\phi}^m\,d\phi . \]
\end{lem}

\begin{proof}
 By definition~\eqref{eqn:defn_cT}, we have that
 \[ \cT_{1,\phi}^m = \left(J_\phi^m+m\kappa v^{-1}\right)g - P_\phi^m . \]
 Lemma~\ref{lem:div_and_tr} implies that
 \[ \delta_\phi \cT_{1,\phi}^m = -\left(J_\phi^m + (m-1)\kappa v^{-1} - \frac{Y_\phi^m}{m}\right)d\phi . \]
 Comparing with~\eqref{eqn:defn_cs} yields the desired result.
\end{proof}

Lemma~\ref{lem:div_T1} allows us to compute the weighted divergence of the weighted Newton tensors of all orders.

\begin{prop}
 \label{prop:div_T}
 Let $(M^n,g,v,m,\mu)$ be a smooth metric measure space and fix $\kappa\in\bR$.  Let $k\in\bN_0$.  Then
 \[ \delta_\phi\cT_{k,\phi}^m = -\cs_{k,\phi}^m\,d\phi + \sum_{j=0}^{k-2} (-1)^j\cT_{k-2-j,\phi}^m\left( dP_\phi^m \cdot \cQ_{j+1,\phi}^m \right) , \]
 where
 \[ \cQ_{\ell,\phi}^m := \bigl(P_\phi^m\bigr)^\ell - \bigl(\cZ_\phi^m\bigr)^\ell g \]
 and $\cdot\colon \Gamma\left(\otimes^3T^\ast M\right)\times \Gamma\left(S^2T^\ast M\right) \to \Gamma\left(T^\ast M\right)$ denotes the fiber-wise contraction of the second argument into the first and third components of the first argument; i.e.\ $(A\cdot T)_\gamma:=A_{\alpha\gamma\beta}T^{\alpha\beta}$.
\end{prop}

\begin{proof}
 We first show by induction that
 \begin{equation}
  \label{eqn:pf_div_T_nabla_sigma}
  d\csigma_{k,\phi}^m = \sum_{j=0}^{k-1}\frac{(-1)^j}{j+1}\csigma_{k-1-j,\phi}^m\,d\cN_{j+1,\phi}^m
 \end{equation}
 for all $k\in\bN_0$, with the convention that the empty sum equals zero.  Clearly~\eqref{eqn:pf_div_T_nabla_sigma} holds for $k=0$.  Suppose that~\eqref{eqn:pf_div_T_nabla_sigma} holds for some $k\in\bN_0$.  Using the definition
 \begin{equation}
  \label{eqn:defn_csigma_expanded}
  (k+1)\csigma_{k+1,\phi}^m = \sum_{j=0}^k(-1)^j\csigma_{k-j,\phi}^m\cN_{j+1,\phi}^m
 \end{equation}
 of the weighted $\sigma_k$-curvature, we compute that
 \begin{align*}
  (k+1)d\csigma_{k+1,\phi}^m & = \sum_{j=0}^k(-1)^j\csigma_{k-j,\phi}^m\,d\cN_{j+1,\phi}^m \\
   & \quad + \sum_{j=0}^k\sum_{\ell=0}^{k-1-j} \frac{(-1)^{j+\ell}}{\ell+1}\csigma_{k-1-j-\ell,\phi}^m\cN_{j+1,\phi}^m\,d\cN_{\ell+1,\phi}^m \\
   & = (k+1)\sum_{j=0}^k\frac{(-1)^j}{j+1}\csigma_{k-j,\phi}^m\,d\cN_{j+1,\phi}^m,
 \end{align*}
 as desired.

 We next show that
 \begin{multline}
  \label{eqn:pf_div_T_div_P}
  \delta_\phi\left(P_\phi^m\right)^k+\bigl(\cZ_\phi^m\bigr)^k\,d\phi = \sum_{j=0}^{k-1}\frac{1}{j+1}\left(P_\phi^m\right)^{k-1-j}\left(\nabla\cN_{j+1,\phi}^m\right) \\ + \sum_{j=0}^{k-2}\left(P_\phi^m\right)^{k-2-j}\left(dP_\phi^m\cdot \cQ_{j+1,\phi}^m\right)
 \end{multline}
 for all $k\in\bN_0$.  By definition,
 \begin{equation}
  \label{eqn:pf_div_T_div_P_step1}
  \delta_\phi\left(P_\phi^m\right)^k = \left(P_\phi^m\right)^{k-1}\left(\delta_\phi P_\phi^m\right) + \sum_{j=1}^{k-1} \left(P_\phi^m\right)^{k-1-j}\left((\nabla P_\phi^m)\cdot\left(P_\phi^m\right)^j\right) .
 \end{equation}
 Observe that
 \begin{equation}
  \label{eqn:pf_div_T_div_P_step2}
  (\nabla P_\phi^m)\cdot\left(P_\phi^m\right)^j = \frac{1}{j+1}d\cN_{j+1,\phi}^m - \bigl(\cZ_\phi^m\bigr)^j\,d\cY_\phi^m + dP_\phi^m\cdot\left(P_\phi^m\right)^j .
 \end{equation}
 Using Lemma~\ref{lem:div_and_tr} to evaluate $\tr dP_\phi^m:=dP_\phi^m\cdot g$, using Lemma~\ref{lem:div_T1} to evaluate $\delta_\phi P_\phi^m$, and inserting~\eqref{eqn:pf_div_T_div_P_step2} into~\eqref{eqn:pf_div_T_div_P_step1} yields~\eqref{eqn:pf_div_T_div_P}.

 Finally, the definition~\eqref{eqn:defn_cT} of the $k$-th weighted Newton tensor implies that
 \[ \delta_\phi\cT_{k,\phi}^m = \sum_{j=0}^k (-1)^j\left[ \left(P_\phi^m\right)^j(\nabla\csigma_{k-j,\phi}^m) + \csigma_{k-j,\phi}^m\delta_\phi\left(P_\phi^m\right)^j \right] . \]
 Using~\eqref{eqn:pf_div_T_nabla_sigma} and~\eqref{eqn:pf_div_T_div_P} to evaluate the right-hand side of the above display yields the desired result.
\end{proof}

\subsection{Weighted Einstein manifolds}
\label{subsec:smms/we}

We conclude this section with a brief discussion of weighted Einstein manifolds and their properties as needed in Theorem~\ref{thm:intro/stable_qe}, Theorem~\ref{thm:intro/obata}, Theorem~\ref{thm:intro/full}, and related discussions.

\begin{defn}
 A \emph{weighted Einstein manifold} is a smooth metric measure space $(M^n,g,v,m,\mu)$ such that $P_\phi^m=\lambda g$ for some $\lambda\in\bR$.
\end{defn}

The following lemma states that for every weighted Einstein manifold, there is a scale for which it has constant weighted $\sigma_1$-curvature (cf.\ \cite[Proposition~5]{Kim_Kim}).

\begin{lem}
 \label{lem:we_kappa}
 Let $(M^n,g,v,m,\mu)$ be such that $P_\phi^m=\lambda g$ for $\lambda\in\bR$.  Then there is a unique constant $\kappa\in\bR$ such that $\csigma_{1,\phi}^m=(m+n)\lambda$.
\end{lem}

\begin{defn}
 The \emph{scale} of a weighted Einstein manifold $(M^n,g,v,m,\mu)$ satisfying $P_\phi^m=\lambda g$ is the constant $\kappa\in\bR$ such that $\csigma_{1,\phi}^m=(m+n)\lambda$.
\end{defn}

\begin{proof}[Proof of Lemma~\ref{lem:we_kappa}]
 We must find a $\kappa\in\bR$ such that $J_\phi^m+m\kappa v^{-1}=(m+n)\lambda$.  Lemma~\ref{lem:div_and_tr} implies that
 \[ \delta_\phi\left(P_\phi^m-\lambda g\right) - \frac{1}{m}\tr\left(P_\phi^m-\lambda g\right)\,d\phi = v^{-1}d\left(\left(J_\phi^m-(m+n)\lambda\right)v\right), \]
 from which the conclusion readily follows.
\end{proof}

A special case of weighted Einstein manifolds already studied in the literature are quasi-Einstein manifolds~\cite{Case_Shu_Wei}, a class of manifolds which include static metrics in general relativity, the bases of Einstein warped product manifolds, and gradient Ricci solitons.

\begin{defn}
 \label{defn:qe}
 A \emph{quasi-Einstein manifold} is a weighted Einstein manifold with scale $\kappa=0$.
\end{defn}

\begin{remark}
 It is readily checked that if $(M^n,g,v,m,\mu)$ is a quasi-Einstein manifold in the sense of Definition~\ref{defn:qe}, then $\Ric_\phi^m=\frac{2(m+n-1)}{m+n-2}\lambda g$; i.e.\ $(M^n,g)$ is quasi-Einstein in the sense of~\cite{Case_Shu_Wei}.
 
 Conversely, if $(M^n,g)$ is such that $\Ric_\phi^m=\lambda g$ for some $\phi\in C^\infty(M)$ and some constants $\lambda\in\bR$ and $m\in\bR_+$, it is known~\cite{Kim_Kim} that there is a constant $\mu\in\bR$ such that $(M^n,g,e^{-\phi/m},m,\mu)$ satisfies $R_\phi^m=(m+n)\lambda$.  It follows that $(M^n,g,e^{-\phi/m},m,\mu)$ is quasi-Einstein in the sense of Definition~\ref{defn:qe}.
\end{remark}

We next discuss the positively-curved flat models of quasi-Einstein and weighted Einstein manifolds; i.e.\ the smooth metric measure spaces which are locally conformally flat in the weighted sense and are weighted Einstein manifolds with nonnegative scale and $\lambda>0$.  In what follows, we regard $S^n\subset\bR^{n+1}$ as the set
\[ S^n = \left\{ x \in \bR^{n+1} \suchthat \lv x\rv^2 = 1\right\} \]
and let $x_1,\dotsc,x_{n+1}$ denote both the standard coordinates on $\bR^{n+1}$ and their restriction to $S^{n+1}$.  The upper hemisphere $S_+^n$ is defined by
\[ S_+^n = \left\{ x\in S^n \suchthat x_{n+1}\geq0 \right\} . \]
Note that in both of the examples below the function $v$ is allowed to vanish on a set of measure zero.  Thus these spaces are not examples of smooth metric measure spaces on closed manifolds as defined in this article.

First we discuss the model spaces for quasi-Einstein manifolds.  Additional examples of quasi-Einstein manifolds are discussed, for example, in~\cite{Case2010a,HePetersenWylie2010}.

\begin{example}
 \label{ex:model_zero}
 Fix $n\in\bN$ and $m\in\bR$.  The \emph{positive elliptic $m$-Gaussian} is the quasi-Einstein manifold $(S_+^n,d\theta^2,\cos r,m,1)$, where $d\theta^2$ is the round metric of constant sectional curvature $1$ on the $n$-sphere and $r$ denotes the geodesic distance from $(0,\dotsc,0,1)$.  Indeed, using the well-known facts
 \begin{equation}
  \label{eqn:cosr_facts}
  \begin{split}
  \nabla^2\cos r & = -\cos r\,d\theta^2, \\
  \lv\nabla\cos r\rv^2 & = 1-\cos^2 r,
  \end{split}
 \end{equation}
 we readily compute that the positive elliptic $m$-Gaussian satisfies
 \begin{align*}
  P_\phi^m & = \frac{m+n-2}{2}g, \\
  J_\phi^m & = \frac{(m+n)(m+n-2)}{2} .
 \end{align*}
 
 Given $a\in\bR$ and $\xi\in\bR^{n+1}$ such that $\lv\xi\rv^2<1$ and $\xi_{n+1}=0$, denote by $u$ the function on $S_+^n$ given by
 \begin{equation}
  \label{eqn:uqe}
  u(\zeta) = \frac{a}{\sqrt{1-\lv\xi\rv^2}}\left(1+\xi\cdot\zeta\right) .
 \end{equation}
 Consider the metric-measure structure $(\hg,\hv)=(u^{-2}d\theta^2,u^{-1}\,\cos r)$.  Using ``hats'' to denote weighted invariants determined by $(\hg,\hv)$, it is straightforward to check that
 \begin{align*}
  \hP_\phi^m & = \frac{m+n-2}{2}a^2\hg, \\
  \hJ_\phi^m & = \frac{(m+n)(m+n-2)}{2}a^2 .
 \end{align*}
 Thus $(\hg,\hv)$ is a quasi-Einstein metric-measure structure which is pointwise conformally equivalent to the positive elliptic $m$-Gaussian.  Indeed, all such quasi-Einstein metrics are given by conformally rescaling by a function of the form~\eqref{eqn:uqe}; see Proposition~\ref{prop:wKe_rigidity} below.  Note that if we choose $\xi_{n+1}\not=0$, then $\hJ_\phi^m$ is not constant; more precisely, $(\hg,\hv)$ is a weighted Einstein metric-measure structure with nonzero scale.
\end{example}

Next we discuss the model spaces for weighted Einstein manifolds with positive scale.  Further examples of weighted Einstein manifolds will be discussed elsewhere.

\begin{example}
 \label{ex:model_positive}
 Fix $n\in\bN$ and $m\in\bR$.  The \emph{standard $m$-weighted $n$-sphere} is the weighted Einstein manifold $(S^n,d\theta^2,1+\cos r,m,0)$, where $d\theta^2$ denotes the round metric of constant sectional curvature $1$ and $r$ denotes the geodesic distance from $(0,\dotsc,0,1)$.  Indeed, using the facts~\eqref{eqn:cosr_facts}, we readily compute that the standard $m$-weighted $n$-sphere satisfies
 \begin{align*}
  P_\phi^m & = \frac{m+n-2}{2}, \\
  J_\phi^m + m(m+n-2)v^{-1} & = \frac{(m+n)(m+n-2)}{2} .
 \end{align*}
 
 Given $a\in\bR$ and $\xi\in\bR^{n+1}$ such that $\lv\xi\rv^2<1$, denote by $u$ the function on $S^n$ given by
 \[ u(\zeta) = \frac{a}{\sqrt{1-\lv\xi\rv}^2}\left(1+\xi\cdot\zeta\right) . \]
 Consider the metric-measure structure $(\hg,\hv)=(u^{-2}d\theta^2,u^{-1}(1+\cos r))$.  Using ``hats'' to denote weighted invariants determined by $(\hg,\hv)$, it is straightforward to check that
 \begin{align*}
  \hP_\phi^m & = \frac{m+n-2}{2}a^2\hg, \\
  \hJ_\phi^m + m(m+n-2)\left(\frac{a(1-\xi_{n+1})}{\sqrt{1-\lv\xi\rv^2}}\right)\hv^{-1} & = \frac{(m+n)(m+n-2)}{2}a^2 .
 \end{align*}
 Thus $(\hg,\hv)$ is a weighted Einstein metric-measure structure with positive scale which is pointwise conformally equivalent to the standard $m$-weighted $n$-sphere.  Indeed, all such weighted Einstein metrics are given by a function of this form; see Proposition~\ref{prop:wKe_rigidity} below.
\end{example}

\begin{remark}
 Note that the space of quasi-Einstein metrics in the weighted conformal class of the $n$-dimensional positive elliptic $m$-Gaussian is $(n+1)$-dimensional, while the space of weighted Einstein metrics in the weighted conformal class of the standard $m$-weighted $n$-sphere is $(n+2)$-dimensional.  This and the relation between weighted Einstein manifolds and sharp Gagliardo--Nirenberg inequalities~\cite{Case2011gns} provide strong evidence that weighted Einstein metrics are, from the point of view of conformal geometry, the more natural Einstein-type structure to study on smooth metric measure spaces.
\end{remark}

Lemma~\ref{lem:we_kappa} asserts that a weighted Einstein manifold $(M^n,g,v,m,\mu)$ with scale $\kappa$ has constant weighted $\sigma_1$-curvature.  Such manifolds also have constant weighted $\sigma_k$-curvatures and constant weighted Newton tensors.  However, the weighted Bach tensor is not necessarily constant, though it can be computed in terms of only the two-jets of $g$ and $v$.

\begin{prop}
 \label{prop:we_sigma2}
 Let $(M^n,g,v,m,\mu)$ be a weighted Einstein manifold with scale $\kappa\in\bR$ and such that $P_\phi^m=\lambda g$.  Let $k\in\bN_0$.  Then
 \begin{align}
  \label{eqn:we_csigmak} \csigma_{k,\phi}^m & = \binom{m+n}{k}\lambda^k, \\
  \label{eqn:we_csk} \cs_{k,\phi}^m & = \binom{m+n-1}{k}\lambda^k, \\
  \label{eqn:we_cTk} \cT_{k,\phi}^m & = \binom{m+n-1}{k}\lambda^kg, \\
  \label{eqn:we_bach} B_\phi^m & = m\kappa v^{-1}P_\phi^{m-1} - \frac{m(m+n-3)}{m+n-2}\lambda\kappa v^{-1}g,
 \end{align}
 where
 \begin{equation}
  \label{eqn:P-1}
  P_\phi^{m-1} = P_\phi^m + v^{-1}\nabla^2v + \frac{1}{m(m+n-2)}Y_\phi^m g
 \end{equation}
 is the weighted Schouten tensor of $(M^n,g,v,m-1,\mu)$.\end{prop}

\begin{proof}
 To begin, consider an arbitrary smooth metric measure space $(M^n,g,v,m,\mu)$.  Lemma~\ref{lem:eval_Yphim} implies that the weighted Schouten scalar $J_\phi^{m-1}$ of $(M^n,g,v,m-1,\mu)$ satisfies
 \begin{equation}
  \label{eqn:J-1}
  J_\phi^{m-1} = \frac{m+n-3}{m+n-2}\left(J_\phi^m - \frac{1}{m}Y_\phi^m\right) .
 \end{equation}
 This readily yields~\eqref{eqn:P-1}.

 Now suppose that $(M^n,g,v,m,\mu)$ is a weighted Einstein manifold with scale $\kappa$.  Since $\csigma_{1,\phi}^m=\tr P_\phi^m+\cY_\phi^m$, we see that
 \begin{equation}
  \label{eqn:we_cY}
  m\lambda = \cY_\phi^m = Y_\phi^m + m\kappa v^{-1} .
 \end{equation}
 This yields~\eqref{eqn:we_csigmak}, \eqref{eqn:we_csk}, and~\eqref{eqn:we_cTk}.    From the definition of the weighted Bach tensor, we see that $B_\phi^m=\kappa v^{-1}\tr A_\phi^m$.  Combining this observation with Lemma~\ref{lem:div_and_tr} and~\eqref{eqn:P-1} yields~\eqref{eqn:we_bach}.
 \end{proof}

In the special case of a closed weighted Einstein smooth metric measure space $(M^n,g,v,m,0)$, Proposition~\ref{prop:we_sigma2} yields a formula for $\int v^{-1}d\nu$ in terms of only $\lambda$, $\kappa$, and the weighted volume.

\begin{prop}
 \label{prop:we_integral_relation}
 Let $(M^n,g,v,m,0)$ be a closed weighted Einstein manifold with scale $\kappa\in\bR$ and such that $P_\phi^m=\lambda g$.  Then
 \begin{equation}
  \label{eqn:we_integral}
  \lambda\int_M d\nu = \frac{2m+n-2}{2(m+n-1)}\kappa\int_M v^{-1}\,d\nu .
 \end{equation}
\end{prop}

\begin{proof}
 Given any closed smooth metric measure space $(M^n,g,v,m,0)$ and any $\kappa\in\bR$, Lemma~\ref{lem:eval_Yphim} implies that
 \begin{equation}
  \label{eqn:we_integral_nocsigma}
  (2m+n-2)m\kappa\int_M v^{-1}\,d\nu = \int_M \left(m\csigma_{1,\phi}^m + (m+n-2)\cY_\phi^m\right)d\nu .
 \end{equation}
 The conclusion~\eqref{eqn:we_integral} then follows from~\eqref{eqn:we_csigmak} and~\eqref{eqn:we_cY}.
\end{proof}

%% file: moduli.tex
\section{The space of metric-measure structures}
\label{sec:moduli}

In this section we briefly discuss a natural formalism for studying the space of metric-measure structures on weighted manifolds.  Recall that the \emph{space of metric-measure structures} on $(M^n,m,\mu)$ is
\[ \kM(M,m,\mu) := \Met(M)\times C^\infty(M;\bR_+) , \]
where $\Met(M)$ is the space of Riemannian metrics on $M$ and $C^\infty(M;\bR_+)$ is the space of positive smooth functions on $M$.  When the weighted manifold $(M^n,m,\mu)$ is clear by context, we denote by $\kM$ its space of metric-measure structures.  Note that, as a set, $\kM(M,m,\mu)$ depends only on $M$.  The role of the parameters $m$ and $\mu$ is to determine the geometry of elements of $\kM$, especially in the definitions of weighted invariants.  The fact that $\kM(M,m,\mu)=\kM(M,m^\prime,\mu)$ as sets for any $m,m^\prime\in\bR$ provides a useful way to relate metric-measure structures for different values of $m$.

It is clear that a weighted conformal class $\kC$ on a weighted manifold is a subset of $\kM$; moreover, any $(g,v)\in\kM$ uniquely determines a weighted conformal class $\kC=[g,v]$.  Just as many geometric variational problems are most naturally posed under a volume constraint, we consider the sets
\begin{align*}
 \kM_1 & := \left\{ (g,v)\in\kM \suchthat \int_M d\nu(g,v) = 1 \right\} , \\
 \kC_1 & := \kM_1 \cap \kC 
\end{align*}
consisting of unit-volume metric-measure structures within $\kM$ and $\kC$, respectively.

A \emph{weighted invariant} is a function $I$ defined on $\kM(M,m,\mu)$ which is invariant with respect to the action of the diffeomorphism group $\Diff(M)$ of $M$.  One special class of weighted invariants consists of \emph{weighted functionals}, namely maps $\mS\colon\kM\to\bR$ such that $\mS(f^\ast g,f^\ast v)=\mS(g,v)$ for every $f\in\Diff(M)$ and every $(g,v)\in\kM$.  Another special class consists of \emph{weighted scalar invariants}, namely maps $I_\phi^m\colon\kM\to C^\infty(M)$ such that $I_\phi^m(f^\ast g,f^\ast v)=f^\ast\bigl(I_\phi^m(g,v)\bigr)$ for every $f\in\Diff(M)$ and every $(g,v)\in\kM$.  For example, the weighted volume element $d\nu$ is a volume-element-valued weighted invariant, the weighted $\sigma_k$-curvatures are weighted scalar invariants, and the total weighted $\sigma_k$-curvature functionals obtained by integrating the weighted $\sigma_k$-curvature with respect to $d\nu$ are weighted functionals.

Let $(g,v)\in\kM$ and consider the formal tangent space $T_{(g,v)}\kM$ consisting of derivatives $\gamma^\prime(0)$ of smooth paths $\gamma\colon\bR\to\kM$ such that $\gamma(0)=(g,v)$.  We denote by
\[ T\kM := \bigcup_{(g,v)\in\kM} T_{(g,v)}\kM \]
the formal tangent bundle of $\kM$.  We identify $T_{(g,v)}\kM\cong \Gamma\bigl(S^2T^\ast M\bigr)\oplus C^\infty(M)$ via the bijection
\[ \Psi_{(g,v)}^m\colon T_{(g,v)}\kM \to \Gamma\left(S^2T^\ast M\right)\oplus C^\infty(M) \]
defined as follows: Given $\gamma^\prime(0)\in T_{(g,v)}\kM$, set
\begin{subequations}
 \label{eqn:TkM_geometric}
 \begin{align}
  \label{eqn:TkM_geometric_h} h & = v^2\left.\frac{\partial}{\partial t}\right|_{t=0}\left( v_t^{-2}g_t\right), \\
  \label{eqn:TkM_geometric_psi} \psi & = -\left.\frac{\partial}{\partial t}\right|_{t=0}\frac{d\nu(g_t,v_t)}{d\nu(g,v)},
 \end{align}
\end{subequations}
for $(g_t,v_t):=\gamma(t)$.  Then set $\Psi_{(g,v)}^m\left(\gamma^\prime(0)\right):=(h,\psi)$.  The inverse is
\begin{multline}
 \label{eqn:TkM}
 \bigl(\Psi_{(g,v)}^m\bigr)^{-1}(h,\psi) = \gamma^\prime(0) \\ := \left( h-\frac{2}{m+n}\left(\psi+\frac{1}{2}\tr_gh\right)g, -\frac{1}{m+n}\left(\psi+\frac{1}{2}\tr_gh \right)v \right) .
\end{multline}
From~\eqref{eqn:TkM_geometric_h} we observe that $h=0$ if and only if $\bigl(\Psi_{(g,v)}^m\bigr)^{-1}(h,\psi)$ is tangent to a curve in the weighted conformal class $[g,v]$.  When the weighted manifold $(M^n,m,\mu)$ and metric-measure structure $(g,v)\in\kM$ are clear from context, we omit the function $\Psi_{(g,v)}^m$ and simply identify $T_{(g,v)}\kM$ with $\Gamma\left(S^2T^\ast M\right)\oplus C^\infty(M)$.

The splitting $S^2T^\ast M=S_0^2T^\ast M\oplus \bR g$ of $S^2T^\ast M\cong T_g\Met(M)$ into trace-free and pure trace parts is the Riemannian analogue of the decomposition~\eqref{eqn:TkM_geometric}, especially in that the $S_0^2T^\ast M$-component of a formal tangent vector $\gamma^\prime(0)\in S^2T^\ast M$ vanishes if and only if $\gamma^\prime(0)$ is tangent to a curve in the weighted conformal class of $\gamma(0)$.  This motivates us to define a $C^\infty(M)$-linear map $\tf_\phi\colon \Gamma\left(S^2T^\ast M\right)\to \Gamma\left(S^2T^\ast M\right)$ on any smooth metric measure space $(M^n,g,v,m,\mu)$ which produces the ``weighted trace-free component'' of a given tensor field.  While it is unclear how to define this map in general, it is clear on a case-by-case basis: We denote
\begin{align*}
 \tf_\phi g & := 0 , \\
 \tf_\phi \left(P_\phi^m\right)^k & := \left(P_\phi^m\right)^k - \frac{1}{m+n}N_{k,\phi}^mg, \\
 \tf_\phi \nabla^2u & := \nabla^2u - \frac{1}{m+n}\Delta_\phi u\,g, \\
 \tf_\phi B_\phi^m & := B_\phi^m
\end{align*}
for all $k\in\bN_0$ and all $u\in C^\infty(M)$.  These definitions are motivated by considering the trace-free part of the analogous Riemannian invariants in $(m+n)$-dimensions.

Formally, the first variation (or linearization) of a weighted invariant $I$ is the exterior derivative $DI$.  For example, if $I_\phi^m\colon\kM\to C^\infty(M)$ is a weighted scalar invariant, its first variation is the map $DI_\phi^m\colon T\kM\to C^\infty(M)$ defined by
\[ DI_\phi^m[h,\psi] := \left.\frac{\partial}{\partial t}\right|_{t=0}I_\phi^m\left(\gamma(t)\right) \]
for all $(h,\psi)\in T_{(g,v)}\kM$ and all $(g,v)\in\kM$, where $\gamma\colon\bR\to\kM$ is a smooth path such that $\gamma(0)=(g,v)$ and $\gamma^\prime(0)=\bigl(\Psi_{(g,v)}^m\bigr)^{-1}(h,\psi)$.  A special case of interest is the linearization of the restriction of a weighted scalar invariant $I_\phi^m$ to a weighted conformal class $\kC$, which we regard as a map $DI_\phi^m\colon T\kC\to C^\infty(M)$.  Our convention~\eqref{eqn:TkM_geometric} is such that for any $(g,v)\in\kC$, we may identify $T_{(g,v)}\kC\cong C^\infty(M)$ by
\[ T_{(g,v)}\kC = \left(\Psi_{(g,v)}^m\right)^{-1}\left\{ (0,\psi) \suchthat \psi\in C^\infty(M) \right\} . \]
Equivalently, we identify a function $\psi\in C^\infty(M)$ with the tangent vector to the curve
\[ \gamma(t) = \left(e^{-\frac{2t\psi}{m+n}}g,e^{-\frac{t\psi}{m+n}}v\right) \]
at $t=0$.  Note also that the identification $T_{(g,v)}\kC\cong C^\infty(M)$ gives rise to the identification
\begin{equation}
 \label{eqn:TkC1_identification}
 T_{(g,v)}\kC_1 \cong \left\{ \psi\in C^\infty(M) \suchthat \int_M \psi\,d\nu = 0 \right\} .
\end{equation}
When $(M^n,m,\mu)$ and $(g,v)\in\kC$ are clear by context, we simply identify $T_{(g,v)}\kC$ with $C^\infty(M)$ and we identify $T_{(g,v)}\kC_1$ with mean-free elements of $C^\infty(M)$.

We say that $DI_\phi^m$ is \emph{formally self-adjoint} if for each $(g,v)\in\kC$, the operator $DI_\phi^m\colon C^\infty(M)\to C^\infty(M)$ is formally self-adjoint with respect to the natural $L^2$-inner product on $C^\infty(M)$ induced by $d\nu(g,v)$.

Let $\kC$ be a weighted conformal class on a weighted manifold $(M^n,m,\mu)$.  A weighted scalar invariant $I_\phi^m$ is \emph{conformally variational (on $\kC$)} if there is a weighted functional $\mS$ such that
\begin{equation}
 \label{eqn:conf_var_el}
 D\mS[\psi] = \int_M I_\phi^m\psi\,d\nu
\end{equation}
for all $\psi\in T_{(g,v)}\kC$ and all $(g,v)\in\kC$.  This property is equivalent to the self-adjointness of $DI_\phi^m$ (cf.\ \cite{Bourguignon1985,BransonGover2008}).

\begin{lem}
 \label{lem:var_to_sa}
 Let $\kC$ be a weighted conformal class on $(M^n,m,\mu)$.  A weighted scalar invariant $I_\phi^m$ is conformally variational if and only if its linearization $DI_\phi^m$ is formally self-adjoint.  Moreover, if $I_\phi^m$ is conformally variational, then
 \begin{equation}
  \label{eqn:var_to_sa_comp}
  D\left(\int_M I_\phi^m\,d\nu\right)[\psi] = -\int_M \left( I_\phi^m - DI_\phi^m[1]\right)\psi\,d\nu
 \end{equation}
 for all $\psi\in T_{(g,v)}\kC$ and all $(g,v)\in\kC$.
\end{lem}

\begin{proof}
 Consider the one-form $\Omega\colon T\kM\to\bR$ defined by
 \[ \Omega[\psi] := \int_M I_\phi^m\psi\,d\nu \]
 for all $\psi\in T_{(g,v)}\kC$ and all $(g,v)\in\kC$.  By definition, $I_\phi^m$ is conformally variational if and only if $\Omega$ is exact.  Since $\kC$ is contractible, $\Omega$ is exact if and only if it is closed.  A straightforward computation shows that the formal exterior derivative $D\Omega\colon \Lambda^2T_{(g,v)}\kC\to\bR$ of $\Omega$ is
 \[ D\Omega[\psi_1,\psi_2] = \int_M \left( \psi_2 DI_\phi^m[\psi_1] - \psi_1 DI_\phi^m[\psi_2]\right)\, d\nu \]
 for all $\psi_1,\psi_2\in T_{(g,v)}\kC$ and all $(g,v)\in\kC$.  Thus $\Omega$ is closed if and only if $DI_\phi^m$ is formally self-adjoint, yielding the first assertion.  The second assertion follows immediately from the self-adjointness of $DI_\phi^m$ and~\eqref{eqn:TkM_geometric_psi}.
\end{proof}

If we restrict our attention to conformally variational weighted scalar invariants $I_\phi^m$ which are homogeneous with respect to homotheties in $\kC$, then $DI_\phi^m[1]$ is a constant multiple of $I_\phi^m$, and hence~\eqref{eqn:var_to_sa_comp} generically identifies the functional $\mS$ in the definition~\eqref{eqn:conf_var_el} of the conformally variational property.  For this reason, one frequently restricts attention to homogeneous invariants in the Riemannian setting (cf.\ \cite{BransonGover2008}).  However, due to our goal of producing weighted functionals which include weighted Einstein manifolds among their critical points, we do not impose this homogeneity requirement; see Section~\ref{sec:full} for a detailed discussion.

When computing the linearizations of the total weighted scalar curvature functionals, it is useful to take advantage of the fact that $\kM(M,m,\mu)=\kM(M,m^\prime,\mu)$ as sets for all $m,m^\prime\in\bR$.  Hence $T\kM(M,m,\mu)=T\kM(M,m^\prime,\mu)$ as sets for all $m,m^\prime\in\bR$.  However, the identification $T_{(g,v)}\kM(M,m,\mu)\cong\Gamma\left(S^2T^\ast M\right)\oplus C^\infty(M)$ via the function $\Psi_{(g,v)}^m$ depends on $m$.  These observations allow us to relate the maps $\Psi_{(g,v)}^m$ and $\Psi_{(g,v)}^{m^\prime}$.

\begin{lem}
 \label{lem:change_TkM}
 Fix $m,m^\prime\in\bR_+$ and $\mu\in\bR$.  Let $M^n$ be a smooth manifold.  Given a metric-measure structure $(g,v)\in\kM(M,m,\mu)=\kM(M,m^\prime,\mu)$, define
 \[ \Phi_m^{m^\prime}\colon\Gamma(S^2T^\ast M)\oplus C^\infty(M) \to \Gamma\left(S^2T^\ast M\right)\oplus C^\infty(M) \]
 by
 \[ \Phi_{m}^{m^\prime}\left(h,\psi\right) = \left( h, \frac{m^\prime+n}{m+n}\psi - \frac{m-m^\prime}{2(m+n)}\tr_g h\right) . \]
 Then $\bigl(\Psi_{(g,v)}^{m^\prime}\bigr)^{-1}\circ\Phi_{m}^{m^\prime}\circ\Psi_{(g,v)}^m\colon T_{(g,v)}\kM(M,m,\mu)\to T_{(g,v)}\kM(M,m^\prime,\mu)$ is the identity map.
\end{lem}

\begin{proof}
 Let $\gamma\colon\bR\to\kM$ be a smooth curve such that $\gamma(0)=(g,v)$.  Let $d\nu^{(m)}$ and $d\nu^{(m^\prime)}$ denote the weighted volume elements on $\kM(M,m,\mu)$ and $\kM(M,m^\prime,\mu)$, respectively.  A straightforward computation yields
 \[ \frac{d\nu^{(m^\prime)}\left(\gamma(t)\right)}{d\nu^{(m^\prime)}\left(\gamma(0)\right)} = \left(\frac{d\nu^{(m)}\left(\gamma(t)\right)}{d\nu^{(m)}\left(\gamma(0)\right)}\right)^{\frac{m^\prime+n}{m+n}}\left(\det \left((v^{-2}g)^{-1}v_t^{-2}g_t\right)\right)^{\frac{m-m^\prime}{2(m+n)}} . \]
 for $(g_t,v_t)=\gamma(t)$.  The conclusion readily follows from~\eqref{eqn:TkM_geometric}.
\end{proof}

Lemma~\ref{lem:change_TkM} can be reformulated as a statement about linearizations of weighted invariants:

\begin{cor}
 \label{cor:change_TkM}
 Fix $m,m^\prime\in\bR_+$ and $\mu\in\bR$.  Let $M^n$ be a smooth manifold.  Given a metric-measure structure $(g,v)\in\kM(M,m,\mu)=\kM(M,m^\prime,\mu)$ and a weighted invariant $I$, denote by $D^{(m)}I$ and $D^{(m^\prime)}I$ the linearizations of $I$ when regarded as functions of
 \begin{align*}
  \Gamma\left(S^2T^\ast M\right)\oplus C^\infty(M) & \cong T_{(g,v)}\kM(M,m,\mu), \\
  \Gamma\left(S^2T^\ast M\right)\oplus C^\infty(M) & \cong T_{(g,v)}\kM(M,m^\prime,\mu),
 \end{align*}
 respectively.  Then
 \[ D^{(m)}I = D^{(m^\prime)}I\circ\Phi_m^{m^\prime} . \]
\end{cor}

Since $\kC(M,m,\mu)=\kC(M,m^\prime,\mu)$ as sets, Corollary~\ref{cor:change_TkM} also applies to linearizations of weighted invariants within weighted conformal classes.  This observation is also reflected in the fact that $\Phi_m^{m^\prime}$ acts as the identity in its first component.

Corollary~\ref{cor:change_TkM} is quite useful for computing the linearizations of the weighted $\sigma_k$-curvatures for arbitrary scales $\kappa$.  The reason for this is that it is straightforward to compute the linearizations of the weighted $\sigma_k$-curvatures with scale $\kappa=0$ (cf.\ Section~\ref{sec:full}), while one can also exhibit the weighted $\sigma_k$-curvatures for arbitrary scales as perturbations of the weighted $\sigma_k$-curvatures with scale $\kappa$ through weighted $\sigma_k$-curvatures of lower degree when computed with respect to different values of $m$.  We separate this latter observation into two results so as to highlight the role of weighted conformal classes which are locally conformally flat in the weighted sense for larger values of $k$.

\begin{lem}
 \label{lem:csigma_to_sigma}
 Let $(M^n,m,\mu)$ be a weighted manifold and fix $\kappa\in\bR$.  Then
 \begin{align}
  \label{eqn:csigma1_to_sigma1} \csigma_{1,\phi}^m & = \sigma_{1,\phi}^m + m\kappa v^{-1}, \\
  \label{eqn:csigma2_to_sigma2} \csigma_{2,\phi}^m & = \sigma_{2,\phi}^m + \frac{m(m+n-2)}{m+n-3}\kappa v^{-1}\sigma_{1,\phi}^{m-1} + \frac{m(m-1)}{2}\kappa^2v^{-2}
 \end{align}
 for all $(g,v)\in\kM$.
\end{lem}

\begin{proof}
 \eqref{eqn:csigma1_to_sigma1} follows immediately from Corollary~\ref{cor:reln_to_genl}.  Applying Corollary~\ref{cor:reln_to_genl} and Lemma~\ref{lem:s_to_sigma-1} implies that
 \begin{equation}
  \label{eqn:csigma_expand2}
  \csigma_{2,\phi}^m = \sigma_{2,\phi}^m + m\kappa v^{-1}s_{1,\phi}^m + \binom{m}{2}\kappa^2v^{-2} .
 \end{equation}
 On the other hand, \eqref{eqn:J-1} implies that
 \begin{equation}
  \label{eqn:sigma-m-1-to-s}
  s_{1,\phi}^m = \frac{m+n-2}{m+n-3}\sigma_{1,\phi}^{m-1} .
 \end{equation}
 Inserting~\eqref{eqn:sigma-m-1-to-s} into~\eqref{eqn:csigma_expand2} yields~\eqref{eqn:csigma2_to_sigma2}.
\end{proof}

\begin{lem}
 \label{lem:csigma_to_sigma_genl}
 Let $(M^n,m,\mu)$ be a weighted manifold and fix a scale $\kappa\in\bR$.  Let $\kC\subset\kM$ be a weighted conformal class which is locally conformally flat in the weighted sense.  Given an integer $k\leq m$, it holds that
 \[ \csigma_{k,\phi}^m = \sum_{j=0}^k \left(\frac{m+n-2}{m+n-2-j}\right)^{k-j}\binom{m}{j}\left(\kappa v^{-1}\right)^j\sigma_{k-j,\phi}^{m-j} \]
 for all $(g,v)\in\kC$.
\end{lem}

\begin{proof}
 Note that Lemma~\ref{lem:s_to_sigma-1} and~\eqref{eqn:sigma-m-1-to-s} together imply that
 \begin{equation}
  \label{eqn:preY-k}
  \sigma_1^{m-1}\left(\frac{m-1}{m}Y_\phi^m;P_\phi^m\right) = \frac{m+n-2}{m+n-3}\sigma_1^{m-1}\left(Y_\phi^{m-1};P_\phi^{m-1}\right) .
 \end{equation}
 for all $(g,v)\in\kM$.  On the other hand, since $\kC$ is locally conformally flat in the weighted sense, Lemma~\ref{lem:div_and_tr} and~\eqref{eqn:P-1} imply that
 \[ P_\phi^m = \frac{m+n-2}{m+n-3}P_\phi^{m-1} . \]
 It follows that
 \[ A_\phi^{m-1} := \Rm - \frac{1}{m+n-3}P_\phi^{m-1}\wedge g = 0 ; \]
 i.e.\ $(M^n,g,v,m-1,\mu)$ is locally conformally flat in the weighted sense.  Therefore
 \begin{equation}
  \label{eqn:P-k_lcf}
  P_\phi^m = \frac{m+n-2}{m+n-2-k}P_\phi^{m-k}
 \end{equation}
 for all integers $k\leq m$.  Combining this with~\eqref{eqn:preY-k} and an obvious induction argument yields
 \begin{equation}
  \label{eqn:Y-k_lcf}
  Z_\phi^{m} = \frac{m+n-2}{m+n-2-k}Z_\phi^{m-k}
 \end{equation}
 for all integers $k\leq m$.  Inserting~\eqref{eqn:P-k_lcf} and~\eqref{eqn:Y-k_lcf} into Corollary~\ref{cor:reln_to_genl} yields the desired result.
\end{proof}

%% file: variational.tex
\section{Variational status of the weighted $\sigma_k$-curvatures}
\label{sec:variational}

By Lemma~\ref{lem:var_to_sa}, answering the question of when the weighted $\sigma_k$-curvatures are conformally variational can be achieved by characterizing when the linearizations $D\csigma_{k,\phi}^m$ are formally self-adjoint.  To that end, we compute $D\csigma_{k,\phi}^m$.

\begin{prop}
 \label{prop:linearization_sigmak}
 Let $(M^n,g,v,m,\mu)$ be a smooth metric measure space and fix $\kappa\in\bR$ and $k\in\bN_0$.  Set $\kC=[g,v]$.  The linearization $D\csigma_{k,\phi}^m\colon T_{(g,v)}\kC\to C^\infty(M)$ is
 \begin{align*}
  D\csigma_{k,\phi}^m[\psi] & = \frac{1}{m+n}\left(2k\csigma_{k,\phi}^m - m\kappa v^{-1}\cs_{k-1,\phi}^m\right)\psi + \frac{m+n-2}{m+n}\delta_\phi\left(\cT_{k-1,\phi}^m(\nabla\psi)\right) \\
   & \quad - \frac{m+n-2}{m+n}\sum_{\ell=0}^{k-3}(-1)^\ell\cT_{k-3-\ell,\phi}^m\left(\nabla\psi,dP_\phi^m\cdot\cQ_{\ell+1,\phi}^m\right) ,
 \end{align*}
 where $\cQ_{\ell,\phi}^m$ and $\cdot$ are as in Proposition~\ref{prop:div_T}.
\end{prop}

\begin{proof}
 By Proposition~\ref{prop:div_T}, it suffices to prove that
 \begin{multline}
  \label{eqn:linearization_sigmak_equiv}
  D\csigma_{k,\phi}^m[\psi] = \frac{1}{m+n}\left(2k\csigma_{k,\phi}^m-m\kappa v^{-1}\cs_{k-1,\phi}^m\right)\psi \\ + \frac{m+n-2}{m+n}\left(\lp \cT_{k-1,\phi}^m,\nabla^2\psi\rp - \lp \cs_{k-1,\phi}^m\nabla\phi,\nabla\psi \rp\right)
 \end{multline}
 for all $k\in\bN_0$.  We prove~\eqref{eqn:linearization_sigmak_equiv} by induction.

 Clearly~\eqref{eqn:linearization_sigmak_equiv} holds when $k=0$.  Suppose that~\eqref{eqn:linearization_sigmak_equiv} holds for some $k\in\bN_0$.  From~\eqref{eqn:defn_csigma_expanded}, we note that
 \begin{equation}
  \label{eqn:linearization_leibnitz}
  (k+1)D\csigma_{k+1,\phi}^m = \sum_{j=0}^k (-1)^j\left[ \cN_{j+1,\phi}^mD\csigma_{k-j,\phi}^m + \csigma_{k-j,\phi}^mD\cN_{j+1,\phi}^m \right] .
 \end{equation}
 Using Lemma~\ref{lem:conf_change}, we observe that
 \begin{align*}
  DP_\phi^m[\psi] & = \frac{m+n-2}{m+n}\nabla^2\psi, \\
  D\cZ_\phi^m[\psi] & = \frac{1}{m+n}\left(2\cZ_\phi^m - \kappa v^{-1}\right)\psi - \frac{m+n-2}{m(m+n)}\lp \nabla\phi,\nabla\psi \rp .
 \end{align*}
 In particular,
 \begin{multline}
  \label{eqn:linearization_sigmak_step1}
  D\cN_{k,\phi}^m[\psi] = \frac{k}{m+n}\left(2\cN_{k,\phi}^m - m\kappa v^{-1}\bigl(\cZ_\phi^m\bigr)^{k-1}\right)\psi \\ + \frac{(m+n-2)k}{m+n}\left(\left\lp\left(P_\phi^m\right)^{k-1},\nabla^2\psi\right\rp - \left\lp \bigl(\cZ_\phi^m\bigr)^{k-1}\nabla\phi,\nabla\psi\right\rp \right) .
 \end{multline}
 Inserting~\eqref{eqn:linearization_sigmak_step1} and the inductive hypothesis~\eqref{eqn:linearization_sigmak_equiv} into~\eqref{eqn:linearization_leibnitz} yields
 \begin{multline*}
  D\csigma_{k+1,\phi}^m[\psi] = \frac{1}{m+n}\left(2(k+1)\csigma_{k+1,\phi}^m-m\kappa v^{-1}\cs_{k,\phi}^m\right)\psi \\ + \frac{m+n-2}{m+n}\left(\lp \cT_{k,\phi}^m,\nabla^2\psi\rp - \lp \cs_{k,\phi}^m\nabla\phi,\nabla\psi \rp \right) ,
 \end{multline*}
 as desired.
\end{proof}

Our study of the formal self-adjointness of $D\csigma_{k,\phi}^m$ follows the Riemannian analogue carried about by Branson and Gover~\cite{BransonGover2008}.

\begin{thm}
 \label{thm:variational}
 Let $\kC$ be a weighted conformal class on $(M^n,m,\mu)$.  Fix $\kappa\in\bR$ and $k\in\bN_0$.  If $k\leq 2$ or $\kC$ is locally conformally flat in the weighted sense, then the weighted $\sigma_k$-curvature is conformally variational.  If additionally $k\leq m+n$, then the converse holds.
\end{thm}

\begin{proof}
 Denote
 \[ \cS_{k,\phi}^m := \sum_{\ell=0}^{k-3}(-1)^\ell\cT_{k-3-\ell,\phi}^m\left(dP_\phi^m\cdot\cQ_{\ell+1,\phi}^m\right) . \]
 From Lemma~\ref{lem:var_to_sa} and Proposition~\ref{prop:linearization_sigmak}, we see that $\csigma_{k,\phi}^m$ is conformally variational if and only if
 \begin{equation}
  \label{eqn:cS_selfadjoint_test}
  \int_M \left\lp \cS_{k,\phi}^m, \omega\nabla\eta - \eta\nabla\omega \right\rp\,d\nu = 0
 \end{equation}
 for all $\eta,\omega\in C_0^\infty(M)$ and all representatives of $\kC$.  Clearly $\cS_{k,\phi}^m=0$ is sufficient for~\eqref{eqn:cS_selfadjoint_test} to hold.  Suppose instead that~\eqref{eqn:cS_selfadjoint_test} holds.  Taking $\eta=1$ and $\omega\in C_0^\infty(M)$ yields $\delta_\phi\cS_{k,\phi}^m=0$.  Hence $\cS_{k,\phi}^m$ is orthogonal to $\omega\nabla\eta$ in $L^2(d\nu)$ for all $\omega,\eta\in C_0^\infty(M)$.  This implies that $\cS_{k,\phi}^m=0$.  We conclude that $\csigma_{k,\phi}^m$ is conformally variational if and only if $\cS_{k,\phi}^m=0$ for all representatives of $\kC$.
 
 Clearly $\cS_{k,\phi}^m=0$ if $k\leq2$ or $\kC$ is locally conformally flat in the weighted sense.  We show that if $3\leq k\leq m+n$ and $\cS_{k,\phi}^m=0$ for all $(g,v)\in\kC$, then $\kC$ is locally conformally flat in the weighted sense.  Fix a representative $(g,v)\in\kC$, a point $p\in M$, a vector $X\in T_pM$, and a tensor $\Omega\in S^2T_p^\ast M$.  Let $f\in C^\infty(M)$ be such that $f(p)=0$, $(\nabla f)_p=X$, and $(\nabla^2f)_p=\Omega$.  Set $\left(\hg,\hv\right)=\left(e^{-\frac{2f}{m+n-2}}g,e^{-\frac{f}{m+n-2}}v\right)$.  By Lemma~\ref{lem:conf_change}, at $p$ it holds that
 \begin{equation}
  \label{eqn:variational_conf_change}
  \begin{split}
   \bigl(P_\phi^m\bigr)_{\hg} & = P_\phi^m + \Omega + \frac{1}{m+n-2}X^\flat\otimes X^\flat - \frac{1}{2(m+n-2)}\lv X\rv^2g , \\
   \bigl(\cZ_\phi^m\bigr)_{\hg} & = \cZ_\phi^m - \frac{1}{m}\lp X,\nabla\phi\rp - \frac{1}{2(m+n-2)}\lv X\rv^2 , \\
   \bigl(dP_\phi^m\bigr)_{\hg} & = dP_\phi^m - A_\phi^m(\cdot,\cdot,X,\cdot),
  \end{split}
 \end{equation}
 where the left-hand side (resp.\ right-hand side) of each equality is evaluated in terms of $(M^n,\hg,\hv,m,\mu)$ (resp.\ $(M^n,g,v,m,\mu)$).  We shall use~\eqref{eqn:variational_conf_change} with two different choices of $\Omega$ and $X$.

 First, choose $X_0\in T_pM$ such that $\lv X_0\rv^2=2(m+n-2)$ and $\lp X_0,\nabla\phi\rp=0$.  Let $t\in\bR$, set $X=tX_0$, and choose $\Omega=-\frac{1}{m+n-2}X^\flat\otimes X^\flat$.  It follows from~\eqref{eqn:variational_conf_change} that $(\cQ_{\ell+1,\phi}^m)_{\hg}$, $(dP_\phi^m)_{\hg}$ and $(\cT_{\ell,\phi})_{\hg}$ are polynomial in $t$ for all $\ell\in\bN$; indeed,
 \begin{align*}
  (dP_\phi^m)_{\hg} & = -tA_\phi^m(\cdot,\cdot,X_0,\cdot) + dP_\phi^m, \\
  (\cQ_{\ell+1,\phi}^m)_{\hg} & = (-1)^\ell(\ell+1)t^{2\ell} \cQ_{1,\phi}^m + O(t^{2\ell-2}), \\
  (\cT_{\ell,\phi})_{\hg} & = (-1)^\ell\binom{m+n-1}{\ell} t^{2\ell}g + O(t^{2\ell-2}) .
 \end{align*}
 Using the identity $\sum_{\ell=0}^k (-1)^\ell(\ell+1)\binom{n}{k-\ell}=\binom{n-2}{k}$, we compute that
 \[ \bigl(\cS_{k,\phi}^m\bigr)_{\hg} = (-1)^{k-3}\binom{m+n-3}{k-3}t^{2k-6}\left(dP_\phi^m-tA_\phi^m(\cdot,\cdot,X_0,\cdot)\right)\cdot\cQ_{1,\phi}^m + O(t^{2k-1}) . \]
 Since $t$ is arbitrary, $dP_\phi^m\cdot(P_\phi^m-\cZ_\phi^mg)=0$.  Since the representative $(g,v)$ is arbitrary, this holds for all representatives of $\kC$.

 Second, choose $X=0$ and let $\Omega$ be arbitrary.  From~\eqref{eqn:variational_conf_change} we see that, at $p$,
 \[ \left(dP_\phi^m\cdot\bigl(P_\phi^m-\cZ_\phi^mg\bigr)\right)_{\hg} = dP_\phi^m\cdot\bigl(P_\phi^m-\cZ_\phi^mg\bigr) + dP_\phi^m\cdot \Omega . \]
 The conclusion of the previous paragraph implies that $dP_\phi^m\cdot \Omega=0$ for all $(g,v)\in\kC$ and all $\Omega\in S^2T^\ast M$.  Hence $dP_\phi^m=0$ for all representatives $(g,v)\in\kC$.  Lemma~\ref{lem:conf_change} then implies that $A_\phi^m=0$, as desired.
\end{proof}

\subsection{The weighted $\sigma_k$-curvature functionals}
\label{subsec:variational/mY}

Proposition~\ref{prop:linearization_sigmak} enables us to compute the linearization of the total weighted $\sigma_k$-curvature functionals.  Recalling our goal that weighted Einstein metrics be stable with respect to these functionals, some care is needed in defining them.  In the case of scale zero, the definition of the total weighted $\sigma_k$-curvature functionals is the expected one.

\begin{defn}
 Let $(M^n,m,\mu)$ be a closed weighted manifold.  Given $k\in\bN$, the \emph{$\mF_k$-functional} $\mF_k\colon\kM\to\bR$ is defined by
 \[ \mF_k(g,v) := \int_M \sigma_{k,\phi}^m\,d\nu . \]
\end{defn}

From Proposition~\ref{prop:linearization_sigmak} we deduce that for generic values of $m$, the critical points of the restriction of the $\mF_k$-functionals to $\kC_1$ have constant weighted $\sigma_k$-curvature.

\begin{prop}
 \label{prop:mF_critical}
 Fix $k\in\bN_0$ and let $\kC$ be a weighted conformal class on a closed weighted manifold $(M^n,m,\mu)$; if $k\geq3$, assume that $\kC$ is locally conformally flat in the weighted sense.  Then the first variation $D\mF_k\colon T\kC\to\bR$ of the $\mF_k$-functional is
 \begin{equation}
  \label{eqn:mF_critical}
  D\mF_k[\psi] = -\frac{m+n-2k}{m+n}\int_M \sigma_{k,\phi}^m\psi\,d\nu
 \end{equation}
 for all $\psi\in T_{(g,v)}\kC$ and all $(g,v)\in\kC$.  In particular, $(g,v)\in\kC_1$ is a critical point of the restriction $\mF_k\colon\kC_1\to\bR$ if and only if $(M^n,g,v,m,\mu)$ is such that $\sigma_{k,\phi}^m$ is constant.
\end{prop}

\begin{proof}
 Proposition~\ref{prop:linearization_sigmak} and the proof of Theorem~\ref{thm:variational} imply that
 \[ D\sigma_{k,\phi}^m[\psi] = \frac{2k}{m+n}\sigma_{k,\phi}^m\psi + \frac{m+n-2}{m+n}\delta_\phi\left(T_{k-1,\phi}^m(\nabla\psi)\right) . \]
 The conclusion~\eqref{eqn:mF_critical} follows from this and~\eqref{eqn:var_to_sa_comp}.  The characterization of the critical points of $\mF_k\colon\kC_1\to\bR$ follows from~\eqref{eqn:TkC1_identification} and~\eqref{eqn:mF_critical}.
\end{proof}

In the case of positive scale, the total weighted $\sigma_k$-curvature functionals are the $\mY_k$-functionals defined below.

\begin{defn}
 Let $(M^n,m,0)$ be a closed weighted manifold.  Given $k\in\bN$, the \emph{$\mZ_k$-functional} $\mZ_k\colon\kM\times\bR_+\to\bR$ and the \emph{$\mY_k$-functional} $\mY_k\colon\kM\times\bR_+\to\bR$ are
 \begin{align*}
  \mZ_k(g,v,\kappa) & := \kappa^{-\frac{2mk(m+n-1)}{(m+n)(2m+n-2)}}\int_M \csigma_{k,\phi}^m\,d\nu, \\
  \mY_k(g,v,\kappa) & := \mZ_k(g,v,\kappa)\left(\int_M v^{-1}\,d\nu\right)^{-\frac{2mk}{(m+n)(2m+n-2)}}\left(\int_M d\nu\right)^{-\frac{m+n-2k}{m+n}}
 \end{align*}
 for all $(g,v)\in\kM$ and all $\kappa\in\bR_+$.
\end{defn}

The $\mY_k$-functionals are invariant with respect to both the natural homothetic rescalings in a weighted conformal class and the homothetic rescalings within a Riemannian conformal class.

\begin{lem}
 \label{lem:mY_scaling}
 Let $(M^n,m,0)$ be a closed weighted manifold and let $k\in\bN$.  Then
 \begin{align*}
  \mY_k(c^2g,cv,c^{-1}\kappa) & = \mY_k(g,v,\kappa), \\
  \mY_k(c^2g,v,c^{-2}\kappa) & = \mY_k(g,v,\kappa)
 \end{align*}
 for all $(g,v)\in\kM$ and all $c,\kappa\in\bR_+$.
\end{lem}

\begin{proof}
 Observe that
 \begin{align*}
  \csigma_{k,\phi}^m(c^2g,cv,c^{-1}\kappa) & = c^{-2k}\csigma_{k,\phi}^m(c^2g,cv,c^{-1}\kappa), & d\nu(c^2g,cv,c^{-1}\kappa) & = c^{m+n}d\nu(g,v,\kappa), \\
  \csigma_{k,\phi}^m(c^2g,v,c^{-2}\kappa) & = c^{-2k}\csigma_{k,\phi}^m(g,v,\kappa), & d\nu(c^2g,v,c^{-2}\kappa) & = c^nd\nu(g,v,\kappa) .
 \end{align*}
 The conclusion readily follows.
\end{proof}

In the remainder of this section, we begin to explore the variational properties of the $\mY_k$-functionals within a weighted conformal class $\kC$.  Our interest is in the cases when the weighted $\sigma_k$-curvatures are variational, hence we assume that $\kC$ is locally conformally flat in the weighted sense if $k\geq3$.  The scale-invariance of the $\mY_k$-functionals implies that it suffices to consider the \emph{$\mY_k$-functionals $\mY_k\colon\kC\to\bR$ with scale $\kappa>0$} defined by $\mY_k(g,v):=\mY_k(g,v,\kappa)$.

\begin{remark}
 The scale-invariance of the $\mY_k$-functionals implies that minimizers of $\mY_k\colon\kC\times\bR_+\to\bR$ are in one-to-one correspondence with minimizers of the functionals $\cmY_k\colon\kC\to\bR$ defined by
 \[ \cmY_k(g,v) = \inf_{\kappa>0}\mY_k(g,v,\kappa) . \]
 Up to composition with a monotone function depending only on $m$ and $n$, $\cmY_1$ is equivalent to the functional $\mQ_1$ introduced by the author~\cite{Case2013y} to study the weighted scalar curvature.
\end{remark}

In order to compute the linearizations of the $\mY_k$-functionals, we first consider the linearizations of the $\mZ_k$-functionals through variations of the scale $\kappa$.

\begin{lem}
 \label{lem:mZ_vary_kappa}
 Let $(M^n,g,v,m,0)$ be a closed smooth metric measure space, fix $k\in\bN$, and define $Z\colon\bR_+\to\bR$ by $Z(\kappa):=\mZ_k(g,v,\kappa)$.  Then
 \begin{equation}
  \label{eqn:mZ_vary_kappa}
  \kappa Z^\prime(\kappa) = A(\kappa)\int_M \left[ \csigma_{k,\phi}^m - \frac{(m+n)(2m+n-2)}{2k(m+n-1)}\kappa v^{-1}\cs_{k-1,\phi}^m \right] d\nu ,
 \end{equation}
 where $A(\kappa)=-\frac{2mk(m+n-1)}{(m+n)(2m+n-2)}\kappa^{-\frac{2mk(m+n-1)}{(m+n)(2m+n-2)}}$.  In particular, if $(M^n,g,v,m,0)$ is a weighted Einstein manifold with scale $\kappa>0$, then $Z^\prime(\kappa)=0$.
\end{lem}

\begin{proof}
 \eqref{eqn:mZ_vary_kappa} follows immediately from Lemma~\ref{lem:vary_kappa}.  If $(M^n,g,v,m,0)$ is a weighted Einstein manifold with scale $\kappa$, applying Proposition~\ref{prop:we_sigma2} and Proposition~\ref{prop:we_integral_relation} yields $Z^\prime(\kappa)=0$.
\end{proof}

\begin{remark}
 \label{rk:Zp}
 Let $p\in\bR$ and let $(M^n,g,v,m,0)$ be a weighted Einstein manifold with scale $\kappa>0$ and nonvanishing weighted Schouten tensor.  The same argument shows that the function $Z_p\colon\bR_+\to\bR$ defined by $Z_p(\kappa):=\kappa^p\int\csigma_{k,\phi}^m\,d\nu$ satisfies $Z_p^\prime(\kappa)=0$ if and only if $p=-\frac{2mk(m+n-1)}{(m+n)(2m+n-2)}$.
\end{remark}

We now compute the linearizations of the $\mY_k$-functionals.  Lemma~\ref{lem:mY_scaling} implies that~\eqref{eqn:mZ_vary_kappa} is proportional to the linearization of $\mY_k\colon\kC\to\bR$ (with scale $\kappa$) when restricted to homotheties.  The critical points of the $\mY_k$-functionals are characterized as follows:

\begin{prop}
 \label{prop:mY_critical}
 Let $k\in\bN$ and let $\kC$ be a weighted conformal class on a closed weighted manifold $(M^n,m,\mu)$; if $k\geq3$, assume that $\kC$ is locally conformally flat in the weighted sense.  Fix a scale $\kappa>0$.  Then $(g,v)\in\kC$ is a critical point of $\mY_k\colon\kC\to\bR$ if and only if
 \begin{multline}
  \label{eqn:mY_critical_curvature}
  \csigma_{k,\phi}^m + \frac{m}{m+n-2k}\kappa v^{-1}\cs_{k-1,\phi}^{m} \\ = \frac{\int\csigma_{k,\phi}^m\,d\nu}{\int d\nu} + \frac{m}{m+n-2k}\left(\frac{\int \cs_{k-1,\phi}^mv^{-1}\,d\nu}{\int v^{-1}\,d\nu}\right)\kappa v^{-1}
 \end{multline}
 and
  \begin{equation}
   \label{eqn:mY_critical_integral}
   \int_M\csigma_{k,\phi}^m\,d\nu = \frac{(m+n)(2m+n-2)}{2k(m+n-1)}\int_M \kappa v^{-1}\cs_{k-1,\phi}^{m}\,d\nu .
  \end{equation}
\end{prop}

\begin{proof}
 Using Proposition~\ref{prop:linearization_sigmak} and Theorem~\ref{thm:variational}, we compute that
 \begin{equation}
  \label{eqn:DmZ}
  D\mZ_k[\psi] = -B(\kappa)\int_M\left(\frac{m+n-2k}{m+n}\csigma_{k,\phi}^m + \frac{m}{m+n}\kappa v^{-1}\cs_{k-1,\phi}^m\right)\psi\,d\nu
 \end{equation}
 for all $\psi\in T_{(g,v)}\kC$ and all $(g,v)\in\kC$, where $B(\kappa)=\kappa^{-\frac{2mk(m+n-1)}{(m+n)(2m+n-2)}}$.  It follows that
 \begin{multline*}
  \left(\int_M v^{-1}\,d\nu\right)^{\frac{2mk}{(m+n)(2m+n-2)}}\left(\int_M d\nu\right)^{\frac{m+n-2k}{m+n}}D\mY_k[\psi] \\ = D\mZ_k[\psi] + \frac{2mk(m+n-1)}{(m+n)^2(2m+n-2)}\mZ_k\frac{\int \psi v^{-1}\,d\nu}{\int v^{-1}\,d\nu} + \frac{m+n-2k}{m+n}\mZ_k\frac{\int \psi\,d\nu}{\int d\nu}
 \end{multline*}
 for all $\psi\in T_{(g,v)}\kC$ and all $(g,v)\in\kC$.  Combining these two observations, we see that $(g,v)\in\kC$ is a critical point of $\mY_k\colon\kC\to\bR$ if and only if
 \begin{multline}
  \label{eqn:mY_critical_curvature_pf}
  \csigma_{k,\phi}^m + \frac{m}{m+n-2k}\kappa v^{-1}\cs_{k-1,\phi}^{m} \\ = \left(\frac{1}{\int d\nu} + \frac{2mk(m+n-1)}{(m+n)(m+n-2k)(2m+n-2)}\frac{v^{-1}}{\int v^{-1}\,d\nu}\right)\int_M\csigma_{k,\phi}^m\,d\nu .
 \end{multline}
 Integrating with respect to $d\nu$ yields that~\eqref{eqn:mY_critical_curvature_pf} is equivalent to~\eqref{eqn:mY_critical_curvature} and~\eqref{eqn:mY_critical_integral}.
\end{proof}

%% file: stable.tex
\section{Stability results for the $\mF$- and $\mY$-functionals}
\label{sec:stable}

It is known that closed Einstein metrics and closed gradient Ricci solitons are stable with respect to the total $\sigma_k$-curvature functionals~\cite{Viaclovsky2000} and the total weighted $\sigma_k$-curvature functionals~\cite{Case2014sd}, respectively.  In this section we show that the same is true for quasi-Einstein metrics.  Based on these results and their usual proofs via the Lichnerowicz--Obata Theorem~\cite{Lichnerowicz1958,Obata1962}, we conjecture a Poincar\'e-type inequality for weighted Einstein manifolds which would imply that such manifolds are stable with respect to the $\mY_k$-functionals.

\subsection{Stability for quasi-Einstein manifolds}
\label{subsec:stable/qe}

We begin by computing the second variation of the $\mF_k$-functional in the cases when the weighted $\sigma_k$-curvature is variational.

\begin{prop}
 \label{prop:second_var_no_scale}
 Let $k\in\bN$ and let $\kC$ be a weighted conformal class on a closed weighted manifold $(M^n,m,\mu)$; if $k\geq3$, assume additionally that $\kC$ is locally conformally flat in the weighted sense.  Suppose that $(g,v)\in\kC_1$ is a critical point of the $\mF_k$-functional.  Then the second variation $D^2\mF_k\colon T_{(g,v)}\kC_1\to\bR$ is given by
 \begin{multline*}
  D^2\mF_k[\psi] = \frac{(m+n-2)(m+n-2k)}{(m+n)^2}\int_M T_{k-1,\phi}^m(\nabla\psi,\nabla\psi)\,d\nu \\ - \frac{2k(m+n-2k)}{(m+n)^2}\int_M\sigma_{k,\phi}^m\psi^2\,d\nu
 \end{multline*}
 for all $\psi\in T_{(g,v)}\kC_1$.
\end{prop}

\begin{proof}
 By Proposition~\ref{prop:mF_critical}, since $(g,v)$ is a critical point of $\mF_k\colon\kC_1\to\bR$, the weighted $\sigma_k$-curvature $\sigma_{k,\phi}^m$ is constant.  Therefore
 \[ D^2\mF_k[\psi] = -\frac{m+n-2k}{m+n}\int_M \psi\,D\sigma_{k,\phi}^m[\psi]\,d\nu \]
 for all $\psi\in T_{(g,v)}\kC_1$.  The conclusion now follows from Proposition~\ref{prop:linearization_sigmak}.
\end{proof}

Applying Proposition~\ref{prop:second_var_no_scale} in the case of a closed quasi-Einstein manifold yields the desired stability result.

\begin{thm}
 \label{thm:stable_qe}
 Let $k\in\bN$ and let $(M^n,g,v,m,\mu)$ be a closed quasi-Einstein manifold such that $P_\phi^m=\lambda g>0$; if $k\geq3$, assume additionally that $(g,v)$ is locally conformally flat in the weighted sense.  Then
 \begin{enumerate}
  \item if $k<\frac{m+n}{2}$, then $D^2\mF_k\colon T_{(g,v)}\kC_1\to\bR$ is positive definite; and
  \item if $\frac{m+n}{2}<k\leq m+n$, then $D^2\mF_k T_{(g,v)}\kC_1\to\bR$ is negative definite.
 \end{enumerate}
\end{thm}

\begin{proof}
 It follows readily from Proposition~\ref{prop:we_sigma2} and Proposition~\ref{prop:second_var_no_scale} that
 \begin{multline}
  \label{eqn:stable_qe_second_var}
  D^2\mF_k[\psi] = \binom{m+n-1}{k-1}\frac{(m+n-2)(m+n-2k)}{(m+n)^2}\lambda^{k-1}\int_M\lv\nabla\psi\rv^2\,d\nu \\ - \binom{m+n}{k}\frac{2k(m+n-2k)}{(m+n)^2}\lambda^k\int_M \psi^2\,d\nu
  \end{multline}
 for all $\psi\in T_{(g,v)}\kC_1$.  Since $P_\phi^m=\lambda g>0$, we see that $\Ric_\phi^m=\frac{2(m+n-1)}{m+n-2}\lambda g>0$.  The weighted Lichnerowicz theorem~\cite[Theorem~14]{BakryQian2000} implies that $\lambda_1(-\Delta_\phi)>\frac{2(m+n)}{m+n-2}$.  Inserting this into~\eqref{eqn:stable_qe_second_var} yields the result.
\end{proof}

\begin{remark}
 We only consider the case of quasi-Einstein manifolds with positive weighted Schouten tensor because any closed quasi-Einstein manifold with nonpositive weighted Schouten tensor is an Einstein manifold~\cite{Kim_Kim}.
\end{remark}

\subsection{Stability for weighted Einstein metrics}
\label{subsec:stable/we}

Proposition~\ref{prop:we_sigma2} and Proposition~\ref{prop:mY_critical} imply that weighted Einstein manifolds with $\mu=0$ and scale $\kappa$ are critical points of the $\mY_k$-functional.  In this subsection we conjecture a Lichnerowicz--Obata-type result which, if true, implies that such manifolds are infinitesimal minimizers of the $\mY_k$ functionals.  To that end, we compute the linearization of $\cs_{k-1,\phi}^m$ at a weighted Einstein manifold.

\begin{lem}
 \label{lem:we_csk-1_linearization}
 Let $k\in\bN$ and let $(M^n,g,v,m,0)$ be a weighted Einstein manifold with $P_\phi^m=\lambda g>0$ and scale $\kappa>0$; if $k\geq3$, assume additionally that $\kC=[g,v]$ is locally conformally flat in the weighted sense.  Then
 \begin{multline}
  \label{eqn:we_csk-1_linearization}
  D\cs_{k-1,\phi}^m = \binom{m+n-2}{k-2}\lambda^{k-2}\biggl[\frac{2(m+n-1)}{m+n}\lambda\psi - \frac{m-1}{m+n}\kappa\psi v^{-1} \\ + \frac{m+n-2}{m+n}v\delta_\phi\left(v^{-1}\nabla\psi\right)\biggr] .
 \end{multline}
\end{lem}

\begin{proof}
 By Lemma~\ref{lem:s_to_sigma-1},
 \[ \cs_{k-1,\phi}^m = \sigma_{k-1}^{m-1}\left(\frac{m-1}{m}\left(Y_\phi^m+m\kappa v^{-1}\right);P_\phi^m\right) . \]
 Set $\kappa^{(m-1)}:=\frac{m+n-3}{m+n-2}\kappa$.  Arguing as in the proof of Lemma~\ref{lem:csigma_to_sigma_genl}, we observe that for any weighted conformal class $\kC$, assumed to be locally conformally flat in the weighted sense if $k\geq3$, it holds that
 \begin{equation}
  \label{eqn:csk_to_csigmak-1}
  \cs_{k-1,\phi}^m = \left(\frac{m+n-2}{m+n-3}\right)^{k-1}\csigma_{k-1,\phi}^{m-1},
 \end{equation}
 where $\csigma_{k-1,\phi}^{m-1}$ is defined in terms of the scale $\kappa^{(m-1)}$.  It follows from Corollary~\ref{cor:change_TkM} and Proposition~\ref{prop:linearization_sigmak} that the linearization $D\csigma_{k-1,\phi}^{m-1}\colon T_{(g,v)}\kC(M,m,0)\to C^\infty(M)$ is
 \begin{multline}
  \label{eqn:linearization_sk}
  D\csigma_{k-1,\phi}^{m-1}[\psi] = \frac{1}{m+n}\left(2(k-1)\csigma_{k-1,\phi}^{m-1} - (m-1)\kappa^{(m-1)} v^{-1}\cs_{k-2,\phi}^{m-1}\right)\psi \\ + \frac{m+n-3}{m+n}v\delta_\phi\left(v^{-1}\cT_{k-2,\phi}^{m-1}\left(\nabla\psi\right)\right)
 \end{multline}
 for
 \begin{align*}
  \cs_{k-2,\phi}^{m-1} & := s_{k-2}^{m-1}\left(Y_\phi^{m-1} + (m-1)\kappa^{(m-1)}v^{-1};P_\phi^{m-1}\right), \\
  \cT_{k-2,\phi}^{m-1} & := T_{k-2}^{m-1}\left(Y_\phi^{m-1} + (m-1)\kappa^{(m-1)}v^{-1};P_\phi^{m-1}\right) .
 \end{align*}
 Inserting this into~\eqref{eqn:csk_to_csigmak-1}, using~\eqref{eqn:we_csk} and~\eqref{eqn:csk_to_csigmak-1} to evaluate $\csigma_{k-1,\phi}^{m-1}$, and using~\eqref{eqn:P-k_lcf} to evaluate $\cs_{k-2,\phi}^{m-1}$ and $\cT_{k-2,\phi}^{m-1}$ when $k\geq3$ yields~\eqref{eqn:we_csk-1_linearization}.
\end{proof}

Lemma~\ref{lem:we_csk-1_linearization} enables us to compute the second variation of the $\mY_k$-functional at a weighted Einstein manifold.

\begin{prop}
 \label{prop:second_var_scale}
 Let $k\in\bN$ and let $(M^n,g,v,m,0)$ be a closed weighted Einstein manifold with $P_\phi^m=\lambda g>0$ and scale $\kappa>0$; if $k\geq3$, assume additionally that $\kC=[g,v]$ is locally conformally flat in the weighted sense.  Then $(g,v)$ is a critical point of $\mY_k\colon\kC\to\bR$ and the second variation $D^2\mY_k\colon T_{(g,v)}\kC\to\bR$ is given by
 \begin{multline*}
  D^2\mY_k[\psi] = \frac{(m+n-2)(m+n-2k)}{(m+n)^2}\binom{m+n-1}{k-1}\mV_{-1}^{-a}\mV_0^{-b}\lambda^{k-1}I_1[\psi] \\ + \frac{m(m+n-2)}{(m+n)^2}\binom{m+n-2}{k-2}\mV_{-1}^{-a}\mV_0^{-b}\lambda^{k-2}\kappa I_2[\psi]
 \end{multline*}
 for all $\psi\in T_{(g,v)}\kC$, where $\mV_{-1}$ and $\mV_0$ denote the functionals $\mV_{-1}:=\int v^{-1}d\nu$ and $\mV_0:=\int d\nu$; $a$ and $b$ denote the constants $a:=\frac{2mk}{(m+n)(2m+n-2)}$, $b:=\frac{m+n-2k}{m+n}$; and
 \begin{subequations}
  \label{eqn:we_eigenvalues}
  \begin{multline}
   \label{eqn:we_eigenvalues0}
   I_1[\psi] := \int_M \biggl[ \lv\nabla\psi\rv^2 - \frac{2(m+n)}{m+n-2}\left(\psi-\opsi\right)^2 \\ + \frac{m}{m+n-2}\left(\psi^2 - 2\opsi\psi + \frac{(m+n-1)(2m+n)}{(m+n)(2m+n-2)}\left(\frac{\int\psi v^{-1}}{\int v^{-1}}\right)\psi\right)\kappa v^{-1} \biggr]d\nu
  \end{multline}
  for $\opsi=\frac{\int\psi}{\int 1}$ and
  \begin{multline}
   \label{eqn:we_eigenvalues1}
   I_2[\psi] := \int_M \biggl[ \lv\nabla\psi\rv^2 - \frac{2(m+n-1)}{m+n-2}\lambda\psi^2 + \frac{m-1}{m+n-2}\kappa\psi^2v^{-1} \\ + \frac{2(m+n-1)^2}{(m+n-2)(2m+n-2)}\left(\frac{\int\psi v^{-1}}{\int v^{-1}}\right)\lambda\psi\biggr]v^{-1}\,d\nu .
  \end{multline}
 \end{subequations}
\end{prop}

\begin{proof}
 By rescaling if necessary, we may suppose that $\kappa=1$.  It follows immediately from Proposition~\ref{prop:we_sigma2}, Proposition~\ref{prop:we_integral_relation} and Proposition~\ref{prop:mY_critical} that $(g,v)$ is a critical point of the $\mY_k$-functional.  In particular,
 \begin{equation}
  \label{eqn:mY-crit}
  0 = D\mY_k = \frac{1}{\mV_{-1}^a\mV_0^b}D\mZ_k - a\frac{\mY_k}{\mV_{-1}}D\mV_{-1} - b\frac{\mY_k}{\mV_0}D\mV_0 .
 \end{equation}
 It follows from~\eqref{eqn:mY-crit} that
 \begin{multline}
  \label{eqn:D2mY}
  D^2\mY_k = \frac{1}{\mV_{-1}^a\mV_0^b}D^2\mZ_k - a\frac{\mY_k}{\mV_{-1}}D^2\mV_{-1} - b\frac{\mY_k}{\mV_0}D^2\mV_0 - a(a-1)\frac{\mY_k}{\mV_{-1}^2}\left(D\mV_{-1}\right)^2 \\ - 2ab\frac{\mY_k}{\mV_{-1}\mV_0}\left(D\mV_{-1}\right)\left(D\mV_0\right) - b(b-1)\frac{\mY_k}{\mV_0^2}\left(D\mV_0\right)^2 .
 \end{multline}
 Next, observe that
 \begin{equation}
  \label{eqn:linearize_volumes}
  \begin{aligned}
   D\mV_{-1}[\psi] & = -\frac{m+n-1}{m+n}\int_M \psi v^{-1}\,d\nu, & D\mV_0[\psi] & = -\int_M \psi\,d\nu, \\
   D^2\mV_{-1}[\psi] & = \left(\frac{m+n-1}{m+n}\right)^2\int_M \psi^2v^{-1}\,d\nu , & D^2\mV_0[\psi] & = \int_M \psi^2\,d\nu .
  \end{aligned}
 \end{equation}
 From~\eqref{eqn:DmZ} we note that
 \begin{equation}
  \label{eqn:D2mZ}
  \begin{aligned}
   D^2\mZ_k[\psi] & = -\int_M \left(\frac{m+n-2k}{m+n}D\csigma_{k,\phi}^m[\psi] + \frac{m}{m+n} v^{-1} D\cs_{k-1,\phi}^m[\psi]\right)\psi\,d\nu \\ & \quad + \int_M \left(\frac{m+n-2k}{m+n}\csigma_{k,\phi}^m + \frac{m(m+n-1)}{(m+n)^2} v^{-1}\cs_{k-1,\phi}^m\right)\psi^2\,d\nu .
  \end{aligned}
 \end{equation}
 Since $(M^n,g,v,m,0)$ is weighted Einstein, Proposition~\ref{prop:we_sigma2} and Proposition~\ref{prop:linearization_sigmak} imply that
 \begin{equation}
  \label{eqn:we_Dcsigmak}
  D\csigma_{k,\phi}^m[\psi] = \binom{m+n-1}{k-1}\left(2\lambda^k - \frac{m}{m+n}\lambda^{k-1} v^{-1} + \frac{m+n-2}{m+n}\Delta_\phi\right)\psi .
 \end{equation}
 Inserting~\eqref{eqn:we_csk-1_linearization}, \eqref{eqn:linearize_volumes}, \eqref{eqn:D2mZ} and~\eqref{eqn:we_Dcsigmak} into~\eqref{eqn:D2mY} yields the desired conclusion.
\end{proof}

Based on similar stability results for quasi-Einstein manifolds, we expect that weighted Einstein manifolds are stable with respect to the $\mY_k$-functionals in the cases when the weighted $\sigma_k$-curvatures are variational.  Indeed, based on the proofs of those results, we expect the following Poincar\'e-type inequalities for weighted Einstein manifolds.

\begin{conj}
 \label{conj:lichnerowicz_conjecture}
 Let $(M^n,g,v,m,0)$ be a closed weighted Einstein manifold with $P_\phi^m=\lambda g>0$ and scale $\kappa>0$.  Let $I_1,I_2\colon C^\infty(M)\to\bR$ be as in~\eqref{eqn:we_eigenvalues}.  Then
 \[ \inf\left\{ I_j[\psi] \suchthat \int_M \psi\,d\nu = 1 \right\} > 0 \]
 for $j\in\{1,2\}$.  In particular, $D^2\mY_k\colon T_{(g,v)}\kC\to\bR$ is positive definite.
\end{conj}

%% file: obata.tex
\section{Ellipticity and some Obata-type theorems}
\label{sec:obata}

The results of Section~\ref{sec:stable} prove that quasi-Einstein manifolds are infinitesimally rigid with respect to the $\mF_k$-functionals within a volume-normalized weighted conformal class in the cases when the weighted $\sigma_k$-curvatures are variational.  It is natural to ask if \emph{global} rigidity holds.  Global rigidity is known in the Riemannian~\cite{Obata1971,Viaclovsky2000} and infinite-dimensional~\cite{Case2014sd} cases when $k=1$ or within the weighted conformal class of the respective flat model.  One expects similar results for general $m\in\bR_+$.  In this section we prove the analogous global rigidity result for quasi-Einstein metrics.  We also prove a result which is expected to play a key role in establishing the analogous global rigidity result for weighted Einstein metrics.  These results hold within the positive weighted elliptic $k$-cones, so named because the Euler equations of the $\mF_k$- and $\mY_k$-functions are elliptic within these cones.  Based on these results, we conjecture certain sharp fully nonlinear Sobolev inequalities.

Our strategy is modeled on Obata's proof~\cite{Obata1971} that on a compact manifold, every conformally Einstein constant scalar curvature metric is itself Einstein.  There are two key ingredients in his proof.  First, the variational structure of the scalar curvature yields a particular trace-free tensor which is divergence-free for every constant scalar curvature metric.  Second, if a metric is conformally Einstein, then the trace-free part of the Schouten tensor is a positive multiple of an element of the range of the conformal Killing operator; i.e.\ the trace-free part of the Lie derivative on vector fields.  In our setting, when the weighted $\sigma_k$-curvature is variational, Proposition~\ref{prop:div_T} effectively identifies the desired weighted trace-free tensor which is divergence-free for critical points of the $\mF_k$- and $\mY_k$-functionals (cf.\ Section~\ref{sec:full}).  The analogous formula for the weighted Schouten tensor of a metric-measure structure which is conformal to a weighted Einstein metric is as follows (cf.\ \cite{Case2010b,Case2013y}):

\begin{lem}
 \label{lem:2we_to_wKe}
 Let $(M^n,g,v,m,\mu)$ be a smooth metric measure space and fix a scale $\kappa\in\bR$.  Suppose that $u\in C^\infty(M;\bR_+)$ is such that the smooth metric measure space $(M^n,\hg,\hv,m,\mu)$ with metric-measure structure $(\hg,\hv):=(u^{-2}g,u^{-1}v)$ is a weighted Einstein manifold with scale $\hkappa\in\bR$.  Then
 \begin{equation}
  \label{eqn:2we_to_wKe}
  u\left(P_\phi^m-\cZ_\phi^mg\right) = -(m+n-2)\left(\nabla^2u+\frac{1}{m}\lp\nabla u,\nabla\phi\rp\,g\right) + \left(\hkappa-\kappa u\right)v^{-1}\,g .
 \end{equation}
\end{lem}

\begin{proof}
 Since $(M^n,\hg,\hv,m,\mu)$ is weighted Einstein with scale $\hkappa$, it follows from~\eqref{eqn:we_cY} that
 \[ \hP_\phi^m-\hZ_\phi^m\hg = \hkappa\hv^{-1}\hg . \]
 Using Lemma~\ref{lem:conf_change} to evaluate $\hP_\phi^m$ and $\hZ_\phi^m$ in terms of $(M^n,g,v,m,\mu)$ and $u$ yields the desired result.
\end{proof}

In the cases of interest to us, only the flat models discussed in Section~\ref{sec:smms} admit multiple quasi-Einstein or weighted Einstein metric-measure structures within the given weighted conformal class (cf.\ \cite[Proposition~9.5]{Case2013y}).

\begin{prop}
 \label{prop:wKe_rigidity}
 Let $\kC$ be a weighted conformal class on a weighted manifold $(M^n,m,\mu)$.  Suppose that $(g,v),(\hg,\hv)\in\kC$ are weighted Einstein metric-measure structures with scale $\kappa,\hkappa\in\bR$, respectively.  Set $u=v\hv^{-1}$.  Then either $u$ is constant or $(M^n,g,v,m,\mu)$ splits isometrically as a warped product.  In particular:
 \begin{enumerate}
  \item If $(M^n,g,v,m,\mu)$ is closed and $\kappa=\hkappa=0$, then $u$ is constant.
  \item If $(M^n,g,v,m,\mu)=(S_+^n,d\theta^2,\cos r,m,1)$ and $\hkappa=\kappa=0$, then there exist a constant $c>0$ and a point $\xi\in\bR^n=x_{n+1}^{-1}(0)$ such that
  \[ u(\zeta)=c\left(\sqrt{1+\lv\xi\rv^2}+\xi\cdot\zeta\right) . \]
  \item If $(M^n,g,v,m,\mu)=(\bR^n,dx^2,1,m,0)$ and $\hkappa>0$, then there exist a constant $c>0$ and a point $x_0\in\bR^n$ such that
  \[ u(x) = \frac{\hkappa}{2(m+n-2)}\lv x-x_0\rv^2 + c . \]
 \end{enumerate}
\end{prop}

\begin{remark}
 Recall that the assumption that $(M^n,g,v,m,\mu)$ is closed means that $M^n$ is a closed manifold, $v\in C^\infty(M;\bR_+)$ is a positive function on $M$, and $m\in\bR_+$.
\end{remark}

\begin{proof}
 Suppose that $u$ is nonconstant.  Then~\eqref{eqn:2we_to_wKe} implies that $\nabla^2u=\frac{1}{n}\Delta u\,g$.  It is well-known (cf.\ \cite{CheegerColding96}) that this condition implies that $(M^n,g)$ splits as a warped product over a one-dimensional base and that $u$ depends only on the base.  Indeed, if $P_\phi^m=\frac{m+n-2}{2}\lambda g$ and $\hP_\phi^m=\frac{m+n-2}{2}\hlambda\hg$, then Lemma~\ref{lem:conf_change} implies that
 \[ \nabla^2u = \frac{1}{2}\left(\hlambda u^{-1} - \lambda u + u^{-1}\lv\nabla u\rv^2\right)g . \]
 Integrating this yields a constant $c\in\bR$ such that
 \begin{equation}
  \label{eqn:integrate_wp_u}
  \left(u^\prime\right)^2 = -\lambda u^2 + cu - \hlambda
 \end{equation}
 (cf.\ \cite{Case2013y}).
 
 Suppose now that $(M^n,g,v,m,\mu)$ is closed and $\kappa=\hkappa=0$.  Solving~\eqref{eqn:integrate_wp_u} implies that $u$ is of the form $u(t)=a+b\cos t$ for $b\not=0$ and $a>\lv b\rv$.  Hence $(M^n,g)$ is homothetic to the round $n$-sphere.  By rescaling and changing coordinates if necessary, we may thus suppose that $u(x)=a+bx_{n+1}$ for $a>b>0$.  Since $(M^n,g,v,m,\mu)$ is quasi-Einstein and closed, $v$ must be constant.  Thus~\eqref{eqn:2we_to_wKe} yields $\nabla^2u=0$, a contradiction.  Thus $u$ is constant in this case.
 
 Next, suppose that $\hkappa=0$ and $(M^n,g,v,m,\mu)$ is the positive elliptic $m$-Gaussian.  By homothetically scaling if necessary, we may suppose that $\hP_\phi^m=\frac{m+n-2}{2}\hg$.  From~\eqref{eqn:integrate_wp_u} we conclude that there is a point $\xi\in\bR^{n+1}$ such that
 \[ u(\zeta) = \sqrt{1+\lv\xi\rv^2} + \xi\cdot\zeta . \]
 On the other hand, \eqref{eqn:2we_to_wKe} implies that
 \[ \nabla^2u = \frac{\lp\nabla u,\nabla x_{n+1}\rp}{x_{n+1}}\,d\theta^2 , \]
 from which we conclude that $\xi_{n+1}=0$.
 
 Finally, suppose that $(M^n,g,v,m,\mu)=(\bR^n,dx^2,1,m,0)$ and $\hkappa>0$.  From~\eqref{eqn:2we_to_wKe} we conclude that
 \[ \nabla^2u = \frac{\hkappa}{m+n-2}\,dx^2 . \]
 Hence $u$ is a quadratic polynomial on $\bR^n$ with leading order term $\frac{\hkappa}{2(m+n-2)}\lv x\rv^2$, as desired.
\end{proof}

\subsection{An Obata-type theorem for quasi-Einstein manifolds}
\label{subsec:obata/qe}

When the scale vanishes $\kappa$, the tensor field $E_{k,\phi}^m$ defined below is the desired weighted analogue of the trace-free tensor which is divergence-free for Riemannian metrics with constant $\sigma_k$-curvature.

\begin{lem}
 \label{lem:divE}
 Let $k\in\bN$ and let $(M^n,g,v,m,\mu)$ be a smooth metric measure space; if $k\geq2$, assume additionally that $(g,v)$ is locally conformally flat in the weighted sense.  Define
 \[ E_{k,\phi}^m := T_{k,\phi}^m - \frac{m+n-k}{m+n}\sigma_{k,\phi}^mg . \]
 Then
 \[ \delta_\phi E_{k,\phi}^m - \frac{1}{m}\tr E_{k,\phi}^m\,d\phi = -\frac{m+n-k}{m+n}d\sigma_{k,\phi}^m . \]
\end{lem}

\begin{proof}
 A straightforward computation yields
 \[ \tr T_{k,\phi}^m = (m+n-k)\sigma_{k,\phi}^m - ms_{k,\phi}^m . \]
 In particular, it holds that
 \[ \tr E_{k,\phi}^m = \frac{m(m+n-k)}{m+n}\sigma_{k,\phi}^m - ms_{k,\phi}^m . \]
 The conclusion now follows from Proposition~\ref{prop:div_T}.
\end{proof}

We only expect global rigidity within the positive weighted elliptic $k$-cones.

\begin{defn}
 Fix $k\in\bN_0$.  The \emph{positive weighted elliptic $k$-cone $\Gamma_k^+$} on a weighted manifold $(M^n,m,\mu)$ is the set
 \[ \Gamma_k^+ := \left\{ (g,v)\in\kM \suchthat (g(p),v(p))\in\Gamma_k^+ \text{ for all $p\in M$} \right\} . \]
\end{defn}

Note that Euler equation of the $\mF_k$-functional is elliptic within the positive weighted elliptic $k$-cone.

\begin{prop}
 \label{prop:mFk_elliptic}
 Let $\kC$ be a weighted conformal class on $(M^n,m,\mu)$.  Fix $k\in\bN_0$ and a representative $(g,v)\in\kC$.  Identify
 \begin{equation}
  \label{eqn:Gammak_fn}
  \kC\cap\Gamma_k^+ = \left\{ u\in C^\infty(M;\bR) \suchthat (u^{-2}g,u^{-1}v) \in \Gamma_k^+ \right\} .
 \end{equation}
 Then the operator $D\colon\kC\cap\Gamma_k^+\to C^\infty(M)$ defined by $D(u):=\sigma_{k,\phi}^m(u^{-2}g,u^{-1}v)$ is elliptic.
\end{prop}

\begin{proof}
 Proposition~\ref{prop:linearization_sigmak} implies that the principal symbol of the linearization of $D$ at $u\in\kC\cap\Gamma_k^+$ is $\frac{m+n-2}{m+n}T_{k-1,\phi}^m$, where $T_{k-1,\phi}^m$ is the $(k-1)$-th weighted Newton tensor of $(u^{-2}g,u^{-1}v)$.  Corollary~\ref{cor:weighted_elliptic} then implies that $D$ is elliptic at $u$.
\end{proof}

We now adapt Obata's argument~\cite{Obata1971} to closed quasi-Einstein manifolds.

\begin{thm}
 \label{thm:obata/qe}
 Let $k\in\bN$ and let $(M^n,\hg,\hv,m,\mu)$ be a closed quasi-Einstein manifold such that $\int d\widehat{\nu}=1$ and $P_\phi^m>0$; if $k\geq2$, assume additionally that $\kC:=[\hg,\hv]$ is locally conformally flat in the weighted sense.  Then $(g,v)\in\kC_1\cap\Gamma_k^+$ is a critical point of $\mF_k\colon\kC_1\to\bR$ if and only if $(g,v)=(\hg,\hv)$.
\end{thm}

\begin{proof}
 First, suppose that $(g,v)=(\hg,\hv)$.  It follows from Proposition~\ref{prop:we_sigma2} and Proposition~\ref{prop:mF_critical} that $(g,v)$ is in the weighted elliptic cone $\Gamma_k^+$ and is a critical point of the $\mF_k$-functional $\mF_k\colon\kC_1\to\bR$.
 
 Conversely, suppose that $(g,v)$ is a critical point of the $\mF_k$-functional.  It follows from Proposition~\ref{prop:mF_critical} that $\sigma_{k,\phi}^m$ is constant.  Let $E_{k,\phi}^m$ be as in Lemma~\ref{lem:divE}.  Then
 \begin{equation}
  \label{eqn:divE_vanishes}
  \delta_\phi E_{k,\phi}^m - \frac{1}{m}\tr E_{k,\phi}^m\,d\phi = 0 .
 \end{equation}
 Let $u=v\hv^{-1}$.  Using Lemma~\ref{lem:2we_to_wKe} and~\eqref{eqn:divE_vanishes}, we compute that
 \begin{align*}
  0 & = \int_M \left\lp E_{k,\phi}^m, \nabla^2u + \frac{1}{m}\lp\nabla u,\nabla\phi\rp\,g\right\rp\,d\nu \\
  & = -\frac{1}{m+n-2}\int_M u\left\lp E_{k,\phi}^m, P_\phi^m - Z_\phi^m\,g\right\rp\,d\nu .
 \end{align*}
 It follows from Corollary~\ref{cor:inner_product_maclaurin} that $(g,v)$ is quasi-Einstein.  Proposition~\ref{prop:wKe_rigidity} and the normalization $(g,v),(\hg,\hv)\in\kC_1$ then imply that $(g,v)=(\hg,\hv)$.
\end{proof}

We expect that the assumption that $M$ is closed in Theorem~\ref{thm:obata/qe} can be removed; i.e.\ that one can use the assumption that $d\nu$ is a finite measure to still carry out the integration by parts (cf.\ \cite{ChangGurskyYang2003b,Gonzalez2006c}).  We further expect that, with a lot of work, one can show that the $\mF_k$-functional realizes its infimum under suitable geometric assumptions on the background smooth metric measure space (cf.\ \cite{GuanWang2003b,ShengTrudingerWang2007}).  These expectations motivate the following conjecture (cf.\ \cite{GuanWang2004}).

\begin{conj}
 \label{conj:sobolev/qe}
 Fix $k\in\bN$ and let $\kC$ be the weighted conformal class of the weighted elliptic $m$-Gaussian $(S_+^n,d\theta^2,\cos r,m,1)$.  It holds that
 \begin{equation}
  \label{eqn:sobolev/qe}
  \int_M \sigma_{k,\phi}^m\,d\nu \geq C\left(\int_M d\nu\right)^{\frac{m+n-2k}{m+n}}
 \end{equation}
 for all $(g,v)\in\kC\cap\Gamma_k^+$, where $C=\mY_k(d\theta^2,\cos r)$.  Moreover, equality holds in~\eqref{eqn:sobolev/qe} if and only if $(S_+^n,g,v,m,1)$ is homothetic to a weighted elliptic $m$-Gaussian.
\end{conj}

\subsection{Towards an Obata theorem for the $\mY_k$-functional on $S^n$}
\label{subsec:obata/we}

For smooth metric measure spaces with positive scale, the analogue of Lemma~\ref{lem:divE} is as follows:

\begin{prop}
 \label{prop:obata/we/div}
 Fix $k\in\bN$ and let $(M^n,g,v,m,\mu)$ be a smooth metric measure space; if $k\geq2$, assume additionally that $(g,v)$ is locally conformally flat in the weighted sense.  Define
 \begin{align}
  \label{eqn:cEk} \cE_{k,\phi}^m & := \cT_{k,\phi}^m - \frac{m+n-k}{m+n}\csigma_{k,\phi}^mg, \\
  \label{eqn:cUk} \cU_{k-1,\phi}^m & := T_{k-1}^{m-1}\left(\frac{m-1}{m}\cY_\phi^m;P_\phi^m\right) - \frac{m+n-k}{m+n}\cs_{k-1,\phi}^mg, \\
  \label{eqn:hsigmak} \hsigma_{k,\phi}^m & := \csigma_{k,\phi}^m + \frac{m}{m+n-2k}\left(\cs_{k-1,\phi}^m - \frac{\int\cs_{k-1,\phi}^mv^{-1}}{\int v^{-1}}\right)\kappa v^{-1}, \\
  \label{eqn:hEk} \hE_{k,\phi}^m & := \cE_{k,\phi}^m + \frac{m}{m+n-2k}\kappa v^{-1}\cU_{k-1,\phi}^m \\
  \notag & \qquad - \frac{m(m+n-k)}{(m+n)(m+n-1)(m+n-2k)}\left(\frac{\int\cs_{k-1,\phi}^mv^{-1}}{\int v^{-1}}\right)\kappa v^{-1}g .
 \end{align}
 Then it holds that
 \[ \delta_\phi\hE_{k,\phi}^m - \frac{1}{m}\tr\hE_{k,\phi}^m\,d\phi = -\frac{m+n-k}{m+n}d\hsigma_{k,\phi}^m . \]
\end{prop}

\begin{proof}
 A straightforward computation yields
 \[ \tr\cE_{k,\phi}^m = \frac{m(m+n-k)}{m+n}\csigma_{k,\phi}^m - m\cs_{k,\phi}^m . \]
 Combining this with Proposition~\ref{prop:div_T} yields
 \begin{equation}
  \label{eqn:div_E}
  \delta_\phi\cE_{k,\phi}^m - \frac{1}{m}\tr\cE_{k,\phi}^m\,d\phi = -\frac{m+n-k}{m+n}d\csigma_{k,\phi}^m .
 \end{equation}
 
 Next we show that
 \begin{equation}
  \label{eqn:div_U}
  \delta_\phi\left(v^{-1}\cU_{k-1,\phi}^m\right) - \frac{1}{m}\tr\left(v^{-1}\cU_{k-1,\phi}^m\right)\,d\phi = -\frac{m+n-k}{m+n}d\left(v^{-1}\cs_{k-1,\phi}^m\right) .
 \end{equation}
 This is clear if $k=1$, so suppose $k\geq2$.  Arguing as in the proof of Lemma~\ref{lem:csigma_to_sigma_genl}, we see that
 \begin{equation}
  \label{eqn:cU_to_cE}
  \cU_{k-1,\phi}^m = \left(\frac{m+n-2}{m+n-3}\right)^{k-1}\left(\cE_{k-1,\phi}^{m-1}+\frac{m+n-k}{(m+n)(m+n-1)}\csigma_{k-1,\phi}^{m-1}g\right)
 \end{equation}
 (cf.\ Lemma~\ref{lem:we_csk-1_linearization}), where $\cE_{k-1,\phi}^{m-1}$ is defined by~\eqref{eqn:cEk} in terms of $(M^n,g,v,m-1,\mu)$ and the scale $\kappa^{(m-1)}:=\frac{m+n-3}{m+n-2}\kappa$.  Let $\delta_\phi^{(m-1)}$ denote the divergence with respect to the weighted measure $d\nu^{(m-1)}$ of $(M^n,g,v,m-1,\mu)$.  Note that $\delta_\phi^{(m-1)}=v\circ\delta_\phi^{(m)}\circ v^{-1}$, where $v$ and $v^{-1}$ act as multiplication operators.  In particular, \eqref{eqn:div_E} and~\eqref{eqn:cU_to_cE} together yield~\eqref{eqn:div_U}.
 
 Finally, a simple calculation yields
 \begin{equation}
  \label{eqn:simple_calculation/obata}
  \delta_\phi\left(v^{-1}g\right) - \frac{1}{m}\tr\left(v^{-1}g\right)\,d\phi = -(m+n-1)dv^{-1}.
 \end{equation}
 Combining~\eqref{eqn:div_E}, \eqref{eqn:div_U} and~\eqref{eqn:simple_calculation/obata} yields the conclusion.
\end{proof}

By Proposition~\ref{prop:mY_critical}, the Euler equation of the $\mY_k$-functional with scale $\kappa$ is completely determined by $\hsigma_{k,\phi}^m$.  The Euler equation is also elliptic within the positive weighted elliptic $k$-cone.

\begin{prop}
 \label{prop:mYk_elliptic}
 Fix $k\in\bN_0$ and $\kappa\in\bR_+$.  Let $\kC$ be a weighted conformal class on $(M^n,m,0)$ and fix a representative $(g,v)\in\kC$.  If $k\geq 3$, assume additionally that $\kC$ is locally conformally flat in the weighted sense.  In terms of~\eqref{eqn:Gammak_fn}, the operator $D\colon\kC\cap\Gamma_k^+\to C^\infty(M)$ defined by $D(u):=\hsigma_{k,\phi}^m(u^{-2}g,u^{-1}v)$ is elliptic.
\end{prop}

\begin{proof}
 By Proposition~\ref{prop:linearization_sigmak} and~\eqref{eqn:linearization_sk}, the principal symbol of the linearization of $D$ at $u\in\kC\cap\Gamma_k^+$ is
 \[ \frac{m+n-2}{m+n}\left(\cT_{k-1,\phi}^m + \frac{m}{m+n-2k}\kappa v^{-1}\cT_{k-2,\phi}^{m-1}\right), \]
 where $\cT_{k-1,\phi}^m$ and $\cT_{k-2,\phi}^m$ are defined in terms of $(M^n,u^{-2}g,u^{-1}v,m,0)$ with scale $\kappa$ and $(M^n,u^{-2}g,u^{-1}v,m-1,0)$ with scale $\frac{m+n-3}{m+n-2}\kappa$, respectively.  Corollary~\ref{cor:weighted_elliptic} implies that $\cT_{k-1,\phi}^m>0$.  Lemma~\ref{lem:s_to_sigma-1}, Corollary~\ref{cor:weighted_elliptic}, and an argument as in the proof of Lemma~\ref{lem:csigma_to_sigma_genl} imply that $\cT_{k-2,\phi}^m>0$.  This yields the conclusion.
\end{proof}

The form of the tensor $\hE_{k,\phi}^m$ makes it difficult to deduce a general Obata-type theorem for conformally weighted Einstein manifolds $(M^n,g,v,m,0)$ with scale $\kappa\in\bR_+$ for which $\hsigma_{k,\phi}^m$ is constant.  In this setting, it is still the case that an integral pairing with $\hE_{k,\phi}^m$ vanishes, but it is not apparent how to deduce that $(g,v)$ is a weighted Einstein metric-measure structure.  This difficulty is even apparent in the standard conformal class of the $m$-weighted $n$-sphere, as we illustrate below:

\begin{cor}
 \label{cor:obata/we/div}
 Fix $k\in\bN$ and a scale $\kappa>0$.  Let $\kC=[g_0,v_0]$ be the standard weighted conformal class on the $m$-weighted $n$-sphere $(S^n,m,0)$.  Suppose that $(g,v)\in\kC$ is a critical point of $\mY_k\colon\kC\to\bR$.  Then
 \begin{equation}
  \label{eqn:obata/we/div}
  \int_{S^n} \left\lp \hE_{k,\phi}^m, v\left(P_\phi^m - \cZ_\phi^mg\right) + \kappa g\right\rp \, d\nu= 0 .
 \end{equation}
\end{cor}

\begin{remark}
 When $k=1$, one can check that if $(g,v)\in\kC\cap\Gamma_1^+$, then~\eqref{eqn:obata/we/div} has a sign; cf.\ \cite[Proposition~9.7]{Case2013y}.  It is unclear whether the analogous statement holds for $k\geq2$.
\end{remark}

\begin{proof}
 Since the compactification of the flat metric-measure structure $(dx^2,1)$ on $(\bR^n,m,0)$ is an element of $\kC$, it holds that $v^{-2}g=dx^2$.  In particular, Lemma~\ref{lem:2we_to_wKe} implies that
 \begin{equation}
  \label{eqn:sphere_veqn}
  v\left(P_\phi^m - \cZ_\phi^mg\right) + \kappa g = -(m+n-2)\left(\nabla^2v + \frac{1}{m}\lp\nabla v,\nabla\phi\rp\,g\right) .
 \end{equation}
 Since $(g,v)$ is a critical point of the $\mY_k$-functional, Proposition~\ref{prop:mY_critical} implies that $\hsigma_{k,\phi}^m$ is constant.  In particular, Proposition~\ref{prop:obata/we/div} yields that
 \begin{equation}
  \label{eqn:sphere_diveqn}
  \delta_\phi\hE_{k,\phi}^m - \frac{1}{m}\tr\hE_{k,\phi}^m\,d\phi = 0 .
 \end{equation}
 Combining~\eqref{eqn:sphere_veqn} and~\eqref{eqn:sphere_diveqn} yields~\eqref{eqn:obata/we/div}.
\end{proof}

Motivated both by the fully nonlinear Sobolev-type inequality known in Riemannian geometry~\cite{GuanWang2004} and our discussion surrounding Conjecture~\ref{conj:sobolev/qe}, we expect the following fully nonlinear Gagliardo--Nirenberg inequality:

\begin{conj}
 \label{conj:sobolev/we}
 Fix $k\in\bN$ and let $\kC=[dx^2,1]$ be the standard weighted conformal class on the $m$-weighted Euclidean space $(\bR^n,m,0)$.  It holds that
 \begin{equation}
  \label{eqn:we_ksobolev}
  \kappa^{-\frac{2mk(m+n-1)}{(m+n)(2m+n-2)}}\int_{\bR^n} \csigma_{k,\phi}^m \geq C\left(\int_{\bR^n} v^{-1}\,d\nu\right)^{\frac{2mk}{(m+n)(2m+n-2)}}\left(\int_{\bR^n} d\nu\right)^{\frac{m+n-2k}{m+n}}
 \end{equation}
 for all $(g,v)\in\kC\cap\Gamma_k^+$ and all $\kappa>0$, where $\csigma_{k,\phi}^m$ is defined in terms of the scale $\kappa$ and $C = \mY_k(g_0,v_0)$ is the $\mY_k$-functional evaluated at the standard $m$-weighted $n$-sphere with scale $\kappa=m+n-2$.  Moreover, equality holds in~\eqref{eqn:we_ksobolev} if and only if $(\bR^n,g,v,m,0)$ is homothetic to the standard $m$-weighted $n$-sphere with the point $(0,\dotsc,0,1)$ removed.
\end{conj}

Note that Del Pino and Dolbeault~\cite{DelPinoDolbeault2002} have proven Conjecture~\ref{conj:sobolev/we} in the case $k=1$ (cf.\ \cite{Case2013y}).

%% file: full.tex
\section{Critical points of the $\mY$-functional}
\label{sec:full}

In this section we compute the linearizations of the total weighted $\sigma_1$- and $\sigma_2$-curvatures with the goal of showing that weighted Einstein manifolds are among their critical points.  Specifically, in the case of scale zero, we show that quasi-Einstein manifolds are critical points of the restriction $\mF_k\colon\kM_1\to\bR$ of the $\mF_k$-functional to metric-measure structures of fixed volume when $k\in\{1,2\}$.  In the case when the scale $\kappa$ is positive, we show that weighted Einstein manifolds with $\mu=0$ and scale $\kappa$ are among the critical points of the $\mY_k$-functional $\mY_k\colon\kM\to\bR$ when $k\in\{1,2\}$.

In order to achieve our goal, we compute the linearizations of the $\mF_k$-functionals for $k\in\{0,1,2\}$.  This is enough to compute the linearizations of the $\mY_k$-functionals for $k\in\{0,1,2\}$.  Indeed, fix a scale $\kappa\in\bR$ and define functionals $\cmF_k\colon\kM\to\bR$ by
\[ \cmF_k(g,v) := \int_M \csigma_{k,\phi}^m\,d\nu, \]
where $\csigma_{k,\phi}^m$ is determined by $(g,v)\in\kM$ and the scale $\kappa$.  When $\kappa=0$, it holds that $\cmF_k=\mF_k$.  When $\kappa>0$, it holds that
\[ \mY_k(g,v) = \kappa^{-\frac{2mk(m+n-1)}{(m+n)(2m+n-2)}}\cmF_k(g,v)\left(\int_M v^{-1}\,d\nu\right)^{-\frac{2mk}{(m+n)(2m+n-2)}}\mF_0(g,v)^{-\frac{m+n-2k}{m+n}} . \]
Below we compute the linearizations of $\cmF_k$ and of $\int_M v^{-1}d\nu$ in terms of the linearizations of $\mF_j$, $j\in\{0,1,2\}$.

We compute the linearizations of the $\mF_k$-functionals by first computing the linearizations of the weighted $\sigma_k$-curvatures as functions of $\kM$.  While this level of generality is not necessary for our computations, we include it with the expectation that it will be useful for other purposes, such as computing the second variations of the weighted $\sigma_k$-curvature functionals or computing the linearizations of the weighted $\sigma_k$-curvature functionals for larger values of $k$.

\subsection{The first variation of $J_\phi^m$}
\label{subsec:full/1}

We begin by computing the first variation of the weighted scalar curvature.  As we illustrate below, this readily yields the first variation of integrals of powers of $J_\phi^m$.

\begin{lem}
 \label{lem:Rdot}
 Let $(M^n,g,v,m,\mu)$ be a smooth metric measure space.  The first variation $DJ_\phi^m\colon T_{(g,\psi)}\kM\to C^\infty(M)$ is given by
 \begin{align*}
  DJ_{\phi}^m[h,\psi] & = -\frac{m+n-2}{2(m+n-1)}\left[ \left\lp \tf_\phi P_\phi^m,h\right\rp - \delta_\phi^2h + \frac{1}{m+n}\Delta_\phi\tr h \right] \\
   & \quad + \frac{2}{m+n}J_\phi^m\psi + \frac{m+n-2}{m+n}\Delta_\phi\psi .
 \end{align*}
\end{lem}

\begin{proof}
 Let $\gamma\colon\bR\to\kM$ be a smooth curve with $\gamma(0)=(g,\phi)$ and denote $\gamma^\prime(0)=(\dot g,-\frac{1}{m}v\dot\phi)$.  By~\cite[(4.8)]{Case2010a},
 \[ \left.\frac{\partial R_\phi^m}{\partial t}\right|_{t=0} = -\left\lp\Ric_\phi^m,\dot g\right\rp + \delta_\phi^2\dot g + 2\Delta_\phi\left(\dot\phi-\frac{1}{2}\tr_g \dot g\right) - \frac{2}{m}\lp\nabla\phi,\nabla\dot\phi\rp + 2(m-1)\mu v^{-2}\dot\phi . \]
 The result then follows by using Lemma~\ref{lem:eval_Yphim} and~\eqref{eqn:TkM} to write this in terms of $h$, $\psi$, the weighted Schouten tensor, and its trace.
\end{proof}

An immediate consequence of~\eqref{eqn:TkM_geometric_psi} and Lemma~\ref{lem:Rdot} is the first variation of the total weighted scalar curvature functional.

\begin{cor}
 \label{cor:R1dot}
 Let $(M^n,g,v,m,\mu)$ be a closed smooth metric measure space.  Then
 \[ D\left(\int_M J_\phi^m\,d\nu \right)[h,\psi] = -\int_M \left[\frac{m+n-2}{2(m+n-1)} \left\lp \tf_\phi P_\phi^m, h\right\rp + \frac{m+n-2}{m+n} J_\phi^m\psi\right]\,d\nu . \]
\end{cor}

Lemma~\ref{lem:Rdot} also yields the first variation of $\int\bigl(J_\phi^m\bigr)^2$.

\begin{cor}
 \label{cor:R2dot}
 Let $(M^n,g,v,m,\mu)$ be a closed smooth metric measure space.  Then
 \begin{multline*}
  D\left( \int_M \left(J_\phi^m\right)^2\,d\nu \right)[h,\psi] = \frac{m+n-2}{m+n-1}\int_M \left\lp \tf_\phi\left(\nabla^2J_\phi^m - J_\phi^mP_\phi^m\right),h\right\rp\,d\nu \\ + \int_M \left(\frac{2(m+n-2)}{m+n}\Delta_\phi J_\phi^m - \frac{m+n-4}{m+n}\left(J_\phi^m\right)^2\right)\psi\,d\nu .
 \end{multline*}
\end{cor}

\begin{proof}
 Recall that $\Delta_\phi$ is formally self-adjoint with respect to $d\nu$ and that $\delta_\phi$ is the negative of the formal adjoint with respect to $d\nu$ of the Levi-Civita connection.  In particular,
 \[ \int_M u\delta_\phi^2h\,d\nu = \int_M \left\lp \nabla^2u,h\right\rp\,d\nu \]
 for all $u\in C^\infty(M)$ and all $h\in\Gamma(S^2T^\ast M)$.  The conclusion follows readily from Lemma~\ref{lem:Rdot}.
\end{proof}

\subsection{The first variation of $N_{2,\phi}^m$}
\label{subsec:full/2}

The first step in computing the first variation of $N_{2,\phi}^m$ is to compute the first variation $DP_\phi^m\colon T\kM\to\Gamma\left(S^2T^\ast M\right)$ of the weighted Schouten tensor.  To that end, we require some additional notation.

Given sections $A\in\Gamma\left(S^2\Lambda^2T^\ast M\right)$ and $T\in\Gamma\left(S^2T^\ast M\right)$, define $A\cdot T\in\Gamma\left(S^2T^\ast M\right)$ by
\[ (A\cdot T)(x,y) := \lp A(\cdot,x,\cdot,y), T\rp \]
for all $x,y\in T_pM$ and all $p\in M$.  Denote by $T\hash$ the extension of the natural action of $g^{-1}T\in\Gamma\left(T^\ast M\otimes TM\right)$ on vector fields to a derivation on tensor fields.  In particular, given $S\in\Gamma\left(S^2T^\ast M\right)$, the section $T\hash S\in\Gamma\left(S^2T^\ast M\right)$ is given by
\[ \left(T\hash S\right)(x,y) := -S\left(T(x),y\right) - S\left(x,T(y)\right) . \]
Denote by $dT\in\Gamma\left(\Lambda^2T^\ast M\otimes T^\ast M\right)$ the twisted exterior derivative
\[ dT(x,y,z) := \nabla_x T(y,z) - \nabla_y T(x,z) \]
and denote by $\delta_\phi dT\in\Gamma\left(T^\ast M\otimes T^\ast M\right)$ the composition with the weighted divergence
\[ \left(\delta_\phi dT\right)(x,y) := \sum_{i=1}^n \nabla_{e_i}dT(e_i,x,y) - dT(\nabla\phi,x,y) , \]
where $\{e_i\}_{i=1}^n$ is an orthonormal basis for $T_pM$.

\begin{lem}
 \label{lem:Pdot}
 Let $(M^n,g,v,m,\mu)$ be a smooth metric measure space.  Then
 \begin{align*}
  DP_\phi^m[h,\psi] & = -\frac{1}{2}\delta_\phi dh + \frac{1}{4}L_{\delta_\phi h}g - \frac{1}{2(m+n-1)}\left(\delta_\phi^2h-\Delta_\phi\tr h\right)g \\
   & \quad - \frac{1}{m+n}\nabla^2\tr h - \frac{1}{2}A_\phi^m\cdot h - \frac{1}{2(m+n-2)}J_\phi^m\,h \\
   & \quad + \frac{1}{2(m+n-1)(m+n-2)}\lp T_{1,\phi}^m,h\rp\,g - \frac{m+n}{4(m+n-2)}P_\phi^m\hash h \\
   & \quad - \frac{1}{2(m+n-2)}(\tr h)P_\phi^m - \frac{1}{4m}\left(d\phi\otimes d\phi\right)\hash h + \frac{m+n-2}{m+n}\nabla^2\psi .
 \end{align*}
\end{lem}

\begin{proof}
 Let $\gamma\colon\bR\to\kM$ be a smooth curve such that $\gamma(0)=(g,\phi)$ and denote $\gamma^\prime(0)=(\dot g,-\frac{1}{m}v\dot\phi)$.  It follows readily from well-known variational formulae for the Ricci tensor and the Hessian (cf.\ \cite[Section~1.K]{Besse}) that
 \begin{multline*}
  \left.\frac{\partial}{\partial t}\right|_{t=0}\Ric_\phi^m = -\frac{1}{2}\Delta_\phi\dot g + \frac{1}{2}L_{\delta_\phi\dot g}g - \frac{1}{2}\left(\Ric_\phi^m+\frac{1}{m}d\phi\otimes d\phi\right)\hash\dot g - \Rm\cdot\dot g \\ - \frac{2}{m}d\phi\odot d\dot\phi + \nabla^2\left(\dot\phi-\frac{1}{2}\tr\dot g\right) .
 \end{multline*}
 It follows from~\eqref{eqn:TkM} that
 \begin{multline*}
  D\Ric_\phi^m[h,\psi] = -\frac{1}{2}\Delta_\phi h + \frac{1}{2}L_{\delta_\phi h}g - \frac{1}{m+n}\nabla^2\tr h + \frac{1}{2(m+n)}\Delta_\phi(\tr h)g \\ - \frac{1}{2}\left(\Ric_\phi^m+\frac{1}{m}d\phi\otimes d\phi\right)\hash h - \Rm\cdot h + \frac{m+n-2}{m+n}\nabla^2\psi + \frac{1}{m+n}\Delta_\phi\psi\,g .
 \end{multline*}
 The final conclusion follows from Lemma~\ref{lem:Rdot} and the Weitzenb\"ock formula
 \[ \Delta_\phi T = \delta_\phi dT + \frac{1}{2}L_{\delta_\phi T}g - \Rm\cdot T - \frac{1}{2}\left(\Ric_\phi^m+\frac{1}{m}d\phi\otimes d\phi\right)\hash T \]
 which holds for all $T\in\Gamma\left(S^2T^\ast M\right)$ (cf.\ \cite[Lemma~5.6]{Case2014sd}).
\end{proof}

This allows us to compute the linearization $DY_\phi^m\colon T\kM\to C^\infty(M)$.

\begin{cor}
 \label{cor:Ydot}
 Let $(M^n,g,v,m,\mu)$ be a smooth metric measure space.  Then
 \begin{align*}
  DY_\phi^m[h,\psi] & = -\frac{m}{2(m+n-1)}\left(\delta_\phi^2h-\Delta_\phi h\right) - \frac{1}{2}\delta_\phi\left(h(\nabla\phi)\right) - \frac{1}{2}\lp\delta_\phi h,\nabla\phi\rp \\
   & \quad + \frac{1}{m+n}\lp\nabla\phi,\nabla\tr h\rp + \frac{1}{2}\lp\tr A_\phi^m,h\rp - \frac{1}{2m}h(\nabla\phi,\nabla\phi) \\
   & \quad + \frac{m}{2(m+n-1)(m+n-2)}\lp T_{1,\phi}^m,h\rp + \frac{m+n-4}{2(m+n)(m+n-2)}Y_\phi^m\tr h \\
   & \quad + \frac{2}{m+n}\psi Y_\phi^m - \frac{m+n-2}{m+n}\lp\nabla\phi,\nabla\psi\rp .
 \end{align*}
\end{cor}

\begin{proof}
 It follows from~\eqref{eqn:TkM} and the definition of $Y_\phi^m$ that
 \[ DY_\phi^m[h,\psi] = DJ_\phi^m[h,\psi] - \tr DP_\phi^m[h,\psi] + \left\lp P_\phi^m, h - \frac{2}{m+n}\left(\psi+\frac{1}{2}\tr h\right)g \right\rp . \]
 The final conclusion follows from Lemma~\ref{lem:Rdot} and Lemma~\ref{lem:Pdot}.
\end{proof}

Combining Lemma~\ref{lem:Pdot} and Corollary~\ref{cor:Ydot} yields a formula for the first variation $DN_{2,\phi}^m\colon T\kM\to C^\infty(M)$.  To that end, recall that given $A\in\Gamma\left(\Lambda^2T^\ast M\otimes T^\ast M\right)$ and $T\in\Gamma\left(S^2T^\ast M\right)$, we denote by $A\cdot T\in\Gamma\left(T^\ast M\right)$ the contraction
\[ \left(A\cdot T\right)(x) := \sum_{i=1}^n A\left(e_i,x,T(e_i)\right) . \]

\begin{lem}
 \label{lem:P2dot}
 Let $(M^n,g,v,m,\mu)$ be a smooth metric measure space.  Then
 \[ DN_{2,\phi}^m[h,\psi] = -\left\lp \Psi_{2,\phi}^m, h\right\rp + \Upsilon_{2,\phi}^m\psi + \delta_\phi
 \left(\Xi_{2,\phi}^m[h,\psi]\right) , \]
 where
 \begin{align*}
  \Psi_{2,\phi}^m & = \tf_\phi\Bigl(B_\phi^m + \frac{m+n-4}{m+n-2}\left(P_\phi^m\right)^2 - \frac{m+n-2}{m+n-1}\nabla^2J_\phi^m \\
   & \qquad + \frac{m+n}{(m+n-1)(m+n-2)}J_\phi^mP_\phi^m\Bigr), \\
  \Upsilon_{2,\phi}^m & = \frac{2(m+n-2)}{m+n}\Delta_\phi J_\phi^m + \frac{4}{m+n}N_{2,\phi}^m, \\
  \Xi_{2,\phi}^m[h,\psi] & = dh\cdot P_\phi^m - dP_\phi^m\cdot h + P_\phi^m\left(\delta_\phi h\right) - \frac{2}{m+n}P_\phi^m(\nabla\tr h) - \frac{1}{m}Y_\phi^mh(\nabla\phi) \\
   & \quad - \frac{m+n-2}{m+n-1}h\left(\nabla J_\phi^m\right) + \frac{m+n-2}{(m+n)(m+n-1)}(\tr h)dJ_\phi^m \\
   & \quad + \frac{1}{m+n-1}\left(J_\phi^m(d\tr h-\delta_\phi h)\right) + \frac{2(m+n-2)}{m+n}\left(P_\phi^m(\nabla\psi) - \psi\,dJ_\phi^m\right) .
 \end{align*}
\end{lem}

\begin{proof}
 From the definition of $N_{2,\phi}^m$ we see that
 \begin{multline*}
  DN_{2,\phi}^m[h,\psi] = 2\lp P_\phi^m,DP_\phi^m[h,\psi]\rp + \frac{2}{m}Y_\phi^m\,DY_\phi^m[h,\psi] \\ - 2\left\lp \left(P_\phi^m\right)^2, h - \frac{2}{m+n}\left(\psi+\frac{1}{2}\tr h\right)g\right\rp .
 \end{multline*}
 Lemma~\ref{lem:div_and_tr} implies that
 \begin{multline*}
  \left\lp \left(P_\phi^m-\frac{Y_\phi^m}{m}g\right)(\nabla\phi),h(\nabla\phi)\right\rp - Y_\phi^m\delta_\phi\left(h(\nabla\phi)\right) \\ = \left\lp \tr dP_\phi^m\otimes d\phi, h\right\rp - \delta_\phi\left(Y_\phi^mh(\nabla\phi)\right) .
 \end{multline*}
 Lemma~\ref{lem:div_and_tr} also implies that
 \begin{align*}
  \frac{1}{2}\left\lp P_\phi^m, L_{\delta_\phi h}g\right\rp - \frac{1}{m}\left\lp Y_\phi^m\nabla\phi,\delta_\phi h\right\rp & = \left\lp \nabla^2J_\phi^m,h\right\rp + \delta_\phi\left(P_\phi^m(\delta_\phi h) - h(\nabla J_\phi^m)\right) , \\
  \left\lp P_\phi^m,\nabla^2\tr h\right\rp - \frac{1}{m}\left\lp Y_\phi^m\nabla\phi,\nabla\tr h\right\rp & = (\tr h)\Delta_\phi J_\phi^m + \delta_\phi\left(P_\phi^m(\nabla\tr h) - (\tr h)dJ_\phi^m\right), \\
  \left\lp P_\phi^m,\nabla^2\psi\right\rp - \frac{1}{m}\left\lp Y_\phi^m\nabla\phi,\nabla\psi\right\rp & = \psi\Delta_\phi J_\phi^m + \delta_\phi\left(P_\phi^m(\nabla\psi) - \psi\,dJ_\phi^m\right) .
 \end{align*}
 Finally, straightforward computations yield
 \begin{align*}
  \left\lp \delta_\phi dh, P_\phi^m\right\rp & = \left\lp \delta_\phi dP_\phi^m,h\right\rp - \delta_\phi\left(dh\cdot P_\phi^m - dP_\phi^m\cdot h\right) , \\
  J_\phi^m\left(\delta_\phi^2h-\Delta_\phi\tr h\right) & = \left\lp \nabla^2J_\phi^m - \Delta_\phi J_\phi^m\,g,h\right\rp \\
   & \quad + \delta_\phi\left(J_\phi^m\left(\delta_\phi h - d\tr h\right) + (\tr h)dJ_\phi^m - h(\nabla J_\phi^m)\right) .
 \end{align*}
 Combining these facts with Lemma~\ref{lem:Pdot} and Corollary~\ref{cor:Ydot} yields the desired result.
\end{proof}

An immediate consequence of~\eqref{eqn:TkM_geometric_psi} and Lemma~\ref{lem:P2dot} is the first variation of the total $N_{2,\phi}^m$-curvature functional.

\begin{cor}
 \label{cor:N2dot}
 Let $(M^n,g,v,m,\mu)$ be a smooth metric measure space.  Then
 \begin{multline*}
  D\left(\int_M N_{2,\phi}^m\,d\nu\right)[h,\psi] = -\int_M\left\lp\Psi_{2,\phi}^m,h\right\rp\,d\nu \\ + \int_M \left(\frac{2(m+n-2)}{m+n}\Delta_\phi J_\phi^m - \frac{m+n-4}{m+n} N_{2,\phi}^m\right)\psi\,d\nu .
 \end{multline*}
\end{cor}

An immediate consequence of Corollary~\ref{cor:R2dot} and Corollary~\ref{cor:N2dot} is the first variation of the total $\sigma_2$-curvature functional $\mF_2$.

\begin{cor}
 \label{cor:S2dot}
 Let $(M^n,g,v,m,\mu)$ be a smooth metric measure space.  Then
 \[ D\mF_2[h,\psi] = \frac{1}{2}\int_M \left\lp E_{2,\phi}^m, h\right\rp\,d\nu - \frac{m+n-4}{m+n}\int_M \sigma_{2,\phi}^m\psi\,d\nu, \]
 where
 \[ E_{2,\phi}^m := \tf_\phi\left( B_\phi^m + \frac{m+n-4}{m+n-2}T_{2,\phi}^m\right) . \]
\end{cor}

\subsection{The first variation of the $\mY$-functionals}
\label{subsec:full/3}

Corollary~\ref{cor:R1dot} and Corollary~\ref{cor:S2dot} give formulae for the linearizations of the $\mF_1$- and $\mF_2$-functionals.  In particular, combining these results with Proposition~\ref{prop:we_sigma2} immediately shows that quasi-Einstein manifolds are critical points of the volume-normalized $\mF_k$-functionals for $k\in\{1,2\}$.

\begin{prop}
 \label{prop:we_kappa0_crit}
 Let $(M^n,g,v,m,\mu)$ be a closed quasi-Einstein manifold.  Assume that $(g,v)\in\kM_1$.  Then $(g,v)$ is a critical point of $\mF_k\colon\kM_1\to\bR$ for $k\in\{1,2\}$.  Moreover, if $(g,v)$ is a critical point of $\mF_1\colon\kM_1\to\bR$, then it is a quasi-Einstein metric-measure structure.
\end{prop}

Corollary~\ref{cor:R1dot} and Corollary~\ref{cor:S2dot} also provide key ingredients for proving that weighted Einstein manifolds with $\mu=0$ and positive scale are critical points of the $\mY_k$-functionals for $k\in\{1,2\}$.  Indeed, combining these results with Corollary~\ref{cor:change_TkM} and Lemma~\ref{lem:csigma_to_sigma} allows us to compute the linearizations of the $\cmF_k$-functionals.  We begin by explaining the relevance of Corollary~\ref{cor:change_TkM}.  Given a closed weighted manifold $(M^n,m,\mu)$ and an integer $k\leq m$, denote by $\mF_0^{(m-k)},\mF_1^{(m-1)}\colon\kM\to\bR$ the functionals
\begin{align*}
 \mF_0^{(m-k)}(g,v) & := \int_M d\nu^{(m-k)} , \\
 \mF_1^{(m-1)}(g,v) & := \int_M \sigma_{1,\phi}^{m-1}\,d\nu^{(m-1)} ,
\end{align*}
where $d\nu^{(m-k)}$ and $\sigma_{1,\phi}^{m-1}$ are the weighted volume element on $\kM(M,m-k,\mu)$ and the weighted scalar curvature on $\kM(M,m-1,\mu)$, respectively.  The linearizations
\begin{align*}
 &D^{(m-k)}\mF_0^{(m-k)}\colon T\kM(M,m-k,\mu)\to\bR, \\
 &D^{(m-1)}\mF_1^{(m-1)}\colon T\kM(M,m-1,\mu) \to \bR
\end{align*}
are readily computed from~\eqref{eqn:TkM_geometric} and Corollary~\ref{cor:R1dot}, respectively.  Applying Corollary~\ref{cor:change_TkM} yields the following result.

\begin{prop}
 \label{prop:var_changem}
 Let $(M^n,g,v,m,\mu)$ be a closed smooth metric measure space and let $k\in\bR$.  Regard $\mF_0^{(m-k)}$ and $\mF_1^{(m-1)}$ as functionals on $\kM(M,m,\mu)$.  Then
 \begin{align*}
  D\mF_0^{(m-k)}[h,\psi] & = \frac{k}{2(m+n)}\int_M \left\lp v^{-k}g, h\right\rp\,d\nu^{(m)} - \frac{m+n-k}{m+n}\int_M \psi v^{-k}\,d\nu^{(m)} , \\
  D\mF_1^{(m-1)}[h,\psi] & = \frac{m+n-3}{2(m+n-2)}\int_M v^{-1}\left\lp T_{1,\phi}^{m-1} - \frac{m+n-2}{m+n}\sigma_{1,\phi}^{m-1}g,h\right\rp\,d\nu^{(m)} \\
   & \quad - \frac{m+n-3}{m+n}\int_M \psi v^{-1}\sigma_{1,\phi}^{m-1}\,d\nu^{(m)} .
 \end{align*}
\end{prop}

\begin{proof}
 \eqref{eqn:TkM_geometric} implies that
 \[ D^{(m-k)}\mF_0^{(m-k)}[h,\psi] = -\int_M \psi\,d\nu^{(m-k)} . \]
 Corollary~\ref{cor:change_TkM} then yields the formula for $D\mF_0^{(m-k)}$.  On the other hand, Corollary~\ref{cor:R1dot} implies that
 \begin{multline*}
  D^{(m-1)}\mF_1^{(m-1)}[h,\psi] = -\frac{m+n-3}{m+n-1}\int_M\ \sigma_{1,\phi}^{m-1}\psi \,d\nu^{(m-1)} \\ + \frac{m+n-3}{2(m+n-2)}\int_M \left\lp T_{1,\phi}^{m-1} - \frac{m+n-2}{m+n-1}\sigma_{1,\phi}^{m-1}g,h\right\rp\,d\nu^{(m-1)} .
 \end{multline*}
 Corollary~\ref{cor:change_TkM} then yields the formula for $D\mF_1^{(m-1)}$.
\end{proof}

Given a weighted manifold $(M^n,m,\mu)$ and a positive integer $j\leq m$, Proposition~\ref{prop:var_changem} motivates the definitions
\[ \tf_\phi T_{k-j,\phi}^{m-j} := T_{k-j,\phi}^{m-j} - \frac{m+n-k}{m+n}\sigma_{k-j,\phi}^{m-j}g \]
on $\kM(M,m,\mu)$.

By Lemma~\ref{lem:csigma_to_sigma} and Proposition~\ref{prop:var_changem}, one can compute the linearization of $\cmF_k$, $k\in\{0,1,2\}$, in terms of the linearizations of $\mF_{(k-j)}^{(m-j)}$, $0\leq j\leq k$.  We begin by considering the $\cmF_1$-functional, noting in particular that the critical points of the $\mY_1$-functional are exactly weighted Einstein manifolds.

\begin{thm}
 \label{thm:full_Y1}
 Fix $\kappa\in\bR_+$ and let $(M^n,g,v,m,\mu)$ be a smooth metric measure space.  For every $(h,\psi)\in T_{(g,v)}\kM$ it holds that
 \begin{multline}
  \label{eqn:Dcsigma1}
  D\cmF_1[h,\psi] = -\frac{m+n-2}{m+n}\int_M \left(\csigma_{1,\phi}^m + \frac{m}{m+n-2}\kappa v^{-1}\right)\psi\,d\nu \\ + \frac{m+n-2}{2(m+n-1)}\int_M \left\lp \cE_{1,\phi}^m + \frac{m}{m+n-2}\kappa v^{-1}\cU_{0,\phi}^m,h\right\rp\,d\nu ,
 \end{multline}
 where $U_{0,\phi}^m$ is as in~\eqref{eqn:cUk}.
 In particular, $(M^n,g,v,m,0)$ is a critical point of $\mY_1\colon\kM\to\bR$ if and only if it is a weighted Einstein manifold with scale $\kappa$.
\end{thm}

\begin{proof}
 Lemma~\ref{lem:csigma_to_sigma}, Corollary~\ref{cor:R1dot}, and Proposition~\ref{prop:var_changem} immediately yield~\eqref{eqn:Dcsigma1}.  Combining Proposition~\ref{prop:var_changem} and~\eqref{eqn:Dcsigma1}, we see that $D\mY_1\colon T_{(g,v)}\kM\to\bR$ vanishes if and only if
 \begin{align*}
  \cE_{1,\phi}^m + \frac{m}{m+n-2}\kappa v^{-1}\cU_{0,\phi}^m & = \frac{2m(m+n-1)}{(m+n)^2(m+n-2)(2m+n-2)}\left(\frac{\int\csigma_{1,\phi}^m}{\int v^{-1}}\right)v^{-1}g, \\
  \csigma_{1,\phi}^m + \frac{m}{m+n-2}\kappa v^{-1} & = \frac{2m(m+n-1)}{(m+n)(m+n-2)(2m+n-2)}\left(\frac{\int\csigma_{1,\phi}^m}{\int v^{-1}}\right)v^{-1} \\
   & \quad + \frac{\int\csigma_{1,\phi}^m}{\int 1} .
 \end{align*}
 From the definitions of $\cE_{1,\phi}^m$ and $\cU_{0,\phi}^m$, we conclude that these two conditions are equivalent to
 \begin{align}
  \label{eqn:mY1_crith} P_\phi^m & = \frac{1}{m+n}\left(\frac{\int\csigma_{1,\phi}^m\,d\nu}{\int d\nu}\right)g, \\
  \label{eqn:mY1_critpsi} \hsigma_{1,\phi}^m & = \frac{\int\csigma_{1,\phi}^m\,d\nu}{\int d\nu} .
 \end{align}
 where $\hsigma_{1,\phi}^m$ is as in~\eqref{eqn:hsigmak}.  If $(M^n,g,v,m,0)$ is a weighted Einstein manifold with scale $\kappa$, then Proposition~\ref{prop:we_integral_relation} implies that~\eqref{eqn:mY1_crith} and~\eqref{eqn:mY1_critpsi} hold.  Conversely, if~\eqref{eqn:mY1_crith} holds, then~\eqref{eqn:we_integral_nocsigma} implies that
 \[ \int_M \csigma_{1,\phi}^m\,d\nu = \frac{(m+n)(2m+n-2)}{2(m+n-1)}\kappa\int v^{-1}\,d\nu . \]
 Inserting this into~\eqref{eqn:mY1_critpsi} implies that $\csigma_{1,\phi}^m$ is constant, and hence $(M^n,g,v,m,0)$ is a weighted Einstein manifold with scale $\kappa$.
\end{proof}

We next consider the $\cmF_2$- and $\mY_2$-functionals.  Here the relationship to weighted Einstein manifolds is more subtle.  First, the Euler equation for the $\cmF_2$-functional is fourth-order in the metric, so one cannot expect to characterize the critical points of the $\mY_2$-functional as weighted Einstein manifolds.  Second, the fact that weighted Einstein manifolds with $\mu=0$ and scale $\kappa$ are critical points for the $\mY_2$-functional depends on the subtle cancellation~\eqref{eqn:we_bach} for the weighted Bach tensor of such manifolds.  This latter point is discussed further in Remark~\ref{rk:Y2_unique}.

\begin{thm}
 \label{thm:full_Y2}
 Fix $\kappa\in\bR_+$ and let $(M^n,g,v,m,\mu)$ be a smooth metric measure space.  For every $(h,\psi)\in T_{(g,v)}\kM$ it holds that
 \begin{multline}
  \label{eqn:Dcsigma2}
  D\cmF_2[h,\psi] = -\frac{m+n-4}{m+n}\int_M \left(\csigma_{2,\phi}^m + \frac{m}{m+n-4}\kappa v^{-1}\cs_{1,\phi}^m\right)\psi\,d\nu \\ + \frac{m+n-4}{2(m+n-2)}\int_M\left\lp \cE_{2,\phi}^m + \frac{m+n-2}{m+n-4}\cB_{\phi}^m + \frac{m}{m+n-4}\kappa v^{-1}\cU_{1,\phi}^m, h\right\rp\,d\nu ,
 \end{multline}
 where $\cU_{1,\phi}^m$ is as in~\eqref{eqn:cUk} and $\cB_\phi^m$ is the Bach tensor with scale $\kappa$,
 \[ \cB_\phi^m := \delta_\phi dP_\phi^m - \frac{1}{m}(\tr dP_\phi^m)\otimes d\phi + A_\phi^m\cdot\left(P_\phi^m - \cZ_\phi^mg\right) . \]
 In particular, if $(M^n,g,v,m,0)$ is a weighted Einstein manifold with scale $\kappa$, then it is a critical point of $\mY_2\colon\kM\to\bR$.
\end{thm}

\begin{proof}
 Lemma~\ref{lem:csigma_to_sigma}, Corollary~\ref{cor:S2dot}, and Proposition~\ref{prop:var_changem} immediately yield~\eqref{eqn:Dcsigma2}.  Combining Proposition~\ref{prop:var_changem} and~\eqref{eqn:Dcsigma2} implies that $D\mY_2\colon T_{(g,v)}\kM\to\bR$ vanishes if and only if
 \begin{multline*}
  \cE_{2,\phi}^m + \frac{m+n-2}{m+n-4}\cB_\phi^m + \frac{m}{m+n-4}\kappa v^{-1}\cU_{1,\phi}^m \\ = \frac{4m(m+n-2)}{(m+n)^2(m+n-4)(2m+n-2)}\left(\frac{\int\csigma_{2,\phi}^m}{\int v^{-1}}\right)v^{-1} g
 \end{multline*}
 and
 \begin{multline*}
  \csigma_{2,\phi}^m + \frac{m}{m+n-4}\kappa v^{-1}\cs_{1,\phi}^m = \frac{\int\csigma_{2,\phi}^m}{\int 1} \\ + \frac{4m(m+n-1)}{(m+n)(m+n-4)(2m+n-2)}\left(\frac{\int\csigma_{2,\phi}^m}{\int v^{-1}}\right)v^{-1} .
 \end{multline*}
 It is clear that if $(M^n,g,v,m,0)$ is a weighted Einstein manifold with scale $\kappa$, then $\cB_\phi^m=0$.  It then follows from Proposition~\ref{prop:we_sigma2} and Proposition~\ref{prop:we_integral_relation} that $D\mY_2\colon T_{(g,v)}\kM\to\bR$ vanishes for such a manifold.
\end{proof}

\begin{remark}
 \label{rk:Y2_unique}
 A key point in the proof of Theorem~\ref{thm:full_Y2} is that, due to the identity
 \[ \cB_\phi^m = B_\phi^m - m\kappa v^{-1}\left(P_\phi^{m-1} - \frac{m+n-3}{m+n-2}P_\phi^m\right), \]
 the summands $B_\phi^m$ and $P_\phi^{m-1}$ in the formulae for $D\mF_2$ and $D\mF_1^{(m-1)}$, respectively, exactly combine to yield the weighted Bach tensor with scale $\kappa$.  In particular, since $P_\phi^{m-1}$ need not be constant for a weighted Einstein manifold $(M^n,g,v,m,0)$, if we add suitable multiples of $\mF_1^{(m-1)}$ and $\mF_0^{(m-2)}$ to $\mF_2$ to obtain a functional with Euler equation given by $\csigma_{2,\phi}^m$, the resulting functional need not admit weighted Einstein manifolds among its critical points.  This is the final justification for our focus on the $\cmF_2$- and $\mY_2$-functionals.
\end{remark}

\begin{remark}
 \label{rk:Y2_scaling}
 It follows from Lemma~\ref{lem:mY_scaling} that
 \[ D\mY_2[g,0] = -\frac{2m+n}{2}D\mY_2[0,1] . \]
 Using~\eqref{eqn:Dcsigma2}, we conclude that if $(M^n,g,v,m,0)$ is a closed smooth metric measure space such that $(g,v)$ is locally conformally flat in the weighted sense and a critical point of $\mY_2\colon\kC\to\bR$, then
 \[ \int_M \tr \hE_{2,\phi}^m = 0 . \]
 This observation further simplifies~\eqref{eqn:obata/we/div} in the case $k=2$.  It also suggests that if $(M^n,g,v,m,0)$ is a closed smooth metric measure space such that $(g,v)$ is locally conformally flat in the weighted sense and a critical point of $\mY_k\colon\kC\to\bR$, then $\int\tr\hE_{k,\phi}^m=0$.
\end{remark}